\documentclass[12pt]{amsart}

\theoremstyle{definition}
\theoremstyle{plain}

\usepackage[left=2.8cm,top=2.8cm,right=2.8cm,bottom = 2.4cm]{geometry}
\usepackage{amsfonts}  
\usepackage{amsmath}
\usepackage{amscd}
\usepackage{amssymb} 
\usepackage{amsthm} 
\usepackage{bbm}
\usepackage{bm}
\usepackage{enumerate}
\usepackage{enumitem}
\usepackage{hyperref}
\usepackage{latexsym}
\usepackage{mathabx}
\usepackage{mathdots}
\usepackage{mathtools}
\usepackage{xfrac}

\allowdisplaybreaks

\DeclareMathOperator{\Res}{Res}
\DeclareMathOperator{\Hom}{Hom}

\newcommand{\Ahat}{\widehat{A}}
\newcommand{\Atilde}{\widetilde{A}}
\newcommand{\bbar}{\overline{b} }

\newcommand{\bfm}{\bm{m}}
\newcommand{\bfy}{\bm{y}}

\newcommand{\C}{\mathbb{C}}
\newcommand{\calA}{\mathcal{A}}
\newcommand{\calC}{\mathcal{C}}

\newcommand{\calL}{\mathcal{L}}
\newcommand{\calQ}{\mathcal{Q}}
\newcommand{\calY}{\mathcal{Y}}

\newcommand{\cbar}{\overline{c} }
\newcommand{\Cg}{\mathfrak{C}_{g}}

\newcommand{\Chat}{\widehat{\C}}

\newcommand{\D}{\mathcal{D}}
\newcommand{\del}{\partial}
\newcommand{\delx}{\nabla(x)}

\newcommand{\delMg}{\nabla_{\hspace{-1 mm}\Mg}}

\newcommand{\dx}{\partial_{x}}

\newcommand{\dy}{\partial_{y}}

\newcommand{\E}{e}
\newcommand{\End}{\textup{End}}
\newcommand{\F}{\mathcal{F}}
\newcommand{\Fg}{\F} 

\newcommand{\FgMa}{\Fg_{\bm{M_{\alpha}}}}
\newcommand{\Fzero}{\F^{(0)}}
\newcommand{\half}{\frac{1}{2}}

\newcommand{\I}{\mathcal{I}}
\newcommand{\Id}{\textup{Id}}
\newcommand{\im}{\textup{i}}
\newcommand{\Ip}{\I_{+}}

\newcommand{\Ltilde}{\widetilde{L}}

\newcommand{\M}{\mathcal{M}}
\newcommand{\Mg}{\M_{g}}

\newcommand{\N}{\mathbb{N}}
\newcommand{\nabmy}[2]{\nabla^{(#1)}_{#2}}
\newcommand{\nua}{\nu_{\bm{\alpha}}}
\newcommand{\nuag}{\nu_{\bm{\alpha,\gamma}}}

\newcommand{\Omo}{\Omega_{0}}
\newcommand{\Pcal}{\mathcal {P}}
\newcommand{\Poly}{\mathfrak {P}}
\newcommand{\Psihat}{\widehat{\Psi}}

\newcommand{\Schg}{\mathfrak{S}_{g}}
\newcommand{\shalf}{\sfrac{1}{2}}
\newcommand{\Sg}{\mathcal{S}^{(g)}}

\newcommand{\Sone}{\mathcal{S}^{(1)}}
\newcommand{\Szero}{\mathcal{S}^{(0)}}
\newcommand{\SL}{\textup{SL}}
\newcommand{\slLie}{\mathfrak{sl}}

\newcommand{\Sym}{\textup{Sym}}
\newcommand{\tpi}{2\pi \im} 
\newcommand{\Tr}{\textup{tr}}
\newcommand{\vac}{\mathbbm{1}}
\newcommand{\Wmod}{\mathcal W}

\newcommand{\wt}{\textup{wt}}

\newcommand{\Z}{\mathbb{Z}}

\newcommand{\Zg}{Z} 
\newcommand{\Zhu}{\textup{Zhu}}
\newcommand{\Zzero}{Z^{(0)}}

\newtheorem{corollary}{Corollary}[section]
\newtheorem{example}{Example}[section]
\newtheorem{lemma}{Lemma}[section]

\newtheorem{proposition}{Proposition}[section]
\newtheorem{remark}{Remark}[section]
\newtheorem{theorem}{Theorem}[section]

\title{Genus $g$ Zhu Recursion for Vertex Operator Algebras and Their Modules
}
\date{}
\begin{document}
	\author{Michael P. Tuite}
	\address{School of Mathematical and Statistical Sciences\\ 
	University of Galway, Galway H91 TK33, Ireland}
	\email{michael.tuite@universityofgalway.ie}
	\author[Michael Welby]{Michael Welby$^{\dagger}$}
	\thanks{$^{\dagger}$Supported by an Irish Research Council Government of Ireland Postgraduate Scholarship.}
\email{michael.welby@universityofgalway.ie}

\begin{abstract}
We describe Zhu recursion for a vertex operator algebra (VOA) and its modules on a 
genus $g$ Riemann surface in the Schottky uniformisation. 
We show that $n$-point (intertwiner) correlation functions can be written as linear combinations of $(n-1)$-point functions with universal coefficients given by derivatives of the differential of the third kind, the Bers quasiform and certain holomorphic forms. 
We use this formula to describe conformal Ward identities framed in terms of a canonical differential operator which acts with respect to the Schottky moduli and to the insertion points of the $n$-point function. We consider the generalised Heisenberg VOA and determine all its correlation functions by Zhu recursion. We also use Zhu recursion to derive linear partial differential
equations for the Heisenberg VOA partition function and various structures such as the bidifferential of the second kind, holomorphic $1$-forms, the prime form and the period matrix. Finally, we compute the genus $g$ partition function for any rational Euclidean lattice generalised  VOA.
\end{abstract}
\maketitle
\section{Introduction}
Conformal field theories on Riemann surfaces have been studied by physicists for over 35 years e.g. \cite{AGMV, EO,Ma,V}. The purpose of this paper is to make some of these ideas explicitly realisable within vertex operator algebra (VOA) theory e.g.~\cite{K,LL,MT1}. 
There has been significant progress recently to establish convergence properties of VOA conformal blocks using algebraic geometry methods \cite{FBZ}. It has now been established by Gui in \cite{G} that sewn conformal blocks are absolutely and locally uniformly convergent for simple self-dual $C_{2}$-cofinite VOAs of CFT type \cite{Z2, C, DGT1,DGT2,DGT3}.
This paper describes genus $g$ Zhu recursion in the Schottky sewing scheme for similar VOAs but without assuming $C_{2}$-cofiniteness.  Zhu recursion then provides for the explicit computation of formal correlation functions.  Our approach  follows earlier results for Zhu recursion at genus two \cite{MT2,MT3,GT} and higher \cite{TZ1,T,TW2}.
Here we also make substantial use of the Bers quasiform \cite{Be1,Be2} using recent results in the companion paper \cite{TW1}. The present paper  contains many significant improvements on the results of~\cite{TW2} in relation to the geometric meaning of Zhu recursion, along with a treatment of modules and intertwiners.

In his seminal paper~\cite{Z1}, Zhu describes a recursion/reduction formula for genus one (and zero) $n$-point correlation functions for VOAs and their modules in terms of a finite linear combination of $(n-1)$-point functions with universal coefficients given by Weierstrass elliptic functions. 
Zhu recursion provides a powerful calculational tool which, for example, was used by Zhu to prove convergence and modularity of the simple module partition functions for rational $C_{2}$-cofinite VOAs~\cite{Z1}. Zhu recursion was  extended to correlation functions on a genus two surface formed by sewing two tori~\cite{GT}. 
In this work we describe Zhu recursion 
on a genus $g$ Riemann surface in the Schottky uniformisation where $n$-point correlation functions for a VOA (of CFT type) and its modules are expressed in terms of a finite linear combination of $(n-1)$-point functions with universal coefficients given by derivatives of either the Bers quasiform \cite{Be1,Be2} or the differential of the third kind and holomorphic forms. Furthermore, general genus Zhu recursion is shown to be a formal version of a residue expansion for meromorphic forms described in \cite{TW1}.
We describe conformal Ward identities for $n$-point correlation functions for VOAs and their modules expressed in terms of a canonical differential operator.
Applying these conformal Ward identities to the generalised Heisenberg VOA, we obtain differential equations in terms of the canonical derivative of the Heisenberg  partition function but also for classical objects such as the bidifferential of the second kind, the projective connection, holomorphic 1-forms and the prime form, confirming results in \cite{TW1}. 

Section~2 begins with a brief review of the classical Schottky uniformisation of a genus $g$ marked Riemann surface $\Sg$~\cite{Fo,FK, Bo}. 
We define standard objects such as the bidifferential of the second kind, the projective connection, holomorphic 1-forms, the prime form and the differential of the third kind.
We review properties of the Bers quasiform $\Psi_{N}(x,y)$~\cite{Be1,Be2} for $N\ge 2$ and $g\ge 2$, a meromorphic $N$-differential in $x$ and a $(1-N)$-quasidifferential in $y$ with a unique simple pole at $x=y$ and a related spanning set of holomorphic $N$-forms $\{\Theta_{N,a}^{\ell}(x)\}$ discussed in~\cite{TW1}. We further define $\Psi_{1}$ to be a differential of the third kind which shares many properties in common with the Bers quasiform.
We review the expansion of a meromorphic form on $\Sg$ in terms of $\Psi_{N}$ and $\Theta_{N,a}^{\ell}$ which parallels 
genus $g$ Zhu recursion formulas. 
We also review  some sewing formulas for $\Psi_{N}$ and $\Theta_{N,a}^{\ell}$ that are applied in deriving Zhu recursion in Sections~4 and 5.

Section~3 begins with some standard VOA theory e.g.~\cite{K,LL,MT1}.
We consider a VOA $V$ which is simple, self-dual and of strong CFT-type which therefore admits a unique non-degenerate symmetric invariant bilinear form that we call the Li-Z metric \cite{FHL,L}.
We introduce the rational genus zero $n$-point (correlation) function and determine its M\"obius invariance properties. 
We describe a general Ward identity for $n$-point functions associated with any weight $N$ quasiprimary vector. 
We then develop genus zero Zhu recursion which features prominently in the later genus $g$ analysis. We find the most general form of genus zero Zhu recursion~\cite{Z1} for an $n$-point function expressed as a sum of $(n-1)$-point functions where we reduce with respect to a quasiprimary vector of weight $N$.  The expansion coefficients depend on $N$ but not on $V$. Although the coefficients are not unique, we obtain the same Zhu expansion due to the Ward identity. Finally, we extend this result to any quasiprimary descendant.

Section~4 is concerned with genus $g$ Zhu recursion for $V$. 
We define the genus $g$ partition and $n$-point functions 
in terms of infinite sums of certain genus zero $(n+2g)$-point correlation functions for $n$ given $V$ vectors, $g$ independent $V$-basis vectors and their Li-Z duals where we sum over all basis vectors. We prove that this definition agrees with Zhu's trace functions at genus one~\cite{Z1}. We describe the formal M\"obius-invariance properties of genus $g$ correlation functions. Theorem~\ref{theor:ZhuGenusg} describes the genus $g$ Zhu recursion formula for $V$ where $n$-point functions are expanded as linear combinations of $(n-1)$-point functions. The proof relies partly on genus zero Zhu recursion but also on another recursion property following from the invariance of the Li-Z metric. 
We find universal coefficients in our Zhu recursion formula expressed in terms of derivatives of $\Psi_{N}$ and $\Theta_{N,a}^{\ell}$ for $N\ge 1$ (identified by the sewing formulas of Section~2), where $N$ is the conformal weight of the reduced quasiprimary vector. Similarly to the genus zero case, we can extend this formula to quasiprimary descendants.

In Section~5 we define genus zero and genus $g$ partition and correlation functions for intertwiner vertex operators associated with irreducible $V$-modules \cite{FHL}. We describe genus zero Zhu recursion for general intertwining $n$-point functions where we reduce with respect to an element of $V$. We obtain M\"obius invariance for genus zero correlation functions containing at most two intertwiner operators which allows us to reproduce Zhu's module trace correlation function at genus one. We assume certain conditions on the dimension of the space of intertwiners (that hold for VOAs satisfying the Verlinde formula \cite{H1} or for the Heisenberg abelian intertwiner algebra). We define the genus $g$ partition and $n$-point functions with intertwiner bases and general insertions. Finally, in the main result of the paper, Theorem~\ref{theor:ZhuGenusgInt}, we derive a genus $g$ Zhu recursion formula for $V$, its modules and intertwiners, where the reducing state is a vector in $V$ and the remaining insertions are arbitrary.

Section~6 discusses conformal Ward identities for genus $g$ intertwiner $n$-point correlation functions derived from Zhu recursion.
We find that the Virasoro $1$-point function is given by the action of a canonical differential operator $\nabla(x)$ on the partition function
describing its variation with respect to the Schottky surface moduli.
The general Ward identities for $n$-point functions involve a canonical differential operator $\nabla_{\bm{y}}^{(\bm{m})}(x)$  (constructed from $\Psi_{2}$) describing variations with respect to punctured Riemann surface moduli given by Schottky parameters and $n$ insertion points $\bm{y}=(y_1,\ldots,y_n)$. 
The operator, which  maps meromorphic  $(\bm{m})=(m_1,\ldots,m_n)$-forms in $\bm{y}$ to meromorphic $(2,\bm{m})$-forms in $(x,\bm{y})$  is discussed in more depth in~\cite{TW1}. 
We also derive a commutative property of this operator discussed in \cite{TW1} and a recursive Ward identity for Virasoro $n$-point generating functions.

Section~7 discusses applications of genus $g$ Zhu recursion for the rank one generalised Heisenberg VOA~\cite{TZ2}. We use a formula for the generalised partition function, derived via combinatorial means~\cite{T}, and Zhu recursion to determine an explicit expression for all $n$-point functions for the Heisenberg VOA and all its modules. We derive  differential equations describing the action of $\nabla(x)$ on the partition function $Z_M$ and that of $\nabla_{\bm{y}}^{(\bm{m})}(x)$ on several standard objects on $\Sg$, such as the bidifferential of the second kind, the projective connection, the prime form and the period matrix~\cite{Fa,Mu,Bo}. 
We finish with the example of the genus $g$ partition function for the generalised VOA for a rational Euclidean lattice expressed in terms of the lattice Siegel theta function and the Heisenberg VOA partition function.  

\noindent \textbf{Acknowledgements.} 
	We thank Tom Gilroy, Geoff Mason, Hiroshi Yamauchi and Giulio Codogni for a number of very helpful comments and suggestions.

\section{Differential Structures on Riemann Surfaces}
\label{sec:Genus g}
\subsection{Notational conventions}
Define the following indexing sets for integers $g,N\ge 1$ 
\begin{align}
	\label{eq:ILconv}
	\I:=\{-1,\ldots,- g,1,\ldots,g\},\quad
	\Ip:=\{1,\ldots,g\},\quad 
	\calL:=\{0,1,\ldots,2N-2\}.
\end{align}
For functions $f(x)$ and $g(x,y)$ and integers $i,j\ge 0$ we define
\begin{align}
	\notag
	&f^{(i)}(x):=\del^{(i)} f(x) :=\del_{x}^{(i)} f(x):= \frac{1}{i!}\frac{\del^{i}}{\del x^{i}}f(x),
	\\
	\label{eq:delconv}
	&g^{(i,j)}(x,y):=\del_{x}^{(i)} \del_{y}^{(j)}  g(x,y).
\end{align} 
We let $\Poly_{n}$ denote the space of polynomials with complex coefficients of degree at most $n$.

\subsection{The Schottky uniformisation of a Riemann surface}
\label{subsec:Schottky}
Consider a compact marked Riemann surface $\Sg$ of genus $g$
with canonical homology basis $\alpha_{a}$, $\beta_{b}$ for $a,b\in\Ip$ e.g.~\cite{FK,Mu,Fa,Bo}. 
We review the  construction of a genus $g$ Riemann surface $\Sg$ using the Schottky uniformisation where we sew $g$ handles to the Riemann sphere $\Szero\cong\Chat:=\C\cup \{\infty\}$ e.g.~\cite{Fo, Bo}. Every Riemann surface can be (non-uniquely) Schottky uniformised~\cite{Be2}.

For  $a\in\I$  let $\calC_{a}\subset \Szero$ be $2g$ non-intersecting Jordan curves where $z\in \calC_{a}$, $z'\in \calC_{-a}$ for $a\in\Ip$ are identified by a sewing relation
\begin{align}\label{eq:SchottkySewing}
	\frac{z'-W_{-a}}{z'-W_{a}}\cdot\frac{z-W_{a}}{z-W_{-a}}=q_{a},\quad a\in\Ip,
\end{align}
for $q_{a}$ with $0<|q_{a}|<1$ and $W_{\pm a}\in\Chat$. 
Thus  $z'=\gamma_{a}z$  for $a\in\Ip$ with 
\begin{align}\label{eq:gammaa}
	\gamma_{a}:=\sigma_{a}^{-1}
	\begin{pmatrix}
		q_{a}^{1/2} &0\\
		0 &q_{a}^{-1/2}
	\end{pmatrix}
	\sigma_{a},\quad 
	\sigma_{a}:=(W_{-a}-W_{a})^{-1/2}\begin{pmatrix}
		1 & -W_{-a}\\
		1 & -W_{a}
	\end{pmatrix}.
\end{align}
$\sigma_{a}(W_{-a})=0$ and $\sigma_{a}(W_{a})=\infty$ are, respectively, attractive and repelling fixed points of $Z\rightarrow Z'=q_{a}Z$ for  $Z=\sigma_{a} z$ and $Z'=\sigma_{a} z'$.
$W_{-a}$ and  $W_{a}$  are the corresponding fixed points for $\gamma_{a}$. We identify the standard homology cycles $\alpha_{a}$  with $\calC_{-a}$ and  $\beta_{a}$ with a path connecting  $z\in  \calC_{a}$ to $z'=\gamma_{a}z\in  \calC_{-a}$.

The genus $g$ Schottky group $\Gamma$  is   the free group with generators $\gamma_{a}$ for $a \in \Ip$.
Define $\gamma_{-a}:=\gamma_{a}^{-1}$. The independent elements of $\Gamma$ are reduced words of length $k$
of the form $\gamma=\gamma_{a_{1}}\ldots \gamma_{a_{k}}$ where $a_{i}\neq -a_{i+1}$ for each $i=1,\ldots,k-1$. 

Let $\Lambda(\Gamma)$ denote the  limit set\footnote{Note that for $g=1$, $\Lambda(\Gamma)$ is the empty set.}  of $\Gamma$ i.e. the set of limit points of the action of $\Gamma$ on $\Chat$. Then $\Sg\simeq\Omo/\Gamma$ where $\Omo:=\Chat-\Lambda(\Gamma)$.
We let $\D\subset\Chat$  denote the standard connected fundamental region with oriented boundary curves $\calC_{a}$. 

Define $w_{a}:=\gamma_{-a}.\infty$. Using~\eqref{eq:SchottkySewing} we find 
\begin{align}
	\label{eq:wa}
	w_{a}=\frac{W_{a}-q_{a}W_{-a}}{1-q_{a}},\quad a\in\I,
\end{align}
where we define $q_{-a}:=q_{a}$.
Then~\eqref{eq:SchottkySewing} is equivalent to
\begin{align}\label{eq:SchottkySewing2}
	(z'-w_{-a})(z-w_{a})=\rho_{a},\quad a\in\Ip,
\end{align}
with 
\begin{align}
	\label{eq:rhoa}
	\rho_{a}:=&-\frac{q_{a}(W_{-a}-W_{a})^{2}}{(1-q_{a})^{2}}
	=-\frac{q_{a}(w_{-a}-w_{a})^{2}}{(1+q_{a})^{2}}.
\end{align}
Equation~\eqref{eq:SchottkySewing2} implies
\begin{align*}
	\gamma_{a}z=w_{-a}+\frac{\rho_{a}}{z-w_{a}}.
\end{align*}
It is convenient (but not necessary) to choose the Jordan curve $\calC_{a}$ to be the  boundary of the disc $\Delta_{a}$ with centre $w_{a}$ and radius $|\rho_{a}|^{\frac{1}{2}}$. Then 
$\gamma_{a}$ maps the exterior (interior) of $\Delta_{a}$  to the interior (exterior) of $\Delta_{-a}$ since
\begin{align*}
	|\gamma_{a}z-w_{-a}||z-w_{a}|=|\rho_{a}|.
\end{align*}
Furthermore,  the discs $\Delta_{a},\Delta_{b}$ are non-intersecting if and only if
\begin{align}
	\label{eq:JordanIneq}
	|w_{a}-w_{b}|>|\rho_{a}|^{\frac{1}{2}}+|\rho_{b}|^{\frac{1}{2}}
	\mbox{ for all } a\neq b.
\end{align} 
We define $\Cg$ to be the set 
$\{ (w_{a},w_{-a},\rho_{a})|a\in\Ip\}\subset \C^{3g}$ satisfying~\eqref{eq:JordanIneq}. We refer to $\Cg$ as the Schottky parameter space.

The relation~\eqref{eq:SchottkySewing} is M\"obius  invariant for $\gamma =\left(\begin{smallmatrix}A&B\\C&D\end{smallmatrix}\right)\in\SL_{2}(\C)$ with $(z,z',W_{a},q_{a})\rightarrow(\gamma z,\gamma z',\gamma W_{a},q_{a})$ giving
an $\SL_{2}(\C)$ action on $\Cg$ as follows
\begin{align}
	\label{eq:Mobwrhoa}
	\gamma:(w_{a},\rho_{a})\mapsto & 
	\left(	\frac { \left( Aw_{a}+B \right)  \left( Cw_{-a}+D \right) -\rho_{a}
		\,AC}{ \left( Cw_{a}+D \right)  \left( Cw_{-a}+D \right) -\rho_{a}\,{
			C}^{2}},
	{\frac {\rho_{a}}{ \left(  \left( Cw_{a}+D \right)  \left( Cw_{-a}+D
			\right) -\rho_{a}\,{C}^{2} \right) ^{2}}}\right).
\end{align}
We define Schottky space as $\mathfrak{S}_{g}:=\Cg/\SL_{2}(\C)$ which provides a natural covering space for the moduli space of genus $g$ Riemann surfaces (of dimension 1 for $g=1$ and $3g-3$ for $g\ge 2$).  
We exploit the $\Cg$  parameterization throughout  because the sewing relation~\eqref{eq:SchottkySewing2} is more  readily implemented in the theory of vertex operators.

\subsection{Some differential forms on a Riemann surface}
Let $\Sg$  be a marked compact genus $g$ Riemann surface with canonical homology basis $\alpha_{a}, \beta_{a}$ for $a\in \Ip$. We define some standard objects on $\Sg$ e.g. \cite{Mu,Fa, Bo}.
The meromorphic bidifferential form of the second kind is a unique symmetric form  
\begin{align*}
	\omega(x,y)=\left(\frac{1}{(x-y)^{2}}+\text{regular terms}\right)\, dxdy,
\end{align*}%
for local coordinates $x,y$ where
$
\oint_{\alpha_{a}}\omega(x,\cdot )=0$ for all $a\in \Ip$.
It follows that
\begin{align}
	\nu_{a}(x)=\oint\limits_{\beta_{a}}\omega(x,\cdot ),  \quad a\in \Ip,
	\label{eq:nu}
\end{align}%
is the holomorphic differential of the first kind, a  1-form normalised by
$\oint\limits_{\alpha_{a}}\nu_{b}=\tpi\, \delta_{ab}$. 
The meromorphic differential of the third kind for $y,z\in \Sg$ is defined by \begin{align}
	\omega_{y-z}(x):=\int^{y}_{z}\omega(x,\cdot )=\left(\frac{1}{x-y}-\frac{1}{x-z}+\text{regular terms}\right)\, dx,
	\label{eq:omp2p1}
\end{align}
with $
\oint_{\alpha_{a}}\omega_{y-z}(\cdot )=0$ for all $a\in \Ip$.
The  period matrix $\Omega$ is defined by 
\begin{align}
	\Omega_{ab}=\frac{1}{\tpi}\oint\limits_{\beta_{a}}\nu_{b},\quad \quad a,b\in \Ip.
	\label{eq:period}
\end{align}%
$\omega(x,y)$ can  be
expressed in terms of the prime form $E(x,y)=
K(x,y)dx^{-1/2}dy^{-1/2}$, a holomorphic form of weight
$(-\frac{1}{2},-\frac{1}{2})$ (with a multiplier system) where 
\begin{align}
	\omega (x,y) = &\dx\dy\log K(x,y)\,dxdy,
	\label{eq:prime} 
\end{align}%
where $K(x,y)=(x-y)+O\left((x-y)^{3}\right)$ and  $K(x,y)=-K(y,x)$.
Furthermore the projective connection $s(x)$ is defined by
\begin{align}\label{eq:projcon}
	s(x)=\lim_{y\rightarrow x}\left(\omega(x,y)-\frac{dxdy}{(x-y)^{2}}\right).
\end{align} 
In the Schottky parametrisation for $A_{0}\in \Omega_{0}$  we have Poincar\'{e} sum expressions \cite{Bu}
\begin{align}
	\label{eq:Psi1def}
	\omega_{y-A_{0}}(x)&=\Psi_{1}(x,y):=
	\sum_{\gamma\in\Gamma} \Pi_{1}(\gamma x,y),
	\\
	\label{eq:omPoincare2}
	\omega(x,y)&=\Psi_{1}^{(0,1)}(x,y)dy=\sum_{\gamma\in\Gamma}\Pi_{1}^{(0,1)}(\gamma x,y)dy,	
\end{align}
where
\begin{align}\label{eq:Pi1def}
	\Pi_{1}(x,y):=\left(\frac{1}{x-y}-\frac{1}{x-A_{0}}\right)dx.
\end{align} 
$\Psi_{1}(x,y)$ is a 1-form in $x$ with simple poles at $x=y$ and $x=A_{0}$ but is a weight zero  quasiform in $y$ where for each Schottky group generator $\gamma_{a}$ 
\begin{align}\label{eq:Psi1diff}
	\Psi_{1}(x,y)-\Psi_{1}(x,\gamma_{a} y)=\nu_{a}(x),\quad a\in\Ip.
\end{align}
There are also Poincar\'e sum formulas for $\nu_{a}(x)$, $\Omega_{ab}$, $ E(x,y)$ and $s(x)$ e.g. \cite{Bo}.

\subsection{The Bers quasiform}
We consider the Bers quasiform $\Psi_{N}(x,y)$ which is defined for $g\ge 2$ and $N\ge 2$ and shares some  properties in common with $\Psi_{1}$ of~\eqref{eq:Psi1def}. $\Psi_{N}(x,y)$ was introduced by Bers \cite{Be1,Be2} in order to construct the Bers potential for each holomorphic $N$-form.
It also appears in the description of the  Laplacian determinant line bundle associated with $N$-forms \cite{McIT}. 
We review how a meromorphic $N$-form has a canonical expansion in terms of derivatives of $\Psi_{N}$ and an associated spanning set of holomorphic $N$-forms for $N\ge 2$ \cite{TW1}. In fact this result also holds for $N=1$ for  $\Psi_{1}$ and 1-form basis $\{\nu_{a}\}$. 
Later in \S\ref{sec:gZhu}  and \S\ref{sec:Genus g Inter} we will show that  $\Psi_{N}(x,y)$ and the associated  $N$-form spanning set play a similar role in genus $g$ Zhu recursion for VOAs and their modules.
It also shown in \cite{TW1} that $\Psi_{2}(x,y)$ enters into natural differential operators acting on meromorphic forms in multiple variables that appear in the conformal Ward identities described in \S\ref{sec:Wardids}.

The Bers quasiform of weight $(N,1-N)$ for $g\ge 2$ and $N\ge2$ is defined by the Poincar\'e series \cite{Be1,Be2,McIT,TW1} 
\begin{align}
	\label{eq:PsiNdef}
	\Psi_{N}(x,y):&=
	\sum_{\gamma\in\Gamma} \Pi_{N}(\gamma x,y),\quad x,y\in \Omo,
	\\
	\label{eq:PipiN}
	\Pi_{N}(x,y)&:=\Pi_{N}(x,y;\bm{A}):=\frac{1}{x-y}\prod_{\ell\in\calL}\frac{y-A_{\ell}}{x-A_{\ell}}dx^{N}dy^{1-N},
\end{align}
for index set $\calL$ of~\eqref{eq:ILconv} and
where $\bm{A}:=A_{0},\ldots,A_{2N-2}\in\Lambda(\Gamma)$ are distinct limit points of $\Gamma$.
$\Pi_{N}$ is M\"obius invariant with 
\begin{align}\label{eq:PiN_Mobius}
	\Pi_{N}(\gamma x, \gamma y;\bm{\gamma A})=\Pi_{N}(x,y;\bm{A}),
\end{align} 
for all $\gamma\in\SL_{2}(\C)$. 
$\Psi_{N}(x,y)$ is meromorphic  for $x,y\in \Omo$ with  simple poles  of residue one at $y=\gamma x$ for all $\gamma\in\Gamma$.
$\Psi_{N}$ was introduced by Bers  to construct a Bers potential $F(y)=f(y)dy^{1-N}$, with $y\in\Omega_{0}$, for any holomorphic $N$-differential $\Phi $ given by \cite{Be1,Be2}
\begin{align*}
	F(y):=-\langle \Psi_{N}(\cdot,y), \Phi \rangle_{P},
\end{align*}
for $N$-differential Petersson product $\langle \cdot,\cdot \rangle_{{P}}$ \cite{TW1}. 
$\Psi_{N}(x,y)$ is a bidifferential $(N,1-N)$-quasiform as follows.
It is an $N$-differential in $x$ since
\begin{align*}
	\Psi_{N}(\gamma x,y) &= \Psi_{N}(x,y),\quad \gamma\in\Gamma,
\end{align*} 
by construction and is a quasiperiodic $(1-N)$-form  in $y$ with  
\begin{align*}
	\Psi_{N}( x,\gamma y) - \Psi_{N}(x,y)&=\chi[\gamma](x,y),\quad \gamma\in\Gamma,
\end{align*}
where $\chi[\gamma](x,y)$ is a holomorphic  $N$-form in $x$ and an Eichler 1-cocycle in $y$ \cite{Be1},\cite{TW1}. In particular, for a Schottky group generator $\gamma_{a}$, $a\in\Ip$, we find 
\begin{align}
	\label{eq:chiTheta}
	\chi[\gamma_{a}](x,y)=-\sum_{\ell\in\calL}\Theta_{N,a}^{\ell}(x)y_{a}^{\ell}dy^{1-N},
\end{align}
with $y_{a}:=y-w_{a}$ and where 
$	\{\Theta_{N,a}^{\ell}(x)\}_{a\in\Ip}^{\ell\in\calL}$
spans the $d_{N}=(g-1)(2N-1)$-dimensional space of holomorphic $N$-forms \cite{TW1}.  

The Bers quasiform \eqref{eq:PipiN} depends on the choice of limit set points $\{A_{\ell}\}$. Thus with an alternative choice  $\{\Ahat_{\ell}\}$ we find $\Psihat_{N}(x,y)-\Psi_{N}(x,y)$ is a holomorphic $N$-form in $x$ with
\begin{align}\label{eq:Psitilde}
	\Psihat_{N}(x,y)-\Psi_{N}(x,y)= -\sum_{r=1}^{d_{N}}\Phi_{r}(x)P_{r}(y), 
\end{align}
for a holomorphic $N$-form basis $\{ \Phi_{r}(x)\}_{r=1}^{d_{N}}$ and $P_{r}(y):=p_{r}(y)dy^{1-N}$ for some polynomials $p_{r}\in \mathfrak{P}_{2N-2}$. For the corresponding $N$-form spanning set of~\eqref{eq:chiTheta} we find
\begin{align}\label{eq:Thetatilde}
	\widehat{\Theta}_{N,a}^{\ell}(x)-\Theta_{N,a}^{\ell}(x)
	=\sum_{r=1}^{d_{N}} (p_{r})_{a}^{\ell}\Phi_{r}(x),
\end{align}
where for any $P(y)=p(y)dy^{1-N}$ with $p\in \mathfrak{P}_{2N-2}$,
we determine $p_{a}^{\ell}$ from  the trivial coboundary Eichler cocycle $P(\gamma_{a} y)-P(y)=\sum_{\ell\in\calL}p_{a}^{\ell}y_{a}^{\ell}dy^{1-N}$ \cite{Be1,TW1}. In particular, we may expand
$p(y)=\sum_{\ell\in\calL}p^{(\ell)}(w_{a})y_{a}^{\ell}$ and also find
\begin{align*}
	&p(\gamma_{a}y)\left(\frac{d(\gamma_{a} y)}{dy}\right)^{1-N}
	=p\left(w_{-a}+\frac{\rho_{a}}{y_{a}}\right)
	\left(-\frac{\rho_{a}}{y_{a}^{2} }\right)^{1-N}
	= (-1)^{N-1}\sum_{\ell\in\calL}
	p^{(\ell)}(w_{-a})\rho_{a}^{\ell+1-N}y_{a}^{2N-2-\ell},
\end{align*}
so  that
\begin{align}
	\label{eq:pal}
	p_{a}^{\ell} = (-1)^{N+1} \rho_{a}^{N-\ell-1}p^{(2N-2-\ell)}(w_{-a})-p^{(\ell)}(w_{a}),\quad a\in\Ip,\ell\in \calL.
\end{align}
\begin{remark}\label{rem:Psi1}
	Note that \eqref{eq:PipiN}  reproduces $\Psi_{1}$  of \eqref{eq:Psi1def} for $N=1$ provided $A_{0}\in \Omega_{0}$ i.e. $A_{0}$ is not a limit point.
	Thus $\Psi_{1}(x,y)$ is not a Bers quasiform since it has an extra simple pole at $x=A_{0}$. Nevertheless,  \eqref{eq:chiTheta} agrees with \eqref{eq:Psi1diff} for $N=1$ with 1-form basis $\Theta^{0}_{1,a}:=\nu_{a}$ for $a\in\Ip$. 
For an alternative choice of $\Ahat_{0}$ in \eqref{eq:Psi1def} we find that
	\begin{align}	\label{eq:Theta1tilde}		
		&\widehat{\Theta}_{1,a}^{0}(x)-\Theta_{1,a}^{0}(x)=0,
		\\
	\label{eq:Psi1tilde}
		&\Psihat_{1}(x,y)-\Psi_{1}(x,y)= \omega_{A_{0}-\Ahat_{0}}(x).
	\end{align}
Eqn.~\eqref{eq:Theta1tilde} agrees with \eqref{eq:Thetatilde} for $N=1$ (since $p^{0}_{a}=0$) but \eqref{eq:Psi1tilde} does not agree with \eqref{eq:Psitilde}.
\end{remark}
Lastly, we define a canonical symmetric meromorphic form $\omega_{N}(x,y)$ of weight $(N,N)$, generalising the classical bidifferential   $\omega(x,y)=\omega_{1}(x,y)$, as follows:
\begin{align}
	\label{eq:omegaN}
	\omega_{N}(x,y):=&
	\Psi_{N}^{(0,2N-1)}(x,y)dy^{2N-1}
	=\sum_{\gamma \in\Gamma} \dfrac{d(\gamma x)^{N}dy^{N}}{(\gamma x-y)^{2N}}.
\end{align}

\subsection{Expansion of a meromorphic $N$-form}
We describe the expansion of a meromorphic $N$-form in terms of finite sums of standard objects for Riemann surfaces of genus $g=0,1$ and $g\ge 2$. These expansions are reproduced in VOA formal Zhu recursion formulas at each genus. We begin with the cases $g=0,1$ before reviewing generalisation to the Schottky uniformisation for $g\ge2$.  
\subsubsection{Genus zero}
Let $H_{N}(x)$ denote a meromorphic $N$-form with poles at  $y_{1},\ldots,y_{n}\in\Chat:=\C \cup \{\infty\}$, the Riemann sphere, with $y_{k}\neq \infty$. We define
\begin{align}
	\label{eq:resHN}
	\Res_{y_{k}}^{j}H_{N}:=\frac{1}{\tpi}\oint_{\calC_{k}}(x-y_{k})^{j}H_{N}(x)dx^{1-N}, \quad j\ge 0,\, k=1,\ldots,n,
\end{align}
where $\calC_{k}$ is a simple Jordan curve surrounding $x=y_{k}$ but no other poles of $H_{N}(x)$. 
Note that $\Res_{y_k}^{j}H_{N}=0$ for $j\gg  0$.
We then find
\begin{proposition}\label{prop:HNzero}
	Let $H_{N}(x)$ be a meromorphic $N$-form  on $\Chat$ with poles at  $x=y_{k}\neq \infty$ for $k=1,\ldots,n$. 
	\begin{enumerate}
		\item [(i)] 		
		$H_{N}(x)$	can be expanded
		\begin{align*}
			H_{N}(x)=\sum_{k=1}^{n}\sum_{j\ge 0}(x-y_{k})^{-1-j}dx^{N}
			\Res_{y_{k}}^{j}H_{N}.
		\end{align*} 
		\item[(ii)] 
		Let $P(x)=p(x)dx^{1-N}$ with $p\in\Poly_{2N-2}$. Then 
		\begin{align*}
			\sum_{k=1}^{n}\sum_{\ell\in\calL}p^{(\ell)}(y_{k})
			\Res_{y_{k}}^{\ell}H_{N}=0.
		\end{align*} 
		\item[(iii)] Let $\pi_{N}(x,y):=(x-y)^{-1}+\sum_{\ell\in\calL}f_{\ell}(x)y^{\ell}$ for any Laurent series $f_{\ell}(x)$. Then
		\begin{align*}
			H_{N}(x)=dx^{N}\sum_{k=1}^{n}\sum_{j\ge 0}\pi_{N}^{(0,j)}(x,y_{k})
			\Res_{y_{k}}^{j}H_{N}.
		\end{align*}
	\end{enumerate}
\end{proposition} 
\begin{proof}
	Write $H_{N}=h(z)dz^{N}$ for some rational function $h$ on $\Chat$ with poles at $z=y_{k}$.  
	Consider a simple Jordan curve  $\calC$ whose interior contains $x,y_{1},\ldots,y_{n}$. Then 
	\begin{align*}
		0=& \oint_{\calC}\frac{h(z)}{x-z}dz = \oint_{\calC_{x}}\frac{h(z)}{x-z}dz
		+\sum_{k=1}^{n}\oint_{\calC_{k}}\frac{h(z)}{x-z}dz,
	\end{align*}
	where $\calC_{x}$ is a contour about $x$. Using
	$(x-z)^{-1}=\sum_{j\ge 0}(x-y_{k})^{-1-j}(z-y_{k})^{j}$ then part (i) follows from Cauchy's theorem.
	
	$P(x)H_{N}(x)$ is a meromorphic 1-form with poles at $x=y_{k}$. Then we find (ii):
	\begin{align*}
		0=&\oint_{\calC}P(x)H_{N}(x)
		=\sum_{k=1}^{n}\oint_{\calC_{k}}p(x)H_{N}(x)dx^{1-N}
		=\tpi \sum_{k=1}^{n}\sum_{\ell\in\calL}p^{(\ell)}(y_{k})
		\Res_{y_{k}}^{\ell}H_{N},
	\end{align*}
	using
	$p(x)=\sum_{\ell\in\calL}p^{(\ell)}(y_{k})(x-y_{k})^{\ell}$.
	(iii) follows directly from (i) and (ii).
\end{proof} 

\subsubsection{Genus one}
Using the standard elliptic function uniformisation, every meromorphic $N$-form on a torus $\Sone$ with modulus $\tau$ is of the form $h(z)dz^{N}$ for some elliptic function $h(z)$ with periods $\tpi$ and $\tpi\tau$ along standard homology cycles $\alpha$ and $\beta$. Consider the quasiperiodic Weierstrass function (related to the Weierstrass sigma function e.g.~\cite{WW}) 
\begin{align}\label{eq:P1def}
P_{1}(z):=P_{1}(z,\tau):=-\half -\sum_{m \in \Z,m\neq 0}\frac{e^{mz}}{1-q^{m}}=z^{-1}+\mbox{regular terms},
\end{align}
where $q=e^{\tpi \tau}$. $P_{1}(z)$ is odd in $z$ with periodicities\footnote{For  genus one we have $\omega_{y-A_{0}}(x)=\Psi_{1}(x,y)=\left(P_{1}(x-y)-P_{1}(x-A_{0})\right)dx$.}
\begin{align}\label{eq:P1_periods}
	P_{1}(z+\tpi)-P_{1}(z)=0,\quad 
	P_{1}(z+\tpi\tau)-P_{1}(z)=-1.
\end{align}
Define $\Res_{y_{k}}h$ for each pole $y_{k}\in \Sone$ of $h$ in a similar fashion to~\eqref{eq:resHN}. We then find (c.f. \S20.52 of~\cite{WW})
\begin{proposition}\label{prop:ellipticP1}
	Let $h(x)$ be a meromorphic elliptic function with poles at $x=y_{k}$ for $k=1,\ldots,n$. Then
	\begin{align}
		\label{eq:ellipticexp}
		h(x)=\frac{1}{\tpi}\oint_{\alpha}h(z)dz
		+\sum_{k=1}^{n}\sum_{j\ge 0}(-1)^{j}P_{1}^{(j)}(x-y_{k})\Res_{y_{k}}^{j}h.
	\end{align}
\end{proposition} 
\begin{proof}
	Let $\calC$ be the counter clockwise boundary contour of the torus fundamental parallelogram. Then
	\begin{align*}
		\oint_{\calC}P_{1}(x-z)h(z)dz=& \int_{0}^{\tpi}(P_{1}(x-z)-P_{1}(x-z-\tpi \tau))h(z)dz
		\\
		&
		+ \int_{0}^{\tpi\tau }(P_{1}(x-z-\tpi)-P_{1}(x-z))h(z)dz
		= -\oint_{\alpha}h(z)dz,
	\end{align*}
	using \eqref{eq:P1_periods}. Since 
	$P_{1}(x-z)=1/(x-z)+\ldots$ we find
	\begin{align*}
		\frac{1}{\tpi}\oint_{\calC}P_{1}(x-z)h(z)dz=& -h(x)+ 
		\sum_{k=1}^{n} 	\frac{1}{\tpi}\oint_{\calC_{k}}P_{1}(x-z)h(z)dz.
	\end{align*}
	Expanding $P_{1}(x-z)$ about $z=y_{k}$ we find 
	\[
	\frac{1}{\tpi}\oint_{\calC_{k}}P_{1}(x-z)h(z)dz
	=\sum_{j\ge 0}(-1)^{j}P_{1}^{(j)}(x-y_{k})\Res_{y_{k}}^{j}h.
	\]
	The proof follows by combining these results.
\end{proof}
\begin{remark}
	The expansion~\eqref{eq:ellipticexp} precisely matches the formal Zhu expansion for torus $n$-point functions \cite{Z1} where the residues are replaced by formal VOA residues.
\end{remark}
\subsubsection{Genus $g\ge 2$}
Let $H_{N}(x)$ denote a meromorphic $N$-form for $N\ge 1$ on the Riemann surface $\Sg$ in the Schottky parametrisation. Define $\Res_{y_{k}}^{j}H_{N}$ for each pole $y_{k}\in\D$ similarly to 	\eqref{eq:resHN} and also define
\begin{align*}
	\Res_{w_{a}}^{\ell}H_{N}:=\frac{1}{\tpi}\oint_{\calC_{a}}z_{a}^{\ell}
	H_{N}(z)dz^{1-N}, \quad \ell\in\calL,\, a\in\I,
\end{align*}
for $z_{a}:=z-w_{a}$
for Schottky scheme parameter $w_{a}$ and contour $\calC_{a}$. From the sewing relation~\eqref{eq:SchottkySewing2} we also note that 
\begin{align}\label{eq:ResRel}
	\Res_{w_{a}}^{\ell}H_{N}=(-1)^{N}\rho_{a}^{1-N+\ell}\Res_{w_{-a}}^{2-2N-\ell}H_{N}.
\end{align}
We obtain the following generalisation of Proposition~\ref{prop:HNzero} and Proposition~\ref{prop:ellipticP1}:
\begin{proposition}\label{prop:GNexp}
	For $N\ge 1$ let $H_{N}(x)$ be a meromorphic $N$-form on $\Sg$ with poles at $x=y_{k}\neq \infty$ for $k=1,\ldots,n$ in the Schottky parametrisation. 	
	\begin{enumerate} 
\item[(i)]	
For holomorphic $N$-form spanning set  $\{\Theta_{N,a}^{\ell}\}$ of~\eqref{eq:chiTheta} we find
	\begin{align}
		\label{eq:HNexp}
		H_{N}(x)=&\sum_{a\in\Ip}\sum_{\ell\in \calL}\Theta_{N,a}^{\ell}(x)
		\Res_{w_{a}}^{\ell}H_{N}
		+\sum_{k=1}^{n}\sum_{j\ge 0}\Psi_{N}^{(0,j)}(x,y_{k})dy_{k}^{N-1}
		\Res_{y_{k}}^{j}H_{N}.
	\end{align}
\item[(ii)] Let $P(z)=p(z)dz^{1-N}$ for $p\in\mathfrak{P}_{2N-2}$. Then for  $p_{a}^{\ell}$ of~\eqref{eq:pal} we have
\begin{align*}
	-\sum_{a\in\Ip}\sum_{\ell\in\calL}p_{a}^{\ell}
	\Res_{w_{a}}^{\ell}H_{N}
	+\sum_{k=1}^{n}\sum_{\ell\in\calL}p^{(\ell)}(y_{k})
	\Res_{y_{k}}^{\ell}H_{N}=0.
\end{align*} 
\end{enumerate}	
\end{proposition}
\begin{proof}
The proof for $N\ge 2$ is given in \cite{TW1}. We give  the proof for $N=1$ which is very similar. 
Consider a simple Jordan curve $\calC$ surrounding $x,y_{1},\ldots,y_{n}$ and the Schottky contours $\calC_{a}$ for all $a\in\I$ so that $\oint_{\calC}\Psi_{1}(x,\cdot)H_{1}(\cdot)=0$. 
Since $\Psi_{1}(x,z)\sim -(z-x)^{-1}\, dx $ for $x\sim z$ we find by Cauchy's theorem that
\begin{align*}
	0= -\tpi H_{1}(x)+\sum_{a\in\I}\oint_{\calC_{a}}\Psi_{1}(x,\cdot)H_{1}(\cdot)
	+\sum_{k=1}^{n}\oint_{\calC_{k}}\Psi_{1}(x,\cdot)H_{1}(\cdot).
\end{align*}
Since $\calC_{-a}=-\gamma_{a} \calC_{a}$ we find using \eqref{eq:Psi1diff} that
\begin{align*}
	\sum_{a\in\I}\oint_{\calC_{a}}\Psi_{1}(x,\cdot)H_{1}(\cdot)=&
	\sum_{a\in\Ip}\oint_{\calC_{a}(z)}\left(\Psi_{1}(x,z)-\Psi_{1}(x,\gamma_{a}z)\right)H_{1}(z)=
	\tpi \sum_{a\in\Ip}\nu_{a}(x)\Res_{w_{a}}^{0}H_{1}.
\end{align*}
Using $\Psi_{1}(x,z)=\sum_{j\ge 0}(z-y_{k})^{j} \Psi_{1}^{(0,j)}(x,y_{k})$ we find 
\begin{align}
	\notag
	\oint_{\calC_{k}}\Psi_{1}(x,\cdot)H_{N}(\cdot)=
	 \tpi\sum_{j\ge 0}\Psi_{1}^{(0,j)}(x,y_{k})
	\Res_{y_{k}}^{j}H_{1}.
\end{align}
Thus (i) holds for $N=1$. 
(ii) follows for $N=1$ from a similar analysis of $\oint_{\calC}H_{1}=0$. Since $\Res_{w_{a}}^{0}H_{1}=-\Res_{w_{-a}}^{0}H_{1}$ from \eqref{eq:ResRel} we find that 
$\sum_{k=1}^{n}\Res_{y_{k}}^{0}H_{1}=0$. This is equivalent to (ii) noting that $p^{0}_{a}=0$ for $N=1$ from \eqref{eq:pal}.
\end{proof}
\begin{remark}
Proposition~\ref{prop:GNexp}~(ii) implies that the expansion~\eqref{eq:HNexp} holds for all choices of $\Psi_{N}$ described in~\eqref{eq:Psitilde}, \eqref{eq:Thetatilde}, \eqref{eq:Theta1tilde}	 and~\eqref{eq:Psi1tilde}.
\end{remark} 
The main results of this paper, Theorems~\ref{theor:ZhuGenusg}
and \ref{theor:ZhuGenusgInt}, show that genus $g$ Zhu reduction follows the structure of~\eqref{eq:HNexp} for formal VOA and module residues. 

\subsection{Schottky sewing formulas for $\Psi_{N}(x,y)$ and $\Theta_{N,a}^{\ell}(x)$}\label{subsec:SchottkySewing}
In this section we review results of \cite{TW1} where expansion formulas for the Bers quasiform $\Psi_{N}(x,y)$ and $\Theta_{N,a}^{\ell}(x)$ are given in terms of the sewing parameters $\rho_{a}$ for $N\ge 2$. These results also hold for $N=1$ with $\Psi_{1}$ of \eqref{eq:Psi1def}. These expressions are much more useful in VOA theory than Poincar\'e sums.  Let $\Pi_{N}(x,y)=\pi_{N}(x,y)dx^{N}dy^{1-N}$  for $N\ge 1$ with
\begin{align}
	\label{eq:piN}
	\pi_{N}(x,y):=\frac{1}{x-y}\prod_{i\in\calL}\frac{y-A_{i}}{x-A_{i}}=\frac{1}{x-y}
	-\sum_{i\in\calL}\frac{1}{x-A_{i}}\calQ_{i}(y)
	=\frac{1}{x-y}+\sum_{\ell\in\calL} f_{\ell}(x)y^{\ell},
\end{align}
for Lagrange polynomial $\calQ_{i}(y):=\prod_{j\neq i}\frac{y-A_{j}}{A_{i}-A_{j}}\in \Poly_{2N-2}(y)$ (where $\sum_{i\in\calL} \calQ_{i}(y)=1$) and
\begin{align}
	\label{eq:flx}
	f_{\ell}(x):=-\sum_{i\in\calL}\frac{1}{x-A_{i}}\calQ_{i}^{(\ell)}(0), \quad \ell\in\calL.	
\end{align}
It is useful to define the following forms  labelled by $a,b\in\I$ and integers $m,n\ge 0$ constructed from moment integrals of $\Pi_{N}(x,y)$  as follows:
\begin{align}
	L_{b}^{n}(x):=&\frac{\rho_{b} ^{\half n}}{\tpi}
	\oint\limits_{\mathcal{C}_{b}(y)}
	\Pi_{N}(x,y)y_{b}^{-n-1}\,dy^{N}
	\label{eq:Ldef}
	=\rho_{b}^{\half n} \pi_{N}^{(0,n)}(x,w_{b})dx^{N},  
	\\
	R_{a}^{m}(y):=&\frac{(-1)^{N}\rho_{a} ^{\half (m+1)}}{\tpi}
	\oint\limits_{\mathcal{C}_{-a}(x)}
	\Pi_{N}(x,y)x_{-a}^{-m-1}\,dx^{1-N}
	\label{eq:Rdef}
	\\
	=&(-1)^{N}\rho_{a}^{\half (m+1)} \pi_{N}^{(m,0)}(w_{-a},y)dy^{1-N},  \notag
\end{align}
where $y_{b}:=y-w_{b}$ and $x_{-a}:=x-w_{-a}$.
We let $L(x)=(L_{b}^{n}(x))$ and 
$R(x)=(R_{a}^{m}(y))$ denote doubly indexed
infinite row and column vectors, respectively, indexed by $a,b\in \I$ and $m,n\ge 0$.
We also define the doubly indexed matrix $A=(A_{ab}^{mn})$ with components 
\begin{align}
	\notag
	A_{ab}^{mn} :=&\frac{(-1)^{N}\rho_{a} ^{\half (m+1)}}{\tpi}
	\oint\limits_{\mathcal{C}_{-a}(x)}
	L_{b}^{n}(x)x_{-a}^{-m-1}\,dx^{1-N}
	=
	\frac{\rho_{b}^{\half n}}{\tpi}
	\oint\limits_{\mathcal{C}_{b}(y)} R_{a}^{m}(y)  y_{b}^{-n-1}\,dy^{N} 
	\\
	\label{eq:Adef}
	= &
	\begin{cases}(-1)^{N}\rho_{a}^{\half (m+1)}\rho_{b}^{\half n}\pi_{N}^{(m,n)}(w_{-a},w_{b}),&a\neq-b,\\ 
		(-1)^{N}\rho_{a}^{\half(m+n+1)}e_{m}^{n}(w_{-a}),&a=-b,
	\end{cases}
\end{align}
where $e_{m}^{n}(y):=\sum_{\ell\in\calL}\binom{\ell}{n}f_{\ell}^{(m)}(y)y^{\ell-n}$ for $f_{\ell}$ of~\eqref{eq:flx}.
We note that $A_{a,-a}^{mn}=0$ for all $n> 2N-2$.
We also define the matrix $\Delta$ with components
\begin{align}
	\Delta_{ab}^{mn}:=\delta_{m,n+2N-1}\delta_{ab}.\label {eq:Deltadef}
\end{align}
Let  $\Ltilde(x):=L(x)\Delta$ and $\Atilde:=A\Delta $. These are independent of the $f_{\ell}(x)$ terms with
\begin{align}
	\label{eq:Ptilde}
	\Ltilde_{b}^{n}(x)&=
	\frac{\rho_{b}^{\half (n+2N-1)}}{(x-w_{b})^{n+2N}}dx^{N},
	\\
	\label{eq:Rtilde}
	\Atilde_{ab}^{mn}&=
	\begin{cases}
		\displaystyle{(-1)^{m+N}\binom{m+n+2N-1}{m}
			\frac{\rho_{a}^{\half (m+1)}\rho_{b}^{\half (n+2N-1)}}
			{(w_{-a}-w_{b})^{m+n+2N}}},&a\neq-b\\ 
		0,&a=-b.
	\end{cases}
\end{align}
We define $(I-\Atilde)^{-1}:=\sum_{k\geq 0}\Atilde^{k}$ where $I$ denotes the infinite identity matrix. Then 
$\Psi_{N}(x,y)$ can be expressed in terms of $\Pi_{N},\Ltilde,R,\Atilde$ as follows \cite{TW1}
\begin{theorem}
	\label{theor:Psisew} 
	For all $N\ge 1$ and $x,y\in \D$ 
	\begin{align}
		\Psi_{N} (x,y)=\Pi_{N}(x,y)+\Ltilde(x)
		\left(I-\Atilde \right)^{-1}R(y),
		\label{eq:Psisew}
	\end{align}
	where $\left(I-\Atilde \right)^{-1}$ is convergent for all $(\bm{w,\rho})\in \Cg$. 
\end{theorem}
We also find the holomorphic $N$-form $\Theta_{N,a}^{\ell}(x)$ of~\eqref{eq:chiTheta} is given by \cite{TW1}:
\begin{proposition}\label{prop:Thetaexp}
	Let $a\in\Ip$ and $\ell\in\calL$. Then
	\begin{align}
		\label{eq:Thetaexp}
		\Theta_{N,a}^{\ell}(x)=T_{a}^{\ell}(x)+(-1)^{N }\rho_{a}^{N-1-\ell}T_{-a}^{2N-2-\ell}(x),
	\end{align}
	where $
	T_{a}^{\ell}(x):=
	\rho_{a}^{-\half \ell}
	(L(x) +\Ltilde(x)(I-\Atilde)^{-1}A  )_{a}^{\ell}$.
\end{proposition}
\section{Vertex Operator Algebras and Genus Zero Zhu Recursion}
\label{sec:VOAZhu0}
\subsection{Vertex operator algebras} 
For indeterminates  $x,y$ we adopt the binomial expansion convention that for $m\in\Z$ 
\begin{align}
	\label{eq:binexp}
	(x+y)^{m}:=\sum_{k\ge 0}\binom{m}{k}x^{m-k}y^{k},
\end{align} 
i.e. we expand in non-negative powers of the second indeterminate $y$.
We  review some aspects of vertex operator algebras e.g.~\cite{K,FHL,LL,MT1}. A vertex operator algebra (VOA) is a quadruple $(V,Y(\cdot,\cdot),\vac,\omega)$ consisting of a graded vector space $V=\bigoplus_{n\ge 0}V_{n}$, with $\dim V_{n}<\infty$,  with two distinguished elements: the vacuum vector $\vac\in V_{0}$  and the Virasoro conformal vector $\omega\in V_{2}$.  
For each $u\in V$ there exists a vertex operator, a formal Laurent series in $z$, given by   
\begin{align*}
	Y(u,z)=\sum_{n\in\Z}u(n)z^{-n-1},
\end{align*}
for \emph{modes} $u(n)\in\End(V)$. 
For each  $u,v\in V$ we have $u(n)v=0$ for all $n\gg 0$, known as \emph{lower truncation}, and 
$u=u(-1)\vac $ and $u(n)\vac =0$ for all $n\ge 0$,
known as \emph{creativity}.
The vertex operators also obey  \emph{locality}:
\begin{align*}
	(x-y)^{N}[Y(u,x),Y(v,y)]=0,\quad N\gg 0.
\end{align*}
For the Virasoro conformal vector  $\omega$
\begin{align*}
	Y(\omega,z)&=\sum_{n\in\Z}L(n)z^{-n-2},
\end{align*}
where the operators  $L(n)=\omega(n+1)$  satisfy the Virasoro  algebra
\begin{align*}
	[L(m),L(n)]=(m-n)L(m+n)+\frac{C}{2}\binom{m+1}{3}\delta_{m,-n}\Id_{V},
\end{align*}
for a constant \emph{central charge} $C\in\C$.   Vertex operators  satisfy the \emph{translation property}: 
\begin{align*}
	Y(L(-1)u,z)=\del Y(u,z).
\end{align*}
Finally, $V_{n}=\{v\in V:L(0)v=nv\}$ where
$v\in V_{n}$ is the \emph{(conformal) weight} $\wt(v)=n$. 

\medskip
We quote a number of basic VOA properties e.g.~\cite{K,FHL,LL,MT1}. For $u\in V_{N}$ 
\begin{align}
	\label{eq:vjVm}
	u(j):V_{k}\rightarrow V_{k+N-j-1}.
\end{align}
The commutator identity: for all $u,v\in V$ we have 
\begin{align}
	[u(k),Y(v,z)]=\left (\sum_{j\ge 0}Y(u(j)v,z)\partial_{z} ^{(j)}\right)z^{k}.
	\label{eq:ComId}
\end{align} 
The associativity identity: for each $u,v\in V$  there exists $M\ge 0$ such that
\begin{align}
	(x+y)^{M}Y(Y(u,x)v,y)&=(x+y)^{M}Y(u,x+y)Y(v,y).
	\label{eq:AssocId}
\end{align}
M\"obius maps: $\{L(0),L(\pm 1)\}$ generates an $\slLie_{2}$ Lie algebra where for all $u\in V$ there are the following associated formal M\"obius maps
\begin{align}
	e^{xL(-1)}Y(u,z)e^{-xL(-1)}&=Y(e^{xL(-1)}u,z)=Y(u,z+x),
	\label{eq:Trans}
	\\
	x^{L(0)}Y(u,z)x^{-L(0)}&=Y(x^{L(0)}u, xz),\quad x\neq 0,
	\label{eq:Dil}
	\\
	e^{xL(1)}Y(u,z)e^{-xL(1)}&=
	Y\left(e^{x(1-xz)L(1)}(1-xz)^{-2L(0)}u,\frac{z}{1-xz}\right).
	\label{eq:Inv}
\end{align}
$L(\pm 1),L(0)$ obey the following standard identities e.g. \cite{FHL}, Lemma~5.2.2:
\begin{align}
	x^{L(0)} L(\pm 1)x^{-L(0)}&=x^{\mp 1}L(\pm 1),\quad x\neq 0,
	\label{eq:L0Lpm1conj}
	\\[3pt]
	e^{x L(\pm 1)}L(0)e^{-x L(\pm 1)}&=L(0)\pm xL(\pm 1),
	\label{eq:LpmL0conj}
	\\[3pt]
	e^{x L(1)}L(-1)e^{-x L(1)}&=L(-1)+2xL(0)+x^{2}L(1).
	\label{eq:L1Lm1conj}
\end{align}	
We also note the following useful identities:
\begin{proposition}
With $z=1-xy\neq 0$ we have
\begin{align}
	e^{x L(-1)}e^{yL(1)}&=z^{L(0)}e^{y L(1)} e^{x L(-1)}z^{L(0)}
	\label{eq:Lm1L1com1}
	\\
	&= 
	z^{2L(0)}e^{yzL(1)}e^{xz^{-1} L(-1)}.
	\label{eq:Lm1L1com2}
\end{align}
\end{proposition}
\begin{proof}
Let $K(x,y):=z^{L(0)}e^{y L(1)} e^{x L(-1)}z^{L(0)}$, the right hand side of~\eqref{eq:Lm1L1com1}. We will show that $\del_{x}K(x,y)=L(-1)K(x,y)$.
Using~\eqref{eq:L0Lpm1conj}--\eqref{eq:L1Lm1conj}  we find $\del_{x}K(x,y)$ is given by
	\begin{align*}
	& z^{L(0)-1}\Big(
	-yL(0)e^{y L(1)} e^{x L(-1)}
	+ze^{y L(1)}L(-1) e^{x L(-1)}
 -ye^{y L(1)} e^{x L(-1)}L(0)\Big)z^{L(0)}
	\\
	&= z^{L(0)-1}\Big(
	-yL(0)e^{y L(1)} e^{x L(-1)}
	+z(L(-1)+2yL(0)+y^{2}L(1))e^{y L(1)}e^{x L(-1)}
	\\
	&\quad   -ye^{y L(1)}\left(L(0)-xL(-1)\right) e^{x L(-1)}\Big)z^{L(0)}
	\\
	&= z^{-1}\Big(
	z^{2}L(-1)+(y-2xy^{2})L(0)+y^{2}L(1)\Big)K(x,y)
	\\
	&\quad 
	 +z^{L(0)-1}\Big(
	 xy L(-1)-(y-2xy^{2})L(0)-y^2zL(1) \Big)z^{-L(0)}K(x,y)
=L(-1)K(x,y).
	\end{align*}
Thus  $K(x,y)=e^{xL(-1)}K(0,y)=e^{x L(-1)}e^{yL(1)}$.
Eqn.~\eqref{eq:Lm1L1com1} implies~\eqref{eq:Lm1L1com2} using~\eqref{eq:L0Lpm1conj}.
	\end{proof}
	\begin{remark} 
Consider the fundamental $\slLie_{2}(\C)$ representation with the canonical identifications:
 $L(-1)=\left(\begin{smallmatrix}
	   		0 & 1 \\ 0 & 0\end{smallmatrix}\right)$, 
 $L(0)=\frac{1}{2}\left(\begin{smallmatrix}
	   		1 & 0 \\ 0 & -1 \end{smallmatrix}\right)$
 and $L(1)=\left(\begin{smallmatrix}
	   		0 &  0 \\ -1 & 0\end{smallmatrix}\right)$
 giving   $\SL_{2}(\C)$  elements 
 $e^{x L(-1)}=\left(\begin{smallmatrix}
	   		1 & x \\ 0 & 1\end{smallmatrix}\right)$, 
 $e^{yL(1)}=\left(\begin{smallmatrix}
	   		1&  0 \\ -y & 1\end{smallmatrix}\right)$
 and 
 $
x^{L(0)} = \left(\begin{smallmatrix}
 		\sqrt{x}  & 0 \\ 0 & 1/\sqrt{x}
 	\end{smallmatrix}\right)$.
 It is straightforward to confirm that 
 	\eqref{eq:L0Lpm1conj}--\eqref{eq:Lm1L1com2} hold.  
	\end{remark}

Associated with the formal M\"obius map $z\rightarrow {\rho}/{z}$, for  a given scalar $\rho\neq 0$,
we define an adjoint vertex operator \cite{FHL, L}
\begin{align}\label{eq:Yadj}
	Y_{\rho}^{\dagger}(u,z):=\sum_{n\in\Z}u_{\rho}^{\dagger}(n)z^{-n-1}=Y\left(e^{\frac{z}{\rho}L(1)}\left(-\frac{\rho}{z^{2}}\right)^{L(0)}u,\frac{\rho}{z}\right).
\end{align}
For quasiprimary $u$ (i.e. $L(1)u=0$) of weight $N$ we have
\begin{align}
	\label{eq:udagger}
	u_{\rho}^{\dagger}(n)=(-1)^{N}\rho^{n+1-N}u(2N-2-n),
\end{align}
e.g. $L_{\rho}^{\dagger}(n)=\rho^{n}L(-n)$.
A bilinear form $\langle \cdot,\cdot\rangle_{\rho}$ on $V$  is said to be invariant if
\begin{align}
	\label{eq:LiZinvar}
	\langle Y (u, z)v, w\rangle_{\rho} = \langle v, Y_{\rho}^{\dagger}
	(u, z)w\rangle_{\rho},\quad \forall\;u,v,w\in V.
\end{align}
When $\rho=1$ we write $Y^{\dagger}(u,z)$ for the adjoint and $\langle \cdot,\cdot\rangle$ for the bilinear form.
$\langle \cdot,\cdot\rangle_{\rho}$ is symmetric where $\langle u,v\rangle_{\rho}=0$ for $\wt(u)\neq\wt(v)$~\cite{FHL} and 
\begin{align}
	\label{eq:uvA}
	\langle u,v\rangle_{\rho}=\rho^{N}\langle u,v\rangle,
	\quad N=\wt(u)=\wt(v).
\end{align} 
We assume throughout this paper that $V$ is of strong CFT-type i.e. $V_{0} = \C\vac$ and $L(1)V_{1}= 0$. Then the bilinear form with normalisation $\langle \vac,\vac\rangle_{\rho}=1$ is unique~\cite{L}. We also assume that $V$ is simple and isomorphic to the contragredient $V$-module $V'$ \cite{FHL}. 
Then the bilinear form is non-degenerate~\cite{L}. We refer to this unique invariant non-degenerate bilinear form as the Li-Zamolodchikov (Li-Z) metric.

\subsection{Genus zero correlation functions}
We consider some general properties of genus zero $n$-point correlation functions including M\"obius transformation  properties and genus zero Zhu recursion \cite{Z1}. 
 We also describe generalised Ward identities for genus zero $n$-point functions associated with any quasiprimary vector of weight $N$. These results are directly related to Proposition~\ref{prop:HNzero} concerning genus zero meromorphic $N$-forms and are exploited in the later genus $g>0$ analysis.

Define the genus zero $n$-point (correlation) function for $\bm{v}:=v_{1},\ldots,v_{n}$ inserted at $\bm{y}:=y_{1},\ldots,y_{n}$, respectively,  by\footnote{The superscript $(0)$ on $\Zzero(\bm{v,y})$ and on $\Fzero(\bm{v,y})$ in \eqref{eq:BFForm} refers to the genus.}
\begin{align*}
	\Zzero(\bm{v,y}):=\Zzero(\ldots;v_{k},y_{k};\ldots)=\langle \vac,\bm{Y(v,y)}\vac\rangle,
\end{align*}
for ($\rho=1$) Li-Z metric $\langle\cdot,\cdot\rangle$ and 
\begin{align*}
	\bm{Y(v,y)}:= Y(v_{1},y_{1}) \ldots Y (v_{n},y_{n}).
\end{align*}
$\Zzero(\bm{v,y})$ can be extended to a rational function 
in $\bm{y}$ in the domain $|y_{1}|>\ldots >|y_{n}|$.  We define the $n$-point correlation differential form for $v_{k}$ of weight $\wt(v_{k})$:
\begin{align}\label{eq:BFForm}
	\Fzero(\bm{v,y}):=\Zzero(\bm{v,y})\bm{dy^{\wt(v)}},
\end{align}
where $\bm{dy^{\wt(v)}}:=\prod_{k=1}^{n}dy_{k}^{\wt(v_{k})}$. We extend this by linearity for non-homogeneous vectors.
\begin{lemma} 
\label{lem:FMobiusQP}	
Let $v_{k}$ be quasiprimary of weight $\wt(v_{k})$ for $k=1,\ldots,n$. Then for all $\gamma=\left(\begin{smallmatrix}a&b\\c&d\end{smallmatrix}\right)\in\SL_{2}(\C)$ we have
	\begin{align}\label{eq:MobF0}
		\Fzero(\bm{v,y})= \Fzero\left(v_{1},\gamma y_{1}
		; \ldots ; v_{n},\gamma y_{n}
		\right).
	\end{align}
\end{lemma} 
\begin{proof}
The RHS of~\eqref{eq:MobF0} is well-defined for all $\gamma\in\SL_{2}(\C)$ since $\Zzero(\bm{v,y})$ can be extended to a rational function and $d(\gamma y_{k})=(cy_{k}+d)^{-2} dy_{k}$.
Since $e^{xL(\pm 1)}\vac=e^{xL(0)}\vac=\vac$ we obtain~\eqref{eq:MobF0} for the $\SL_{2}(\C)$ generators
$\left(\begin{smallmatrix}
	1&x\\0&1
\end{smallmatrix}\right)$, 
$x^{-1/2}\left(\begin{smallmatrix}
	x & 0 \\ 0 & 1
\end{smallmatrix}\right)$ 
and 
$\left(\begin{smallmatrix}
	1 & 0 \\ -x & 1
\end{smallmatrix}\right)$, 
respectively, by~\eqref{eq:Trans}--\eqref{eq:Inv}. 
Thus~\eqref{eq:MobF0} holds for all $\gamma\in\SL_{2}(\C)$.
\end{proof}
\begin{remark}
For quasiprimary $v_{1},\ldots,v_{n}$ we therefore find $\Fzero(\bm{v,y})$ is a genus zero meromorphic form in each $y_{k}$ of weight $\wt(y_{k})$.
\end{remark}
In general, $\Fzero(\bm{v,y})$ is not a meromorphic form and Lemma~\ref{lem:FMobiusQP} generalises as follows:
\begin{proposition}
\label{prop:ZMobius}
Let $v_{k}$ be of weight $\wt(v_{k})$ for $k=1,\ldots,n$. Then for all $\gamma=\left(\begin{smallmatrix}a&b\\c&d\end{smallmatrix}\right)\in\SL_{2}(\C)$ we have
\begin{align}\label{eq:MobZ0}
	\Zzero(\bm{v,y})= \Zzero\left(
	\ldots ;e^{ -c(cy_{k}+d)L(1)}	\left(cy_{k}+d\right)^{-2\wt(v_{k})}v_{k},\gamma y_{k};\ldots
	\right).
\end{align}
\end{proposition}
\begin{remark}
Note that although~\eqref{eq:MobZ0} holds automatically for the M\"obius maps generated by $L(0),L(\pm 1)$ from~\eqref{eq:Trans}--\eqref{eq:Dil}, we must confirm that~\eqref{eq:MobZ0} holds for all M\"obius maps.
\end{remark}
\begin{proof}
We prove~\eqref{eq:MobZ0} for all quasiprimary weight $r_{k}+\wt(u_{k})$ descendants $L(-1)^{r_{k}}u_{k}$  of quasiprimary vectors $u_{k}$ for $k=1,\ldots,n$. In particular, we consider the exponential generating function for quasiprimary descendant $n$-point functions 
\begin{align*}
	G(\bm{u,y,\xi}):=&\sum_{r_{1}\ge 0}	\frac{\xi_{1}^{r_{1}}}{r_{1}!}
	\ldots
	\sum_{r_{n}\ge 0}\frac{\xi_{n}^{r_{n}}}{r_{n}!}
	\Zzero\left(
\ldots ;L(-1)^{r_{k}}u_{k},y_{k} ;\ldots
	\right),
\end{align*}
for indeterminates $\bm{\xi}:=\xi_{1},\ldots,\xi_{n}$.
Equation~\eqref{eq:MobZ0} is then equivalent to the identity
\begin{align}\label{eq:GZid}
	G(\bm{u,y,\xi})=
	\Zzero\left(\ldots ;e^{ -c(cy_{k}+d)L(1)}
	\left(cy_{k}+d\right)^{-2L(0)}
	e^{\xi_{k}L(-1)}u_{k},\gamma y_{k};\ldots
	\right),
\end{align}
for all $\gamma\in\SL_{2}(\C)$.
Let $u$ be quasiprimary of weight $\wt(u)$.
Using~\eqref{eq:L0Lpm1conj} and~\eqref{eq:Lm1L1com1} we find
\begin{align*}
&e^{ -c(cy+d)L(1)}
\left(cy+d\right)^{-2L(0)}
e^{\xi L(-1)}u
=
\left(cy+d\right)^{-2L(0)}
e^{ -c(cy+d)^{-1}L(1)}
e^{\xi L(-1)}u
\\
& =\left(cy+d\right)^{-2L(0)}\left(1+\frac{\xi c}{cy+d}\right)^{-L(0)}
e^{\xi L(-1)}
e^{ -c(cy+d)^{-1}L(1)}\left(1+\frac{\xi c}{cy+d}\right)^{-L(0)}u
\\
&=
\left(cy+d\right)^{-2L(0)}\left(1+\frac{\xi c}{cy+d}\right)^{-L(0)-\wt(u)}
e^{\xi L(-1)}u.
\end{align*}
 Applying~\eqref{eq:L0Lpm1conj} twice again we find
\begin{align*}
e^{ -c(cy+d)L(1)}
\left(cy+d\right)^{-2L(0)}
e^{\xi L(-1)}u
=
\left(c(y+\xi)+d\right)^{-2\wt(u)}
e^{(cy+d)^{-2}\left(1+\frac{\xi c}{cy+d}\right)^{-1}\xi L(-1)}u.
\end{align*}
But 
\begin{align*}
 \frac{\xi}{(cy+d)^{2}}\left(1+\frac{\xi c}{cy+d}\right)^{-1}
 =\frac{\xi}{(cy+d)(c(y+\xi)+d)}
 =\gamma (y+\xi)-\gamma y.
\end{align*}
Thus 
\begin{align*}
Y\left(e^{ -c(cy+d)L(1)}
\left(cy+d\right)^{-2L(0)}
e^{\xi L(-1)}u,\gamma y\right)
&=
\left(c(y+\xi)+d\right)^{-2\wt(u)}Y\left( 
e^{(\gamma (y+\xi)-\gamma y) L(-1)}u,\gamma y\right)
\\
&=
\left(c(y+\xi)+d\right)^{-2\wt(u)}Y\left( u,\gamma (y+\xi)\right),
\end{align*}
by translation~\eqref{eq:Trans}.
Hence the RHS of~\eqref{eq:GZid} can be re-expressed as follows:
\begin{align*}
&\Zzero\left(\ldots ;e^{ -c(cy_{k}+d)L(1)}
\left(cy_{k}+d\right)^{-2L(0)}
e^{\xi_{k}L(-1)}u_{k},\gamma y_{k};\ldots
\right)
\\
&=\Zzero\left(
\ldots ; u_{k},\gamma (y_{k}+\xi_{k}) ;\ldots
\right)
\prod_{k=1}^{n}\left(c(y_{k}+\xi_{k})+d\right)^{-2\wt(u_{k})}
\\
&=\Zzero\left(
\ldots ; u_{k},y_{k}+\xi_{k} ;\ldots
\right)
= \Zzero\left(
\ldots ;e^{\xi_{k} L(-1)}u_{k},y_{k} ;\ldots
\right)
=G(\bm{u,y,\xi}),
\end{align*}
by~\eqref{eq:Trans} and~\eqref{eq:MobF0}.
Thus the identity~\eqref{eq:GZid} holds.
\end{proof}
\medskip

Let  $u$ be quasiprimary of weight $N$. Using~\eqref{eq:udagger} we find that for all $\ell\in\calL$
\begin{align*}
u(\ell)\vac=
u^{\dagger}(\ell)\vac=0,
\end{align*}
so that
\begin{align*}
0=\langle u^{\dagger}(\ell)\vac,\bm{Y(v,y)}\vac\rangle=
\langle \vac,u(\ell)\bm{Y(v,y)}\vac\rangle=
\langle \vac,\left[u(\ell),\bm{Y(v,y)}\right]\vac\rangle.
\end{align*}
This implies from the commutator identity~\eqref{eq:ComId} that for $\ell\in\calL$
\begin{align}
\label{ZhuPrepLemma}
\sum_{k=1}^n \left(\sum_{j\in\calL}
\Zzero(\ldots;u(j)v_{k},y_{k};\ldots)\partial_{k} ^{(j)} \right) y_{k}^{\ell}=0,
\end{align}
where $\del_{k}:=\del_{y_{k}}$. 
Thus we find a general genus zero Ward identity:
\begin{proposition}\label{prop:ZeroWard}
Let  $u$ be quasiprimary of weight $N$. 
Then for all $p\in \Poly_{2N-2}$ we have
\begin{align*}
	\sum_{k=1}^n \sum_{\ell\in\calL}p^{(\ell)}(y_{k})
	\Zzero(\ldots;u(\ell)v_{k},y_{k};\ldots)
	= 0.
\end{align*}
\end{proposition}
\begin{remark}
Proposition~\ref{prop:ZeroWard} generalizes the current algebra Ward identity for $N=1$ for $u\in V_{1}$ and the conformal Ward identity for $N=2$ for conformal vector $u=\omega\in V_{2}$.  
\end{remark}	
We now develop a genus zero Zhu recursion formula for $\Zzero(u,x;\bm{v,y})$ for  quasiprimary $u$ of weight $N\ge 1$. First note that for all $v\in V$ and $s<0$
\begin{align*}
0=\langle u(-s-1)\vac,v\rangle=
\langle \vac,u^{\dagger}(-s-1)v\rangle=(-1)^{N}\langle\vac,u(s+2N-1)v\rangle.
\end{align*}
Using~\eqref{eq:ComId} and that $u(s+2N-1)\vac=0$ for $s\ge 0$ and $N\ge 1$ we find
\begin{align*}
\Zzero(u,x;\bm{v,y})
&=\sum_{s\ge 0}x^{-s-2N}\left\langle\vac,u(s+2N-1)\bm{Y(v,y)}\vac\right\rangle
\\
&=\sum_{k=1}^{n}\sum_{s\ge 0}x^{-s-2N}
\left\langle\vac,
\ldots\sum_{j\ge 0}\del_{k}^{(j)}\left(y_{k}^{s+2N-1}\right)Y(u(j)v_{k},y_{k})\ldots \vac\right\rangle
\\
&=\sum_{k=1}^{n}\sum_{j\ge 0}x^{-1}
\del_{k}^{(j)}\left(\sum_{s\ge 0}\left(\frac{y_{k}}{x}\right)^{s+2N-1}
\right)
\Zzero(\ldots;u(j)v_{k},y_{k};\ldots)
\\
&=\sum_{k=1}^{n}\sum_{j\ge 0}\widetilde{\pi}^{(0,j)}_{N}(x,y_{k})\Zzero(\ldots;u(j)v_{k},y_{k};\ldots),
\end{align*}
recalling~\eqref{eq:delconv} and~\eqref{eq:binexp}, where 
\begin{align*}
\widetilde{\pi}_{N}(x,y):=x^{-1}\sum_{s\ge 0}\left(\frac{y}{x}\right)^{s+2N-1}
=\left(\frac{y}{x}\right)^{2N-1}\cdot\frac{1}{x-y} =\frac{1}{x-y}-\sum_{\ell\in\calL}\frac{y^{\ell}}{x^{\ell+1}}.
\end{align*}
Proposition~\ref{prop:ZeroWard} implies that we may generalise $\widetilde{\pi}_{N}(x,y)$ to
\begin{align}\label{eq:piNdef}
\pi_{N}(x,y):=\frac{1}{x-y}+\sum_{\ell\in\calL}f_{\ell}(x)y^{\ell},
\end{align}
for \emph{any} formal Laurent series $f_{\ell}(x)$ for $\ell=0,\ldots,2N-2$.  Thus we find
\begin{proposition}[Quasiprimary Genus Zero Zhu Recursion]
\label{prop:GenusZeroZhu}
Let $u$ be quasiprimary of weight $N\ge 1$. Then
\begin{align}\label{eq:GenusZeroZhu}
	\Zzero(u,x;\bm{v,y})=\sum_{k=1}^{n}\sum_{j\ge 0} \pi_{N}^{(0,j)}(x,y_{k})\Zzero(\ldots;u(j)v_{k},y_{k};\ldots).
\end{align}
\end{proposition}
\begin{remark}
Proposition~\ref{prop:GenusZeroZhu} implies Lemma~2.2.1 of \cite{Z1} for a particular choice of Laurent series $f_{\ell}$ in~\eqref{eq:piNdef}. We note that $\pi_N$ is independent of the VOA $V$ and that the RHS of~\eqref{eq:GenusZeroZhu} is independent of the choice of $f_{\ell}$ due to the Ward identity of Proposition~\ref{prop:ZeroWard}.
\end{remark}
Define $L^{(i)}(-1):=\frac{1}{i!}L(-1)^{i}$ so that  $Y\left(L^{(i)}(-1)u,x\right)=\del ^{(i)}Y(u,x)$. 
\begin{corollary}[General Genus Zero Zhu Recursion]
\label{cor:GenGenusZeroZhu}
Let $L^{(i)}(-1)u$ be a quasiprimary descendant of $u$ of $\wt(u)=N$. Then
\begin{align*}
	&\Zzero \left(L^{(i)}(-1)u,x;\bm{v,y}\right)=\sum_{k=1}^{n}\sum_{j\ge 0}\pi_{N}^{(i,j)}(x,y_{k})\Zzero (\ldots;u(j)v_{k};y_{k};\ldots).
\end{align*}
\end{corollary}
\begin{remark}\label{rem:flzero}
Choosing $f_{\ell}(x)=0$ we obtain the neatest  form of Zhu recursion
(which implies that genus zero $n$-point functions are rational) using
\[
\pi_{N}^{(i,j)}(x,y)=(-1)^{i}\binom{i+j}{i}\frac{1}{(x-y)^{1+i+j}}.
\] 
In the next section when we consider genus $g$ partition functions in the Schottky sewing scheme, we will choose $f_{\ell}(x)$ of~\eqref{eq:flx}.  
\end{remark}
Associativity~\eqref{eq:AssocId} implies (up to a rational multiplier) that
\begin{align}
\Zzero(u,x+y_{1};\bm{v,y})
&=\Zzero(Y(u,x)v_{1},y_{1};v_{2},y_{2}\ldots )
=\sum_{m\in\Z}\Zzero(u(m)v_{1},y_{1};\ldots )x^{-m-1}.
\label{eq:Z0assoc}
\end{align}
But Zhu recursion~\eqref{eq:GenusZeroZhu}  implies that 
\begin{align}
\notag
\Zzero(u,x+y_{1};\bm{v,y})
&=
\sum_{j\ge 0}\pi_{N}^{(0,j)}(y_{1}+x,y_{1})\Zzero(u(j)v_{1},y_{1};\ldots)
\\
\label{eq:Z0Zhu}
&\quad+\sum_{k=2}^{n}\sum_{j\ge 0}\pi_{N}^{(0,j)}(y_{1}+x,y_{k})\Zzero(\ldots;u(j)v_{k},y_{k};\ldots).
\end{align}
Note  the $x$-expansions 
\begin{align*}
\pi_{N}^{(0,j)}(y+x,z)&=
\begin{cases}
	x^{-j-1}
	+\sum_{i\ge 0}\E^{j}_{i}(y)x^{i}, & y=z,
	\\
	\sum_{i\ge 0}\pi_{N}^{(i,j)}(y,z)x^{i}, & y\neq z,
\end{cases}
\end{align*}
where
\begin{align}
\label{eq:Eji}
\E^{j}_{i}(y):=\sum_{\ell\in\calL}f_{\ell}^{(i)}(y)\;\del^{(j)}y^{\ell}.
\end{align}
Comparing the coefficients of $x^{i}$ for $i\ge 0$ between \eqref{eq:Z0assoc} and \eqref{eq:Z0Zhu} we obtain
\begin{proposition}[Quasiprimary Genus Zero Zhu Recursion II]
\label{prop:GenusZeroZhu3}
Let $u$ be quasiprimary of weight $N\ge 1$. Then for all $i\ge 0$ we have
\begin{align*}
	\Zzero(u(-i-1)v_{1},y_{1};\ldots ) 
	=&\sum_{j\ge 0}\E^{j}_{i}(y_{1})\Zzero(u(j)v_{1},y_{1};\ldots )
	\\
	&+\sum_{k=2}^{n}\sum_{j\ge 0}\pi_{N}^{(i,j)}(y_{1},y_{k})\Zzero(\ldots ;u(j)v_{k},y_{k};\ldots).
\end{align*}
\end{proposition}
\begin{remark}
\label{rem:Zhu genus zero II}
Using locality, we can consider the action of $u(-i-1)$ on any $v_{k}$ for $k = 1, \ldots, n$ by adjusting the order of the $y_{1}$, $y_{k}$ arguments  accordingly. A similar recursion formula for a general quasiprimary descendant can also be found.
\end{remark}
Consider formal differential forms~\eqref{eq:BFForm} with  $\Pi_{N}(x,y):=\pi_{N}(x,y)dx^{N}dy^{1-N}$ and using~\eqref{eq:vjVm} then Proposition~\ref{prop:ZeroWard} and  ~\ref{prop:GenusZeroZhu} are  equivalent to
\begin{proposition}\label{prop:GenusZeroZhuForms}
For $u$ quasiprimary of weight $N\ge 1$ we have
\begin{align}
	\label{eq:F0ZeroSum}
	&\sum_{k=1}^n \sum_{\ell\in\calL}p^{(\ell)}(y_{k})
	\Fzero(\ldots;u(\ell)v_{k},y_{k};\ldots)dy_{k}^{\ell+1-N}
	= 0 \mbox{ for all }p\in \Poly_{2N-2},
	\\
	&\Fzero(u,x;\bm{v,y})=\sum_{k=1}^{n}\sum_{j\ge 0}\Pi_{N}^{(0,j)}(x,y_{k})\Fzero(\ldots;u(j)v_{k},y_{k};\ldots)\,dy_{k}^{j}
	\label{eq:GenusZeroZhuForms}.
\end{align}
\end{proposition}
\begin{remark}
\eqref{eq:F0ZeroSum} and~\eqref{eq:GenusZeroZhuForms} are formal versions of Proposition~\ref{prop:HNzero} (concerning genus zero meromorphic forms in $x$) with formal residue
\begin{align*}
	\Res_{x-y_{k}}\left(x-y_{k}\right)^{j}\Fzero(u,x;\bm{v,y})
	=	\Fzero(\ldots;u(j)v_{k},y_{k};\ldots)dy_{k}^{j+1-N},
\end{align*}
by associativity and locality. 
\end{remark}

\section{Genus $g$ Zhu Recursion}
\label{sec:gZhu}
In this section  we introduce formal genus $g$  partition and $n$-point correlation functions for a simple VOA $V$ of strong CFT type with $V$ isomorphic to the contragredient module $V'$. These functions are formally associated to a genus $g$ Riemann surface $\Sg$ in the Schottky  scheme   of \S\ref{sec:Genus g}. The approach taken is a generalisation of the genus two sewing schemes of~\cite{MT2,MT3} and genus two Zhu recursion of~\cite{GT}.
 We  describe a genus $g$ Zhu recursion formula 
with universal  coefficients given by derivatives of $\Psi_{N}(x,y)$ and holomorphic $N$-forms $\Theta_{N,a}^{\ell}(x)$.  The recursion formula is a formal version of  Proposition~\ref{prop:GNexp} concerning the expansion of a genus $g$ meromorphic $N$-form.
This is a generalisation of the genus zero situation of~\S\ref{sec:VOAZhu0} and of that at genus one with elliptic Weierstrass function coefficients~\cite{Z1}. 
We note that the genus $g$ objects in this section are notated without any genus $g$ label except in the cases $g=0$ or $1$.
\subsection{Genus $g$ formal partition and $n$-point functions}\label{sec:Genusgnpt}
We first note an important lemma which we exploit below. Recall that  the Li-Z metric $\langle\cdot, \cdot\rangle_{\rho}$ of~\eqref{eq:LiZinvar}
 is invertible and that $\langle u,v \rangle_{\rho}=0$ for  $\wt(u)\neq \wt(v)$ for homogeneous $u,v$.  
Let $\{ b\}$ be a homogeneous basis for  $V$ with Li-Z dual basis $\{\bbar \}$. 
\begin{lemma}\label{lem:adjoint}
For $u$ quasiprimary of weight $N$ we have
\begin{align*}
	\sum_{b\,\in V_n} \left(u(k)b\right)\otimes \bbar =\sum_{b\,\in V_{n+N-k-1}} b\otimes \left(u^\dagger(k)\bbar \right),
\end{align*}
summing over any basis for the given homogeneous spaces.
\end{lemma}
\begin{proof} From~\eqref{eq:vjVm} we have $\wt(u(k)b)=n+N-k+1$ for $b\in V_n$ so that
\begin{align*}
	u(k)b =\sum_{c\,\in V_{n+N-k-1}}\langle \cbar, u(k)b\rangle_{\rho}\, c 
	= \sum_{c\,\in V_{n+N-k-1}}\langle u^\dagger(k)\cbar, b\rangle_{\rho}\, c.
\end{align*}
But $\wt(u^\dagger(k)\cbar )=n$ for $c\in V_{n+N-k-1}$ from~\eqref{eq:udagger} so that  
$
u^\dagger(k)\cbar =\sum_{b\,\in V_n}\langle u^\dagger(k)\cbar, b\rangle_{\rho}\, \bbar
$.
Hence 
\begin{align*}
	\sum_{b\,\in V_n}(u(k)b) \otimes \bbar=\sum_{c\,\in V_{n+N-k-1}}c\otimes\sum_{b\,\in V_n}\langle u^\dagger(k)\cbar, b\rangle_{\rho}\, \bbar=\sum_{c\,\in V_{n+N-k-1}} c\otimes \left(u^\dagger(k)\cbar \right).
\end{align*}
\end{proof}
For each $a\in\Ip$, let $\{b_{a}\}$  denote a homogeneous  $V$-basis and let $\{\bbar _{a}\}$ be the  dual basis with respect to the Li-Z metric $\langle \cdot,\cdot\rangle$  i.e. with $\rho=1$.  Define
\begin{align}
\label{eq:bbar}
b_{-a}:=\rho_{a}^{\wt(b_{a})}\bbar _{a},\quad a\in\Ip,
\end{align}
for formal $\rho_{a}$ (later  identified with a Schottky sewing parameter). 
Then $\{b_{-a}\}$ is a dual basis with respect to the Li-Z metric $\langle \cdot,\cdot\rangle_{\rho_{a}}$ with adjoint (cf.~\eqref{eq:udagger})
\begin{align}\label{eq:RhoAdjoint}
u^{\dagger}_{\rho_{a}}(m)=(-1)^{N}\rho_{a}^{m+1-N}u(2N-2-m),
\end{align}
for $u$ quasiprimary  of weight $N$.
Let $\bm{b}_{+}=b_{1}\otimes\ldots \otimes b_{g}$ denote an element of a $V^{\otimes g}$-basis. Let $w_{a}$ for $a\in\I$ be $2g$ formal variables (later identified with the canonical Schottky parameters).
Consider the genus zero $2g$-point rational function
\begin{align*}
\Zzero(\bm{b,w})=&\Zzero(b_{1},w_{1};b_{-1},w_{-1};\ldots;b_{g},w_{g};b_{-g},w_{-g})
\\
=&\Zzero(b_{1},w_{1};\bbar_{1},w_{-1};\ldots;b_{g},w_{g};\bbar_{g},w_{-g})
\prod_{a\in\Ip}\rho_{a}^{\wt(b_{a})},
\end{align*}
for  
$\bm{b,w}=b_{1},w_{1},b_{-1},w_{-1},\ldots,b_{g},w_{g},b_{-g},w_{-g}$. 
Define the genus $g$ partition function by
\begin{align}\label{eq:Zg}
\Zg_{V}:=\Zg_{V}(\bm{w,\rho})
=\sum_{\bm{b}_{+}}\Zzero(\bm{b,w}),
\end{align}
for $\bm{w,\rho}=w_{1},w_{-1},\rho_{1},\ldots,w_{g},w_{-g},\rho_{g}$  and 
where the sum is over any basis $\{\bm{b}_{+}\}$ of $V^{\otimes g}$.
This definition is motivated by the sewing relation~\eqref{eq:SchottkySewing2} and ideas in~\cite{TW1, MT2,MT3}. This is similar to the sewing analysis  employed in \cite{Z2, C, DGT2,G}.
As noted already, we suppress the genus superscript label $(g)$ except for genus zero and one. 
\begin{remark}\label{rem:Zg_rho_fact} 
$\Zg_{V}$ depends on  $\rho_{a}$ via the dual vectors $\bm{b}_{-}=b_{-1}\otimes\ldots \otimes b_{-g}$ as in~\eqref{eq:bbar}. In particular, setting $\rho_{a}=0$ for some $a\in\Ip$, $\Zg_{V}$ then degenerates to a genus $g-1$ partition function. Furthermore, the genus $g$ partition function for the tensor product $V_{1}\otimes V_{2}$ of two VOAs $V_{1}$ and $V_{2}$ is  $\Zg_{V_{1}\otimes V_{2}}=\Zg_{V_{1}} \Zg_{V_{2}}$.
\end{remark}
We define the genus $g$ formal $n$-point function for $n$ vectors $v_{1},\ldots,v_{n}\in V$ inserted at $y_{1},\ldots,y_{n}$ by
\begin{align}\label{eq:GenusgnPoint}
\Zg_{V}(\bm{v,y}):=\Zg_{V}(\bm{v,y};\bm{w,\rho})
=
\sum_{\bm{b}_{+}}\Zzero(\bm{v,y};\bm{b,w}),
\end{align}
for rational genus zero $(n+2g)$-point functions 
\begin{align*}
\Zzero(\bm{v,y};\bm{b,w})=\Zzero(v_{1},y_{1};\ldots;v_{n},y_{n};b_{1},w_{1};\ldots;b_{-g},w_{-g}).
\end{align*}
 We also define the corresponding genus $g$ formal $n$-point correlation differential form
\begin{align*}
\F_{V}(\bm{v,y}):=Z_{V}(\bm{v,y})\bm{dy^{\wt(v)}}.
\end{align*}
\begin{remark}\label{rem:Gui_convegence}
If $V$ is also $C_{2}$-cofinite then Gui's Theorem~13.1 \cite{G} implies that $\F_{V}(\bm{v,y})$ is absolutely and locally uniformly convergent on the sewing domain since it is a multiple sewing of correlation functions associated with genus zero conformal blocks. This important result proves Conjecture~8.1 of \cite{Z1}. However, here we treat $\F_{V}(\bm{v,y})$ formally since Zhu recursion does not require $V$ to be $C_{2}$-cofinite.
\end{remark}
\subsection{Comparison to genus one Zhu trace functions} \label{subsec:ZhuTrace}
We  now show  that we obtain Zhu's definition~\cite{Z1} of a genus one $n$-point function (including the partition function\footnote{
An alternative expression for the partition function based on the Catalan series  is described in~\cite{MT3}.})
from our genus $g$ definition~\eqref{eq:GenusgnPoint} when $g=1$.  
In this case there are three Schottky parameters: $W_{\pm 1}$ and $q=q_{1}$ (where $|q|<1$)
with corresponding canonical coordinates $w_{\pm 1}$ and $\rho=\rho_{1}$. 
The Schottky group is $\Gamma=\langle \gamma \rangle $ where 
\begin{align}
\label{eq:gam_sigma}
\gamma=\sigma^{-1}
\begin{pmatrix}
	q^{1/2} &0\\
	0 &q^{-1/2}
\end{pmatrix} 
\sigma,\quad 
\sigma:=W^{-1/2}  \begin{pmatrix}1 & -W_{-1}\\1 & -W_{1}\end{pmatrix},
\end{align}
with $W:=W_{-1}-W_{1}$ from~\eqref{eq:gammaa}.
For $z\in\D$ (the Schottky fundamental domain)  and $Z:=\sigma z=\frac{z-W_{-1}}{z-W_{1}}$ then $\zeta: =\log Z$ lies in a fundamental parallelogram with periods $\tpi$ and $\tpi \tau$  in the standard description of a torus with modular parameter $q=e^{\tpi \tau}$.

The genus one $n$-point function  for $v_{1},\ldots,v_{n}$ is defined by Zhu as the trace~\cite{Z1}
\begin{align*}
Z^{\Zhu}_{V}(\bm{v,\zeta};q):=&
\Tr_{V}
\left(\, Y\left(e^{ \wt(v_{1})\zeta_{1}}v_{1},e^{\zeta_{1}}\right)\ldots  
Y\left(e^{\wt(v_{n})\zeta_{n} }v_{n},e^{\zeta_{n}}\right)q^{L(0)}
\right),
\end{align*}
omitting the standard  $q^{-C/24}$ factor (introduced to enhance the $\SL_{2}(\Z)$ modular properties of trace functions).
Comparing to our definition~\eqref{eq:GenusgnPoint} we find
\begin{theorem}\label{theor:ZZhu}
Let $v_{1},\ldots,v_{n}\in V$ be quasiprimary vectors inserted at $z_{1},\ldots, z_{n}$. Then 
\[
\F_{V}^{(1)}(\bm{v,z})=Z^{\Zhu}_{V}(\bm{v,\zeta};q)\bm{d\zeta^{\wt(v)}},
\]
for  $\zeta_{i} =\log Z_{i}$ with $Z_{i}=\sigma z_{i}=\frac{z_{i}-W_{-1}}{z_{i}-W_{1}}$ for $i=1,\ldots, n$.
\end{theorem}
\begin{remark}
One may obtain the Zhu trace $n$-point functions for all quasiprimary descendants by means of the identity 
\[
\del_{\zeta} Y\left(e^{\zeta L(0)}v,e^{\zeta}\right)
=Y\left(e^{\zeta L(0)}L[-1]v,e^{\zeta}\right),\quad v\in V,
\] 
for cylindrical Virasoro mode $L[-1]=L(0)+L(-1)$ \cite{Z1}.
\end{remark}
\begin{proof}[Proof of Theorem~\ref{theor:ZZhu}]
Since the result is trivially true for $q=0$ we assume $q\neq 0$.
Consider  the genus one formal differential form
\begin{align*}
	\F_{V}^{(1)}(\bm{v,z})
	=&\sum_{b_{1}\in V}\left\langle 
	\vac,Y\left(b_{1},w_{1}\right)\bm{ Y\left(v,z\right)}
	Y\left(b_{-1},w_{-1}\right)\vac\right\rangle_{\rho}\bm{dx^{\wt(v)}}
	\\
	=&\sum_{n\ge 0}\rho^{n}\sum_{b\in V_{n}}\left\langle 
	\vac,Y\left(b,w_{1}\right)\bm{Y(v,z)}Y\left(\bbar,w_{-1}\right)\vac\right\rangle \bm{dx^{\wt(v)}},
\end{align*}
recalling~\eqref{eq:bbar}.
For $\sigma$ of~\eqref{eq:gam_sigma} we note that  
\begin{align*}
	&\sigma w_{\pm 1}=q^{\mp 1},
	\quad c(cw_{\pm 1}+d)=\frac{1}{1-q^{\mp 1}},
	\quad (cw_{-1}+d)(cw_{1}+d)=-\frac{qW}{(1-q)^{2}},
\end{align*}
where $c=W^{-1/2}$ and $d=   -W^{-1/2}W_{1}$.
Applying Proposition~\ref{prop:ZMobius} for $\sigma $ with quasiprimary $v_{i}$ states  and  recalling from~\eqref{eq:rhoa} that $\rho =-q W^{2}/(1-q)^{2}$  we obtain
\begin{align*}
	\F_{V}^{(1)}(\bm{v,z})
	&=\sum_{n\ge 0}
	Q^{n}
	\sum_{b\in V_{n}}
	\left\langle 
	\vac,Y\left(e^{\frac{q}{1-q}L(1)}b,q^{-1}\right)
	\bm{ Y\left(v,Z\right)}Y\left(e^{-\frac{1}{1-q}L(1)}\bbar,q\right)\vac\right\rangle 
	\bm{dX^{\wt(v)}},
\end{align*}
for $Q:=-q^{-1}(1-q)^2$ and $Z_{i}:=\sigma z_{i}$.
Using~\eqref{eq:Lm1L1com1} and~\eqref{eq:Yadj} we find
\begin{align*}
	& \left\langle 
	\vac,\;Y\left(e^{\frac{q}{1-q}L(1)}b,q^{-1}\right)
	\bm{ Y\left(v,Z\right)}Y\left(e^{-\frac{1}{1-q}L(1)}\bbar,q\right)\vac\right\rangle 
	\\
	&=
	\left\langle 
	Y\left((-q^2)^{L(0)}e^{\frac{q^2}{1-q}L(1)}b,q\right)\vac,\;
	\bm{ Y\left(v,Z\right)}Y\left(e^{-\frac{1}{1-q}L(1)}\bbar,q\right)\vac\right\rangle 
	\\
	&=
	\left\langle 
	e^{qL(-1)}(-q^2)^{L(0)}e^{\frac{q^2}{1-q}L(1)}b,\;
	\bm{ Y\left(v,Z\right)}e^{qL(-1)}e^{-\frac{1}{1-q}L(1)}\bbar\right\rangle 
	\\
	&=
	\left\langle 
	b,\;
	e^{\frac{q^2}{1-q}L(-1)}(-q^2)^{L(0)}e^{qL(1)}
	\bm{ Y\left(v,Z\right)}e^{qL(-1)}e^{-\frac{1}{1-q}L(1)}\bbar\right\rangle.
\end{align*}
Thus it follows that 
\begin{align*}
	\F_{V}^{(1)}(\bm{v,z})
	&=
	\Tr_{V}\left(
	e^{\frac{q^2}{1-q}L(-1)}(-q^2)^{L(0)}e^{qL(1)}\bm{Y\left(v,Z\right)}e^{qL(-1)}e^{-\frac{1}{1-q}L(1)}Q^{L(0)}\right)
	\bm{dX^{\wt(v)}}
	\\
	&= \Tr_{V}\left(\bm{Y\left(v,Z\right)}
	e^{qL(-1)}e^{-\frac{1}{1-q}L(1)}
	\left(-q^{-1}(1-q)^2\right)^{L(0)}
	e^{\frac{q^2}{1-q}L(-1)}(-q^2)^{L(0)}e^{qL(1)}\right)\bm{dX^{\wt(v)}}
	\\
	&= \Tr_{V}\left(\bm{Y\left(v,Z\right)}
	e^{qL(-1)}e^{-\frac{1}{1-q}L(1)}
	e^{ -q(1-q)L(-1)}e^{\frac{1}{(1-q)^2}L(1)}
	(1-q)^{2L(0)}q^{L(0)}\right)\bm{dX^{\wt(v)}},
\end{align*}
using~\eqref{eq:L0Lpm1conj}.  
But from~\eqref{eq:Lm1L1com2} we find that
\[
e^{qL(-1)}e^{-\frac{1}{1-q}L(1)}e^{ -q(1-q)L(-1)}e^{\frac{1}{(1-q)^2}L(1)}\left(1-q\right)^{2L(0)}=\Id_{V}.
\]
Thus we have shown that
\begin{align*}
	\F_{V}^{(1)}(\bm{v,z})
	&=
	\Tr_{V}\left( \bm{ Y\left(v,Z\right)}q^{L(0)}\right) \bm{dX^{\wt(v)}}
	=Z^{\Zhu}_{V}(\bm{v,\zeta};q)\bm{d\zeta^{\wt(v)}},
\end{align*}
since $\bm{dX^{\wt(v)}}=\prod_{k=1}^{n}e^{\wt(v_{k})\zeta_{k}}d\zeta_{k}^{\wt(v_{k})}$. 
\end{proof}

\subsection{Genus $g$ formal Schottky M\"obius invariance} 
For Schottky parameters $W_{a}$ of~\eqref{eq:SchottkySewing} and
$p\in \Poly_{2}$ we define the M\"obius generator
\begin{align*}
	\D^{p}:=\sum_{a\in\I}p(W_{a})\del_{W_{a}}.
\end{align*}
%
%
This can be written in terms of the $w_{a},\rho_{a}$ parameters of ~\eqref{eq:wa} and~\eqref{eq:rhoa} as
\begin{align}\label{eq:Dpdiff}
\D^{p}=\sum_{a\in\I}\left(  
p(w_{a})\del_{w_{a}}+p^{(1)}(w_{a})\rho_{a}\del_{\rho_{a}}
+p^{(2)}(w_{a}) \rho_{a}\del_{w_{-a}}
\right),
\end{align}
where $p^{(i)}(x)=\partial_{x}^{(i)}p(x)$.  The genus $g$ partition function is formally M\"obius invariant:
\begin{proposition}\label{prop:DpZ}
$\D^{p}\Zg_{V}=0$ for all $p\in \Poly_{2}$.
\end{proposition}
\begin{proof}
Proposition~\ref{prop:ZeroWard} with $u=\omega$ and basis vectors $b_{a}$ inserted at $w_{a}$ implies
\begin{align}\label{eq:WardZ}
	\sum_{\bm{b}_{+}}\sum_{a\in\I}\sum_{\ell=0}^{2}
	p^{(\ell)}(w_{a})	\Zzero(\ldots;L(\ell-1)b_{a},w_{a};\ldots)=0.
\end{align}
But for each $a\in\I$ we have
\begin{align}
	\label{eq:ZL1}
	\sum_{\bm{b}_{+}}\Zzero(\ldots;L(-1)b_{a},w_{a};\ldots)&=\del_{w_{a}}\Zg_{V},
	\\
	\label{eq:ZL2}
	\sum_{\bm{b}_{+}}\Zzero(\ldots;L(0)b_{a},w_{a};\ldots)&=\wt(b_{a})\Zg_{V}
	=\rho_{a}\del_{\rho_{a}}\Zg_{V},
	\\
	\label{eq:ZL3}
	\sum_{\bm{b}_{+}}\Zzero(\ldots;L(1)b_{a},w_{a};\ldots)
	&
	=
	\sum_{\bm{b}_{+}}\Zzero(\ldots;\rho_{a}L(-1)b_{-a},w_{-a};\ldots)
	=\rho_{a}\del_{w_{-a}}\Zg_{V},
\end{align}
recalling~\eqref{eq:bbar} and using Lemma~\ref{lem:adjoint} with $L^{\dagger}_{\rho_{a}}(1)=\rho_{a}L(-1)$ from~\eqref{eq:RhoAdjoint}. Substituting~\eqref{eq:ZL1}--\eqref{eq:ZL3} into~\eqref{eq:WardZ} we obtain the stated result on noting that $p^{(2)}$ is a constant.
\end{proof}
Proposition~\ref{prop:DpZ} can be generalised to an $n$-point formal form $\Fg_{V}(\bm{v,y})$ for $n$ vectors $v_{k}\in V$ of weight $\wt(v_{k})$. For $p\in \Poly_{2}$ we define
\begin{align}\label{eq:Dpy}
\D^{p}_{\bm{y}}:=\D^{p}+\sum_{k=1}^{n}\left(p(y_{k})\del_{y_{k}}+\wt(v_{k}) p^{(1)}(y_{k})\right),
\end{align}
for $\D^{p}$ of~\eqref{eq:Dpdiff}. 
Then a similar analysis implies that
\begin{proposition}\label{prop:DpyF}
$
\D^{p}_{\bm{y}}\Fg_{V}(\bm{v,y})+
\sum_{k=1}^{n}p ^{(2)}(y_{k})\Fg_{V}(\ldots ;L(1)v_{k},y_{k};\ldots)=0$
for all $p\in \Poly_{2}$.
\end{proposition}
\begin{corollary}\label{cor:DpyZ}
$\D^{p}_{\bm{y}}\Fg_{V}(\bm{v,y})=0$ for quasiprimary states $v_{1},\ldots,v_{n}$. 
\end{corollary}
Corollary~\ref{cor:DpyZ} is a formal version of Proposition~5.3~(ii) of \cite{TW1} concerning genus $g$ meromorphic forms in $n$ variables.

\subsection{A genus $g$ Ward identity}
For quasiprimary $u\in V$ of weight $N$ and $n$ vectors $v_{k} \in V$ of weight $\wt(v_{k})$ for $k=1,\ldots,n$ we consider 
\[
\F_{V}(u,x;\bm{v,y})=Z_{V}(u,x;\bm{v,y})dx^{N}\bm{dy^{\wt(v)}}.
\] 
Define the formal residue \cite{K,FHL,LL}
\begin{align}
	\label{eq:ResF}
	\Res_{a}^{\ell}\Fg_{V}:=\Res_{a}^{\ell}\Fg_{V}(u;\bm{v,y}):=&
	\Res_{x-w_{a}}\left(x-w_{a}\right)^{\ell}\Fg_{V}(u,x;\bm{v,y})
	\\
	\notag
	=&	\sum_{\bm{b}_{+}}\Zzero(\ldots;u(\ell)b_{a},w_{a};\ldots)\bm{dy^{\wt(v)}},
\end{align}
for $\ell\in\calL$ and $a\in\I$. Equation~\eqref{eq:ResF} follows from VOA locality and  associativity. Equation~\eqref{eq:F0ZeroSum} implies a general Ward identity for genus $g$ correlation functions that formally follows the structure of Proposition~\ref{prop:GNexp}~(ii) as follows:
\begin{proposition}\label{prop:GenusgWard}
Let $u$ be quasiprimary of weight $N$ and let $P(z)=p(z)dz^{1-N}$ for $p\in\mathfrak{P}_{2N-2}$. Then for  $p_{a}^{\ell}$ of~\eqref{eq:pal} we find
\begin{align*}
	-\sum_{a\in\Ip}	\sum_{\ell\in\calL}p_{a}^{\ell}
	\Res_{a}^{\ell}\Fg_{V}(u;\bm{v,y})
	+\sum_{k=1}^{n}p^{(\ell)}(y_{k})\sum_{\ell\in\calL}
	\Fg_{V}(\ldots;u(\ell)v_{k},y_{k};\ldots)dy_{k}^{\ell+1-N}=0.
\end{align*} 
\end{proposition}
\begin{proof}
Proposition~\ref{prop:ZeroWard} applied to $\Zzero(\bm{v,y};\bm{b,w})$ of~\eqref{eq:GenusgnPoint} implies 
\begin{align*}
0=&\sum_{\ell\in\calL}
\sum_{a\in\Ip}\left(
p^{(\ell)} (w_{a})\Zzero(\ldots;u(\ell)b_{a},w_{a};\ldots)
+
p^{(\ell)}(w_{-a}) \Zzero(\ldots;u(\ell)b_{-a},w_{-a};\ldots)
\right)
\\
&
+\sum_{\ell\in\calL}
\sum_{k=1}^{n}
p^{(\ell)} (y_{k})\Zzero(\ldots;u(\ell)v_{k},y_{k};\ldots).
\end{align*}
By Lemma~\ref{lem:adjoint} and~\eqref{eq:RhoAdjoint} we find summing over any $V$-basis $\{b_{a}\}$ that
\begin{align}
	\label{eq:Zuudag}
\sum_{b_{a}\in V}	\Zzero(\ldots;u(\ell)b_{-a},w_{-a};\ldots)
=&(-1)^{N}\rho_{a}^{\ell+1-N}\sum_{b_{a}\in V}	\Zzero(\ldots;u(2N-2-\ell)b_{a},w_{a};\ldots).
\end{align}
On relabeling, summing over $\bm{b_{+}}$ bases and multiplying by $\bm{dy^{\wt(v)}}$ we obtain
\begin{align*}
0=&\sum_{\ell\in\calL}\sum_{a\in\Ip}
\left(
p^{(\ell)} (w_{a})+(-1)^{N} \rho_{a}^{N-\ell-1}p^{(2N-2-\ell)}(w_{-a})
\right)
\Res_{a}^{\ell}\F_{V}(u;\bm{v,y})
\\
&
+\sum_{\ell\in\calL}
\sum_{k=1}^{n}
p^{(\ell)} (y_{k})\F_{V}(\ldots;u(\ell)v_{k},y_{k};\ldots)dy_{k}^{\ell+1-N},
\end{align*}
since $\wt(u(\ell)v_{k})=\wt(v_{k})+N-\ell-1$. The result follows from~\eqref{eq:pal}.
\end{proof}

\subsection{Genus $g$ Zhu recursion}
We now derive a genus $g$ Zhu reduction formula generalising Propositions~\ref{prop:GenusZeroZhu} and~\ref{prop:GenusZeroZhuForms}. We define $\Pi_{N}(x,y)=\pi_{N}(x,y)dx^{N}dy^{1-N}$ for $\pi_{N}(x,y)$ of~\eqref{eq:piNdef} to be that determined for $N\ge 2$ by $\Psi_{N}(x,y)$  of~\eqref{eq:PipiN} with $f_{\ell}(x)$ of~\eqref{eq:flx} and for $n=1$ by
$\Pi_{1}(x,y)$ of \eqref{eq:Pi1def} and $\Psi_{1}(x,y)$ of \eqref{eq:Psi1def}.  These choices guarantee the convergence of the coefficient functions appearing in the genus $g$ Zhu reduction in terms of derivatives of $\Psi_{N}(x,y_{k})$  and the $N$-form spanning set $\{ \Theta_{N,a}^{\ell}(x)\}$. In particular, we  exploit the sewing formulas~\eqref{eq:Psisew} for $\Psi_{N}(x,y)$ and~\eqref{eq:Thetaexp} for $\Theta_{N,a}^{\ell}(x)$. 
\begin{theorem}[Quasiprimary Genus $g$ Zhu Recursion]\label{theor:ZhuGenusg}
	Let $V$ be a simple VOA of strong CFT-type with $V$ isomorphic  to $V'$.
The genus $g$ correlation  differential form for quasiprimary $u$ of weight $N\ge 1$ inserted at $x$ and  $v_{1},\ldots,v_{n}\in V$ inserted at $y_{1},\ldots,y_{n} $ respectively, satisfies the recursive identity
\begin{align}
	\Fg_{V}(u,x;\bm{v,y})=&\sum_{a\in\Ip}\sum_{\ell\in \calL} \Theta_{N,a}^{\ell}(x)\Res_{a}^{\ell}\Fg_{V}(u;\bm{v,y})
	\label{eq:ZhuGenusg}
	\\
	\notag
	&
	+\sum_{k=1}^{n}\sum_{j\ge 0}\Psi_{N}^{(0,j)}(x,y_{k})\Fg_{V}(\ldots;u(j)v_{k},y_{k};\ldots)\,dy_{k}^{j}.
\end{align}
\end{theorem}
\begin{remark}
\label{rem:MainTheorem}\leavevmode
\begin{enumerate}
	\item [(i)] The $\Theta_{N,a}^{\ell}(x),\Psi_{N}(x,y )$ terms depend on $N=\wt(u)$ but  are otherwise independent of the VOA $V$ i.e. they are  analogues of the genus zero $\Pi_{N}(x,y )$ coefficients and the genus one Weierstrass $P_{1}$ coefficients  found in~\cite{Z1}. 
	\item [(ii)] For $N\ge 2$, \eqref{eq:ZhuGenusg} is independent of the choice of limit points $\{A_{\ell}\}$ in $\widehat{\Psi}_{N}(x,y)$ of~\eqref{eq:Psitilde} and the holomorphic $N$-form spanning set $\{ \widehat{\Theta}^{\ell}_{N,a}(x) \}$ of~\eqref{eq:Thetatilde}. Likewise, for $N=1$, \eqref{eq:ZhuGenusg} is independent of the choice of $A_{0}$ in \eqref{eq:Psi1tilde}.	%
	\item [(iii)]  The $\Psi_{N}^{(0,j)}$ coefficient can be expressed as the formal residue
	\begin{align*}
		\Res_{x-y_{k}}\left(x-y_{k}\right)^{j}\Fg_{V}(u,x;\bm{v,y})
		=	\Fg_{V}(\ldots;u(j)v_{k},y_{k};\ldots)dy_{k}^{j+1-N},
	\end{align*}
	by associativity and locality. Thus the formal expansion~\eqref{eq:ZhuGenusg} precisely matches the meromorphic $N$-form expansion of Proposition~\ref{prop:GNexp}.
	\item[(iv)] We may extend the definition of $\Pi_{N}$ in  \eqref{eq:piNdef} to $g=1$ by choosing distinct $A_{\ell}\in\Chat$ so that $\Psi_{N}(x,y)$ is meromorphic with simple poles at $x=y$ and $A_{\ell}$ for all $\ell\in\calL$. 
	For $\sigma$ of \eqref{eq:gam_sigma} we define $X:=\sigma x$, $Y:=\sigma y$ and $ \Ahat_{\ell}:=\sigma A_{\ell}$. 
	Then \eqref{eq:PiN_Mobius} implies 
	\begin{align*}
		\Psi_{N}(x,y)
		=\sum_{k\in\Z}\frac{q^{kN}}{q^{k}X-Y}
		\prod_{\ell\in\calL}\frac{Y-\Ahat_{\ell}}{X-\Ahat_{\ell}}dX^{N}dY^{1-N}.
	\end{align*}
	This expression is absolutely convergent for $|q|<|X|<|Y|,|\Ahat_{\ell}|$ and is elliptic in $\xi:=\log X$. The simple pole structure implies that with $P_{1}$ of \eqref{eq:P1def} we find  $
		\Psi_{N}(x,y)=\left(P_{1}(\xi-\eta)-\sum_{\ell\in \calL}P_{1}(\xi-\alpha_{\ell}) \widehat{\calQ}_{\ell}(Y)\right)dX^{N}dY^{1-N}$
with $\xi:=\log X$, $\eta:=\log Y$, $\alpha_{\ell}:=\log \Ahat_{\ell}$ and 
$\widehat{\calQ}_{\ell}(Y):=\prod_{j\neq \ell}\frac{Y-\Ahat_{j}}{\Ahat_{\ell}-\Ahat_{j}}\in \Poly_{2N-2}(Y)$. 
Hence $\Psi_{N}(x,\gamma y)=\Psi_{N}(x,y)+dX^{N}dY^{1-N}$ for $\gamma$ of \eqref{eq:gam_sigma} and using \eqref{eq:P1_periods}. The total contribution of the $P_{1}(\xi-\alpha_{\ell}) \widehat{\calQ}_{\ell}(Y)dY^{1-N}$ terms to \eqref{eq:ZhuGenusg} is zero due to Proposition~\ref{prop:GenusgWard} and because 
$\partial_{y}^{(j)}\left(\widehat{\calQ}_{\ell}(Y)dY^{1-N}\right)=0$ for $j>2N-2$. Thus $\Psi_{N}^{(0,j)}(x,y)$ may be replaced by $\partial^{(j)}_{y}\left(P_{1}(\xi-\eta)dX^{N}dY^{1-N}\right)$ in \eqref{eq:ZhuGenusg}. This Weierstrass expansion matches that of Zhu  once the modes $u(j)$ in \eqref{eq:ZhuGenusg} defined on the sphere are expressed in terms of the more  natural cylindrical ``square bracket" modes  \cite{Z1}.
\end{enumerate}
\end{remark}
Much as for Corollary~\ref{cor:GenGenusZeroZhu} we may generalise Theorem~\ref{theor:ZhuGenusg} to find:
\begin{corollary}[General Genus $g$ Zhu Recursion]\label{cor:GenGenusgZhu}
The genus $g$ formal $n$-point  differential  for a quasiprimary  descendant $L^{(i)}(-1)u$ for $u$ of weight $N$ inserted at $x $ and general vectors $v_{1},\ldots,v_{n}$ inserted at $y_{1},\ldots,y_{n} $ respectively, satisfies the recursive identity
\begin{align*}
	\Fg_{V}(L^{(i)}(-1)u,x;\bm{v,y})=&\sum_{a=1}^{g}\sum_{\ell\in \calL} (\Theta_{N,a}^{\ell})^{(i)}(x)\Res_{a}^{\ell}\Fg_{V}(u;\bm{v,y})
	\notag
	\\
	&+\sum_{k=1}^{n}\sum_{j\ge 0}\Psi_{N}^{(i,j)}(x,y_{k})\Fg_{V}(\ldots;u(j)v_{k},y_{k};\ldots)\,dx^{i}dy_{k}^{j}.
\end{align*}
\end{corollary}
A genus $g$ version of Proposition~\ref{prop:GenusZeroZhu3} can also be developed.  

\begin{proof}[Proof of Theorem~\ref{theor:ZhuGenusg}]
We firstly apply Proposition~\ref{prop:GenusZeroZhu} to  the summands of~\eqref{eq:GenusgnPoint} to find
\begin{align}\label{eq:GenusgZhu1}
	\F_{V}(u,x;\bm{v,y})\notag
	&=\sum_{a\in\I}\sum_{j\ge 0}\pi_{N}^{(0,j)}(x,w_{a})dx^{N}\sum_{\bm{b}_{+}}\Zzero(\ldots ;u(j)b_{a},w_{a};\ldots)\bm{dy^{\wt(v)}}
	\notag
	\\
	&\quad+\sum_{k=1}^{n}\sum_{j\ge 0}\pi_{N}^{(0,j)}(x,y_{k})dx^{N}\sum_{\bm{b}_{+}}\Zzero(\ldots ;u(j)v_{k},y_{k};\ldots)\bm{dy^{\wt(v)}}.
\end{align}
Define a column vector $X=(X_{a}^{m})$ indexed by $ m\ge 0$ and $ a\in\I$  with components 
\begin{align*}
	X_{a}^{m}:=\rho_{a}^{-\frac{m}{2}}\sum_{\bm{b}_{+}}\Zzero(\ldots;u(m)b_{a},w_{a};\ldots)\bm{dy^{\wt(v)}}.
\end{align*}
Thus~\eqref{eq:GenusgZhu1} can be rewritten as
\begin{align}\label{eq:ZhuGenusZero}
	\Fg_{V}(u,x;\bm{v,y})=L(x)X+\sum_{k=1}^{n}\sum_{j\ge 0}\Pi_{N}^{(0,j)}(x,y_{k})\Fg_{V}(\ldots;u(j)v_{k},y_{k};\ldots)\,dy_{k}^{j},
\end{align}
for row vector $L(x)=(L_{a}(x;m))$ of~\eqref{eq:Ldef}.

We next develop a recursive formula for $X$ following the genus two  approach in ~\cite{GT}.
We apply Lemma~\ref{lem:adjoint} to the summands of $X_{a}^{m}$ 
with adjoint~\eqref{eq:RhoAdjoint} to find
\begin{align}\label{XaAdjoint}
	X_{a}^{m}=(-1)^{N}\rho_{a}^{\frac{m}{2}-N+1}\sum_{\bm{b}_{+}}\Zzero(\ldots;u(2N-2-m)b_{-a},w_{-a};\ldots)\bm{dy^{\wt(v)}}.
\end{align}
In particular, for $\ell\in\calL$ we find from~\eqref{eq:Zuudag} that
\begin{align}
	X_{a}^{\ell}&=\rho_{a}^{-\frac{\ell}{2}}
	\Res_{a}^{\ell}\Fg_{V}(u;\bm{v,y})=(-1)^{N}X_{-a}^{2N-2-\ell},
	\label{eq:Xa_ell_Rel}
\end{align}
a formal version of~\eqref{eq:ResRel} for an $N$-form meromorphic in $x$. 

Let $i=m-2N+1\ge 0$ for $m\ge 2N-1$. Proposition~\ref{prop:GenusZeroZhu3},  Remark~\ref{rem:Zhu genus zero II} and~\eqref{XaAdjoint} imply
\begin{align*}
	X_{a}^{m}&=(-1)^{N}\rho_{a}^{\half(i+1)}
	\sum_{\bm{b}_{+}}\Zzero(\ldots;u(-i-1)b_{-a},w_{-a};\ldots)\bm{dy^{\wt(v)}}
	\\
	&=(-1)^{N}\rho_{a}^{\half(i+1)} \sum_{k=1}^{n}\sum_{j\ge 0}\pi_{N}^{(i,j)}(w_{-a},y_{k})
	\Zg_{V}(\ldots;u(j)v_{k},y_{k};\ldots)\bm{dy^{\wt(v)}}
	\\
	&\quad+(-1)^{N}\rho_{a}^{\half(i+1)} \sum_{\substack{b\in\I,\\b\neq -a}}\sum_{j\ge 0}\pi_{N}^{(i,j)}(w_{-a},w_{b})
	\sum_{\bm{b}_{+}}\Zzero(\ldots;u(j)b_{b},w_{b};\ldots)\bm{dy^{\wt(v)}}
	\\
	&\quad+(-1)^{N}\rho_{a}^{\half(i+1)} \sum_{j\ge 0}\E^{j}_{i}(w_{-a})
	\sum_{\bm{b}_{+}}\Zzero(\ldots;u(j)b_{-a},w_{-a};\ldots)\bm{dy^{\wt(v)}}.
\end{align*}
We can rewrite this as
\begin{align}\label{XaRelation}
	X_{a}^{m}=(G+AX)_{a}^{m-2N+1},\quad m\ge 2N-1, 
\end{align}
for  $A=(A_{ab}^{mn})$ of~\eqref{eq:Adef} and column vector $G=(G_{a}^{m})$ for $m\ge 0$, $a\in\I$ with components
\begin{align}
	\label{eq:Gdef}
	G_{a}^{m}:=\sum_{k=1}^{n}\sum_{j\ge 0}\left(R_{a}^{m}\right)^{(j)}(y_{k})\,dy_{k}^{j}\,\Fg_{V}(\ldots;u(j)v_{k},y_{k};\ldots),
\end{align}
for $R_{a}^{m}(y)$ of~\eqref{eq:Rdef}.	
With  $\Delta=(\Delta_{ab}^{mn})$ of~\eqref{eq:Deltadef} then~\eqref{XaRelation} can be written
\begin{align*}
	X_{a}^{m}=(\Delta(G+AX))_{a}^{m},\quad m\ge 2N-1.
\end{align*}
We note the identities
\begin{align}\label{DeltaDelta}
	\Delta^{T}\Delta =I,\quad  
	\Delta\Delta^{T}=I-\Pcal,
\end{align}
for identity matrix $I$ and projection matrix $\Pcal =(\Pcal _{ab}^{mn})$ with components
\begin{align*}
	\Pcal _{ab}^{mn}:=\begin{cases}
		\delta_{ab}\delta_{mn}& \mbox{ for } m,n\in\calL,\\
		0 & \mbox{ otherwise.}
	\end{cases}
\end{align*}
Define
\begin{align*}
	X^{\perp}:=\Delta^{T}X,\quad  
	X^{\Pcal}:=\Pcal X.
\end{align*}
Using~\eqref{DeltaDelta} we have
\begin{align*}
	X^{\perp}&=G+AX.
\end{align*}
But 
\begin{align*}
	X=\Pcal X+(I-\Pcal )X=X^{\Pcal }+\Delta X^{\perp},
\end{align*}
implies
\begin{align*}
	X^{\perp}&
	=G+AX^{\Pcal }+\Atilde X^{\perp},
\end{align*}
with $\Atilde=A\Delta$ of~\eqref{eq:Rtilde}. 
Formally solving for $X^{\perp}$ gives:
\begin{align*}
	X^{\perp}=\left(I-\Atilde\right)^{-1}AX^{\Pcal }+\left(I-\Atilde\right)^{-1}G,
\end{align*}
for formal inverse $\left(I-\Atilde\right)^{-1}=\sum_{k\ge 0}\Atilde^{k}$. 
Altogether we have found 
\begin{lemma}\label{XSub}
	Let $u$ be a quasiprimary vector with $\wt(u) = N\ge 1$. Then
	\begin{align*}
		X=\left(I+\Delta \left(I-\Atilde\right)^{-1}A\right)X^{\Pcal }+\Delta\left(I-\Atilde\right)^{-1}G,
	\end{align*}
	for $G$ of~\eqref{eq:Gdef}.
\end{lemma}
Lemma~\ref{XSub} implies that 
\begin{align*}
	L(x)X=&
	\sum_{\ell\in\calL}\sum_{a\in\I}
	(L(x)+\Ltilde(x)(I-\Atilde)^{-1}A )_{a}^{\ell}X_{a}^{\ell}
	\\
	&+
	\sum_{k=1}^{n}\sum_{j\ge 0}
	\Ltilde(x)(I-\Atilde)^{-1}R^{(j)}(y_{k})\,dy_{k}^{j}\,\Fg_{V}(\ldots;u(j)v_{k},y_{k};\ldots).		
\end{align*}
Therefore the preliminary formula~\eqref{eq:ZhuGenusZero} and equation~\eqref{eq:Xa_ell_Rel} imply
\begin{align}\label{GenusgZhu}
	&\Fg_{V}(u,x;\bm{v,y})=
	\sum_{a\in\I}\sum_{\ell\in\calL}T_{a}^{\ell}(x)\Res_{a}^{\ell}\F_{V}
	+\sum_{k=1}^{n}\sum_{j\ge 0}\Psi_{N}^{(0,j)} (x,y_{k})\Fg_{V}(\ldots;u(j)v_{k},y_{k};\ldots)\,dy_{k}^{j},
\end{align}
for $\Psi_{N}(x,y)$ of~\eqref{eq:Psisew} and $T_{a}^{\ell}(x)$ of~\eqref{eq:Thetaexp}. 
Finally, Proposition~\ref{prop:Thetaexp} and~\eqref{eq:Xa_ell_Rel} imply that 
\begin{align*}
	\sum_{a\in\I}\sum_{\ell\in\calL}T_{a}^{\ell}(x)\Res_{a}^{\ell}\F_{V}= \sum_{a\in \Ip}\sum_{\ell\in\calL} \Theta_{N,a}^{\ell}(x)\Res_{a}^{\ell}\F_{V},
\end{align*}
for holomorphic $N$-form $\Theta_{N,a}^{\ell}(x)$ of~\eqref{eq:Thetaexp}. Thus~\eqref{GenusgZhu} implies~\eqref{eq:ZhuGenusg}.
\end{proof}

\section{Genus $g$ Zhu Recursion with Interwiners}
\subsection{Intertwiners for irreducible ordinary $V$-modules}
\label{sec:Genus g Inter}
In this section we consider the generalisation of \S\ref{sec:Genus g} for intertwiner vertex operators associated with $V$-modules.  We assume some conditions on the dimension of the space of intertwiners that hold for VOAs satisfying the Verlinde formula \cite{H1} and for the Heisenberg abelian intertwiner algebra \cite{TZ2}. Despite the fact that genus zero intertwiner correlation functions are not rational, we find similar Zhu reduction formulas for all genera $g\ge 0$ when reducing with respect to an element of $V$.  We also obtain M\"obius invariance for genus zero correlation functions containing at most two intertwiner operators which allows us to reproduce Zhu's module trace correlation function at genus one \cite{Z1}.

Let $\calA$ be a label set for inequivalent irreducible ordinary $V$-modules  $(Y_{\alpha},\Wmod_{\alpha})$ of conformal weight $h_{\alpha}\in\C$ for $\alpha\in\calA$ \cite{FHL,DLM}. We identify $(Y_{0},\Wmod_{0})$ with $ (Y,V)$.
We assume that $\Wmod_{\alpha}$ is graded with
$\Wmod_{\alpha}=\oplus_{n\ge 0}\Wmod _{\alpha,n+h_{\alpha}}$ where $v_{\alpha}\in \Wmod _{{\alpha},n+h_{\alpha}}$ has conformal weight $n+h_{\alpha}$ for integer $n\ge 0$ and $\dim \Wmod _{{\alpha},n+h_{\alpha}}<\infty$.
For $u\in V$  we have $Y_{\alpha}(u,z)=\sum_{n\in\Z}u_{\Wmod_{\alpha}}(n)z^{-n-1}$ for modes $u_{\Wmod_{\alpha}}(n)\in \End\,\Wmod_{\alpha} $. We abbreviate $u_{\Wmod_{\alpha}}(n)$ by $u(n)$ throughout. For $u\in V_{N}$ we assume (cf. \eqref{eq:vjVm})
\begin{align*}
	u(j):\Wmod_{{\alpha},k+h_{\alpha}}\rightarrow \Wmod_{{\alpha},k+h_{\alpha}+N-j-1}.
\end{align*}
Let $\Wmod_{\alpha'}:=\Wmod_{\alpha}^{\,\prime} =\oplus_{n\ge 0}\Wmod^{*}_{{\alpha},n+h_{\alpha}}$ denote the canonical graded dual space of $\Wmod_{\alpha}$. A contragredient irreducible ordinary $V$-module $(\Wmod_{\alpha'},Y_{\alpha'})$ exists where \cite{FHL}
\begin{align}
	\label{eq:YMdual}
	\left\langle Y_{\alpha'}(u,z)v_{\alpha'},v_{\alpha}\right\rangle 
	= 	\left\langle v_{\alpha'},Y_{\alpha}
	\left(e^{z L(1)}\left(-z^{-2}\right)^{L(0)}u,z^{-1}\right)v_{\alpha}\right\rangle,
\end{align}
for all $v_{\alpha}\in \Wmod_{\alpha}$, $v_{\alpha'}\in \Wmod_{\alpha'}$, $u\in V$ for canonical pairing $\langle\cdot ,\cdot\rangle$ of  $\Wmod_{\alpha}$ and  $\Wmod_{\alpha'}$ (cf.~\eqref{eq:Yadj} when $\Wmod_{0}=V\cong V'$).
In particular, $\langle L(n)v_{\alpha'},v_{\alpha}\rangle =\langle v_{\alpha'},L(-n)v_{\alpha}\rangle  $ for Virasoro modes \cite{FHL}. 

For $\alpha,\beta,\gamma\in\calA$, define an intertwiner of type $\binom{\gamma}{\alpha\;\beta}$  to be a linear map \cite{FHL}
\begin{align*}
	\calY^{\gamma}_{\alpha\beta} :\Wmod_{\alpha} \otimes \Wmod_{\beta}   & \rightarrow \Wmod_{\gamma}\{z\}
	\\
	(v_{\alpha},v_{\beta}) & \mapsto 	\calY^{\gamma}_{\alpha\beta} (v_{\alpha},z)v_{\beta},
\end{align*}
where $\calY^{\gamma}_{\alpha\beta} (v_{\alpha},z)=\sum_{r\in\C}v_{\alpha}(r)z^{-r-1}$ for $v_{\alpha}(r)\in 
\Hom(\Wmod_{\beta},\Wmod_{\gamma})$ 
(with an appropriate lower truncation condition) satisfies the Jacobi identity:
\begin{align}
	\label{eq:gen_Intertwiner}
	\notag
	&z_{0}^{-1}\delta\left(\frac{z_{1}-z_{2}}{z_{0}}\right)
	Y_{\gamma}\left(u,z_{1}\right)
	\calY^{\gamma}_{\alpha\beta} \left(v_{\alpha},z_{2}\right)
	-z_{0}^{-1}\delta\left(\frac{z_{2}-z_{1}}{-z_{0}}\right)
	\calY^{\gamma}_{\alpha\beta} \left(v_{\alpha},z_{2}\right)
	Y_{\beta} \left(u,z_{1}\right)
	\\
	&=
	z_{2}^{-1}\delta\left(\frac{z_{1}-z_{0}}{z_{2}}\right)
	\calY^{\gamma}_{\alpha\beta}
	\left(Y_{\alpha }\left(u,z_{1}\right)v_{\alpha},z_{2}\right)
	,
\end{align}
for all $u\in V$,  $v_{\alpha}\in \Wmod_{\alpha} $ and translation
\begin{align}\label{eq:Lmin1Yint}
	\del_{z} \calY^{\gamma}_{\alpha\beta} (v_{\alpha},z)= \calY^{\gamma}_{\alpha\beta} (L(-1)v_{\alpha},z).
\end{align}
Furthermore,~\eqref{eq:gen_Intertwiner} implies the commutator  formula (cf.~\eqref{eq:ComId}) that for all $u\in V$
\begin{align}\label{eq:Com_calYcom}
u(k) \calY^{\gamma}_{\alpha\beta} (v_{\alpha},z)
- \calY^{\gamma}_{\alpha\beta} (v_{\alpha},z)u(k)
=\sum_{j\ge 0} \calY^{\gamma}_{\alpha\beta} 
\left (u(j)v_{\alpha},z\right)\del_{z} ^{(j)} z^{k}.
\end{align} 
Equation~\eqref{eq:gen_Intertwiner} is the standard Jacobi relation for a $V$-module $(Y_{\beta},\Wmod_{\beta})$ with $Y_{\beta}=\calY^{\beta}_{0\beta}$   for $\alpha=0$ and $\beta=\gamma$. The M\"obius maps~\eqref{eq:Trans}--\eqref{eq:Inv} can be suitably generalised to intertwiners.

Let $N^{\gamma}_{\alpha\beta}$ denote the dimension of the space of intertwining operators of type  $\binom{\gamma}{\alpha\;\beta}$. 
We consider VOAs for which
\begin{align}
	\label{eq:Verlinde}
	N^{\gamma}_{\alpha 0}=\delta_{\alpha\gamma},\quad N^{0}_{\alpha\beta}=\delta_{\alpha\beta'}.
\end{align}
The relations~\eqref{eq:Verlinde} hold for all rational VOAs satisfying the Verlinde formula \cite{H1} but also for Heisenberg VOA modules and allow us to define genus $g\ge 1$ partition functions unambiguously below. 
The unique  intertwining operator $\calY^{\alpha} _{\alpha 0}$ is the creative operator for $\Wmod_{\alpha}$ constructed via skew-symmetry \cite{FHL}
\begin{align}
\label{eq:YMMV}
		\calY^{\alpha} _{\alpha  0}(v_{\alpha},z)u:=e^{zL(-1)}Y_{\alpha} (u,-z)v_{\alpha},\quad u\in V.
\end{align}
The unique  intertwining operator $\calY^{0} _{\alpha  \alpha'}$ is the adjoint of $\calY^{\alpha} _{\alpha  0}$ \cite{FHL}, \cite{H1} defined by 
\begin{align}
	\label{eq:YVMpM}
	\left\langle u,\calY^{0} _{\alpha  \alpha'} (v_{\alpha},z)v_{\alpha'}\right\rangle = 
	\left\langle \calY^{\alpha}_{\alpha  0}\left(e^{z L(1)}\left(-z^{-2}\right)^{L(0)}v_{\alpha},z^{-1}\right)u,v_{\alpha'}\right\rangle,
\end{align}
where on the LHS we employ the standard Li-Z metric with $\rho=1$ and on the RHS the canonical pairing of  $\Wmod_{\alpha}$ and  $\Wmod_{\alpha'}$.

\subsection{Genus zero intertwining correlation functions}
For notational simplicity, we employ the convention throughout  that $v_{\alpha_{k}}$ denotes an element of $\Wmod_{\alpha_{k}}$ for $k=1,\ldots ,n$ where,  for $j\neq k$, we do not identify  $v_{\alpha_{j}}$  with $v_{\alpha_{k}}$ for repeated module labels  $\alpha_{j}=\alpha_{k}$. 

Define  the following  product of $n$ intertwiners for $v_{\alpha_{k}}\in\Wmod_{\alpha_{k}}$ for $k=1,\ldots, n$   with  intermediate module labels $\beta_{1},\ldots,\beta_{n-1}$ as follows
\begin{align}
	\label{eq:calY_alpha_beta}
	\bm{\calY^{\beta}_{\alpha}(v_{\alpha},y)}:=\calY^{0}_{\alpha_{1}\beta_{1}}(v_{\alpha_{1}},y_{1})
	\ldots
	\calY^{\beta_{k-1}}_{\alpha_{k}\beta_{k}}(v_{\alpha_{k}},y_{k})
	\ldots
	\calY^{\beta_{n-1}}_{\alpha_{n}0}(v_{\alpha_{n}},y_{n}),
\end{align}
with $\bm{v_{\alpha}}:=v_{\alpha_{1}},\ldots,v_{\alpha_{n}}$ labelled by
$\bm{\alpha}:=\alpha_{1},\ldots,\alpha_{n}$ and where, from~\eqref{eq:Verlinde}, we have 
\begin{align}
	\label{eq:beta_alpha}
	\beta_{1}=\alpha_{1}',\quad \beta_{n-1}=\alpha_{n},
\end{align} 
i.e. there are  $n-3$ independent intermediate module labels $\bm{\beta}$ for $n>3$. 
Equation~\eqref{eq:Verlinde} implies that if $\bm{\alpha}=\bm{0}$ then $\bm{\beta}=\bm{0}$.

Define an intertwining genus zero $n$-point correlation function for  
$v_{\alpha_{k}}\in \Wmod_{\alpha_{k}}$  inserted at $y_{k}$ for $k=1,\ldots,n$ by 
\begin{align}\label{eq:Z0Interdefn}
Z^{(0)}(\bm{v_{\alpha},y}\vert\bm{\alpha; \beta})	
	:=\langle \vac,
	\bm{\calY^{\beta}_{\alpha}(v_{\alpha},y)}
	\vac\rangle.
\end{align}
We suppress the $\bm{\alpha; \beta}$ labels if $\bm{\alpha}=\bm{0}$.
\begin{remark}\label{rem:intertwiner_conv}
We mainly consider the correlation function~\eqref{eq:Z0Interdefn} as a formal expression in this paper since it is not rational for non-integral module weights $h_{\alpha}$. 
However, given appropriate conditions on $V$ and its modules \cite{H1, H2, H3}, the correlation function can be extended to a multi-valued analytic function for $|y_{1}|>|y_{2}|>\ldots >|y_{n}|>0$. 
A braided version of intertwiner commutativity can also be demonstrated which plays a central role in the proof of the Verlinde conjecture for $C_{2}$-cofinite VOAs of CFT type for which every $\N$-gradable weak $V$-module is completely reducible \cite{H1}. 
\end{remark}
We may specialise~\eqref{eq:Z0Interdefn} to where $u\in V$ is inserted at $x$ and define
\begin{align}
	\label{eq:Z0uint}
Z^{(0)}(u,x;\bm{v_{\alpha},y}\vert 0,\bm{\alpha};0,\bm{\beta})	
:=\langle \vac, Y(u,x)
\bm{\calY^{\beta}_{\alpha}(v_{\alpha},y)}
\vac\rangle.
\end{align}
Using~\eqref{eq:Com_calYcom} we find Propositions~\ref{prop:ZeroWard}--\ref{prop:GenusZeroZhu3} generalise as follows:
\begin{proposition}
	\label{prop:GenusZeroZhuM}
Let $u\in V$ be a quasiprimary vector of weight $N\ge 1$.
Then
\begin{align}
	\label{eq:Zhu0M1}
&\sum_{k=1}^n \sum_{j\ge 0}
	Z^{(0)}
	(\ldots;u(j)v_{{\alpha}_{k}},y_{k};\ldots|\bm{\alpha; \beta}) p^{(j)}(y_{k})=0, \mbox{ for all } p\in \Poly_{2N-2},
\\
\label{eq:Zhu0M2}
	&Z^{(0)}(u,x;\bm{v_{\alpha},y}|0,\bm{\alpha};0,\bm{\beta})
=
	\sum_{k=1}^{n}\sum_{j\ge 0}
	\pi_{N}^{(0,j)}(x,y_{k})Z^{(0)}(\ldots;u(j){\alpha_{k}},y_{k};\ldots|\bm{\alpha; \beta}),
\\
\label{eq:Zhu0M3}
	&Z^{(0)}(u(-i-1)v_{\alpha_{1}},y_{1};\ldots |\bm{\alpha; \beta})
	\\
	\notag
	&=\sum_{j\ge 0}\E^{j}_{i}(y_{1})
	\Zzero(u(j)v_{\alpha_{1}},y_{1};\ldots |\bm{\alpha; \beta})
	\\
&
\notag
\quad +\sum_{k=2}^{n}\sum_{j\ge 0}
\pi_{N}^{(i,j)}(y_{1},y_{k})\Zzero(\ldots ;u(j)v_{\alpha_{k}},y_{k};\ldots|\bm{\alpha; \beta}),
\end{align}
for $i\ge 0$, $\pi_{N}(x,y)$ of~\eqref{eq:piNdef} and $\E^{j}_{i}(y)$ of~\eqref{eq:Eji}.
\end{proposition}
\begin{remark}\label{rem:Z0uintr}
Equation~\eqref{eq:Zhu0M2} can be modified to describe $u\in V$ inserted after the $r^{\textrm{th}}$ intertwiner in the correlation function
\begin{align}  
\notag
&\langle \vac, 
	\ldots
	\calY^{\beta_{r-1}}_{\alpha_{r}\beta_{r}}(v_{\alpha_{r}},y_{r})
	Y_{\beta_{r}}(u,x)
	\calY^{\beta_{r}}_{\alpha_{r+1}\beta_{r+1}}(v_{\alpha_{r+1}},y_{r+1})
	\ldots
	\vac\rangle\notag\\
	&=
\sum_{k=1}^{r}\sum_{j\ge 0}
\pi_{N}^{(j,0)}(y_{k},x)Z^{(0)}(\ldots;u(j)v_{\alpha_{k}},y_{k};\ldots|\bm{\alpha; \beta})\notag
\\
\label{eq:Zhu0M2r}
&\quad+
\sum_{k=r+1}^{n}\sum_{j\ge 0}
\pi_{N}^{(0,j)}(x,y_{k})Z^{(0)}(\ldots;u(j)v_{\alpha_{k}},y_{k};\ldots|\bm{\alpha;\beta}).
\end{align} 
\end{remark}
Consider  two dual modules $\Wmod_{\alpha}$, $\Wmod_{\alpha'}$.
Let $v_{k}\in V$ for $k=1,\ldots, n$, $v_{\alpha}\in\Wmod_{\alpha}$, $v_{\alpha'}\in \Wmod_{\alpha'}$ inserted at $z_{k}$, $y_{1}$, $y_{2}$ respectively. We define 
\begin{align}
\label{eq:Z0vmm}
	Z^{(0)}(\bm{v,z};v_{\alpha},y_{1};v_{\alpha'},y_{2}):&=
	\left\langle\vac,	
	\bm{Y(v,z)}
	\calY^{0}_{\alpha\alpha'}\left(v_{\alpha},y_{1}\right)
	\calY^{\alpha'}_{\alpha'0}\left(v_{\alpha'},y_{2}\right)
	\vac\right\rangle,
	\\
	\label{eq:Z0mvm}
		Z^{(0)}(v_{\alpha},y_{1};\bm{v,z};v_{\alpha'},y_{2}):&=
	\left\langle\vac,	
	\calY^{0}_{\alpha\alpha'}\left(v_{\alpha},y_{1}\right)
	\bm{Y_{\alpha'} (v,z)}
	\calY^{\alpha'}_{\alpha'0}\left(v_{\alpha'},y_{2}\right)
	\vac\right\rangle.
\end{align}
We then find \cite{H2}
\begin{proposition}
	\label{prop:Zalpha_rational} 
Let $\Wmod_{\alpha},\Wmod_{\alpha'}$ have conformal weight $h_{\alpha}\in\C$. Then 
\[
\left(y_{1}-y_{2}\right)^{2h_{\alpha}}	Z^{(0)}(\bm{v,z};v_{\alpha},y_{1};v_{\alpha'},y_{2}) \mbox{ \textup{and} }	\left(y_{1}-y_{2}\right)^{2h_{\alpha}}Z^{(0)}(v_{\alpha},y_{1};\bm{v,z};v_{\alpha'},y_{2})
\]
can be extended to the same rational function in the domains $|z_{1}|>\ldots >|z_{n}|>|y_{1}|>|y_{2}|$  and $|y_{1}|>|z_{1}|>\ldots >|z_{n}|>|y_{2}|$, respectively.
\end{proposition}
\begin{proof}
Consider $Z^{(0)}(u_{\alpha},y_{1};u_{\alpha'},y_{2})$ where $u_{\alpha}\in\Wmod_{\alpha}$ and $u_{\alpha'}\in \Wmod_{\alpha'}$ are quasiprimary of weight $\wt(u_{\alpha})=h_{\alpha}+n_{\alpha}$ and $\wt(u_{\alpha'})=h_{\alpha}+n_{\alpha'}$, respectively, for some integers $n_{\alpha},n_{\alpha'}\ge 0$. Using~\eqref{eq:YMMV} and~\eqref{eq:YVMpM}, translation and the fact that $u_{\alpha},u_{\alpha'}$ are quasiprimary to find
\begin{align*}
	Z^{(0)}(u_{\alpha},y_{1};u_{\alpha'},y_{2})
	=&\left\langle \vac,\calY^{0}_{\alpha\alpha'}\left(u_{\alpha},y_{1}\right)
	e^{y_{2}L(-1)}u_{\alpha'}\right\rangle
	=\left\langle \vac,\calY^{0}_{\alpha\alpha'}\left(u_{\alpha},y_{1}-y_{2}\right)
u_{\alpha'}\right\rangle
	\\
	=&
	\left(-\left(y_{1}-y_{2}\right)^{-2}\right)^{\wt(u_{\alpha})}
	\left\langle 
	e^{(y_{1}-y_{2})^{-1}L(-1)}u_{\alpha}, u_{\alpha'} \right\rangle
	\\
	=&
	(-1)^{\wt(u_{\alpha})}
	\left(y_{1}-y_{2}\right)^{-2\wt(u_{\alpha})}\langle u_{\alpha},	u_{\alpha'} \rangle.
\end{align*}
Note that $\langle u_{\alpha},	u_{\alpha'} \rangle\neq 0$ implies $\wt(u_{\alpha})= \wt(u_{\alpha'})$.
Thus $(y_1-y_2)^{2h_{\alpha}}Z^{(0)}(u_{\alpha},y_{1};u_{\alpha'},y_{2})$ is rational in $y_{1},y_{2}$. By the translation property~\eqref{eq:Lmin1Yint} the same holds for all quasiprimary descendants. Thus $(y_1-y_2)^{2h_{\alpha}}Z^{(0)}(v_{\alpha},y_{1};v_{\alpha'},y_{2})$ is rational for all
$v_{\alpha}\in\Wmod_{\alpha}$ and $v_{\alpha'}\in \Wmod_{\alpha'}$.
Equation~\eqref{eq:Zhu0M2} of Proposition~\ref{prop:GenusZeroZhuM} implies that 
$Z^{(0)}(\bm{v,z};v_{\alpha},y_{1};v_{\alpha'},y_{2})$ can be written as  a linear combination of $2$-point correlation functions of the form $Z^{(0)}(u_{\alpha},y_{1};u_{\alpha'},y_{2})$ for some 
$u_{\alpha}\in\Wmod_{\alpha}$ and $u_{\alpha'}\in \Wmod_{\alpha'}$ with coefficients which are rational in $z_{k}$, $y_{1}$, $y_{2}$. Hence for $|z_{1}|>\ldots >|z_{n}|>|y_{1}|>|y_{2}|$ we find $\left(y_{1}-y_{2}\right)^{2h_{\alpha}}	Z^{(0)}(\bm{v,z};v_{\alpha},y_{1};v_{\alpha'},y_{2})$ can be extended to a rational function. By~\eqref{eq:Zhu0M2r} we find that
$ \left(y_{1}-y_{2}\right)^{2h_{\alpha}}Z^{(0)}(v_{\alpha},y_{1};\bm{v,z};v_{\alpha'},y_{2})$ can be extended to the same rational function for $|y_{1}|>|z_{1}|>\ldots >|z_{n}|>|y_{2}|$. 
\end{proof}
We may generalise Lemma~\ref{lem:FMobiusQP} and 	Proposition~\ref{prop:ZMobius} concerning M\"obius maps to correlation functions of the form~\eqref{eq:Z0mvm}. 
For $\Wmod_{\alpha},\Wmod_{\alpha'}$ of conformal weight $h_{\alpha}\in\C$ let
\begin{align*}
\Fzero(v_{\alpha},y_{1};\bm{v,z};v_{\alpha'},y_{2}):=
\Zzero(v_{\alpha},y_{1};\bm{v,z};v_{\alpha'},y_{2})\bm{dz^{\wt(v)}}
dy_{1}^{h_{\alpha}} dy_{2}^{h_{\alpha}}.
\end{align*}
\begin{lemma} 
	\label{lem:FalphaMobiusQP}	
Let $\Wmod_{\alpha},\Wmod_{\alpha'}$ have conformal weight $h_{\alpha}\in\C$ and let  $v_{\alpha},v_{\alpha'},v_{1},\ldots,v_{n}$ be quasiprimary. Then 
	\begin{align}\label{eq:MobFalpha0}
		\Fzero(v_{\alpha},y_{1};\bm{v,z};v_{\alpha'},y_{2})= \Fzero(v_{\alpha},\gamma y_{1};\bm{v,\gamma z};v_{\alpha'},\gamma y_{2}),
	\end{align}
for all $\gamma=\left(\begin{smallmatrix}a&b\\c&d\end{smallmatrix}\right)\in\SL_{2}(\C)$
\end{lemma} 
\begin{proof}
	Applying~\eqref{eq:Trans}--\eqref{eq:Inv} for intertwiners we find~\eqref{eq:MobFalpha0} holds
	for each $\SL_{2}(\C)$ generator $\gamma\in\{ \left(\begin{smallmatrix}
		1&x\\0&1
	\end{smallmatrix}\right), 
	x^{-1/2}\left(\begin{smallmatrix}
		x & 0 \\ 0 & 1
	\end{smallmatrix}\right), 
	\left(\begin{smallmatrix}
		1 & 0 \\ -x & 1
	\end{smallmatrix}\right) \}$.	
Thus for those generators we find
\begin{align}
&\left(y_{1}-y_{2}\right)^{2h_{\alpha}}\Zzero(v_{\alpha},y_{1};\bm{v,z};v_{\alpha'},y_{2})
\notag
\\
\notag
&=
\left(y_{1}-y_{2}\right)^{2h_{\alpha}}\Zzero(v_{\alpha},\gamma y_{1};\bm{v,\gamma z};v_{\alpha'},\gamma y_{2})\prod_{k}\left(\frac{d(\gamma z_{k})}{dz_{k}}\right)^{\wt(v_{k})}
\left(\frac{d(\gamma y_{1})}{dy_{1}}
\frac{d(\gamma y_{2})}{dy_{2}}\right)^{h_{\alpha}}
\\
\label{eq:gamZ}
&=
(\gamma y_{1}-\gamma y_{2})^{2h_{\alpha}}\Zzero(v_{\alpha},\gamma y_{1};\bm{v,\gamma z};v_{\alpha'},\gamma y_{2})
\prod_{k}\left(\frac{d(\gamma z_{k})}{dz_{k}}\right)^{\wt(v_{k})},
\end{align}
using 
$\gamma y_{1}-\gamma y_{2}
=\frac{y_{1}-y_{2}}{(cy_{1}+d)(cy_{2}+d)}$.
By Proposition~\ref{prop:Zalpha_rational} the RHS of~\eqref{eq:gamZ} can be extended to a rational function for all $\gamma\in \SL_{2}(\C)$ so that~\eqref{eq:gamZ} holds for all compositions of the generators.
Hence~\eqref{eq:MobFalpha0} holds.
\end{proof}
We may use Lemma~\ref{lem:FalphaMobiusQP} to repeat the proof of Proposition~\ref{prop:ZMobius} to find 
\begin{proposition}
	\label{prop:Z_alpha_Mobius}
	Let $v_{\alpha}\in \Wmod_{\alpha}$ and $v_{\alpha'}\in \Wmod_{\alpha'}$ of weight $h_{\alpha}\in\C$ and $v_{k}\in V$ of weight $\wt(v_{k})$ for $k=1,\ldots,n$. For all $\gamma=\left(\begin{smallmatrix}a&b\\c&d\end{smallmatrix}\right)\in\SL_{2}(\C)$ we have
	\begin{align*}
		&\Zzero(v_{\alpha},y_{1};\bm{v,z};v_{\alpha'},y_{2})
		\\
		&= 
		\Zzero\left(
		\frac{e^{ -c(cy_{1}+d)L(1)}	}
		{(cy_{1}+d)^{2h_{\alpha}}}v_{\alpha},\gamma y_{1};\ldots; 
		\frac{e^{ -c(cz_{k}+d)L(1)}	}
		{(cz_{k}+d)^{2h_{k}}}v_{k},\gamma z_{k};
		\ldots;
		\frac{e^{ -c(cy_{2}+d)L(1)}	}
		{(cy_{2}+d)^{2h_{\alpha}}}v_{\alpha'},\gamma y_{2}
		\right).
	\end{align*}
\end{proposition}
\subsection{Genus $g$ intertwining correlation functions}
We next define the genus $g$ partition and $n$-point correlation functions for modules with intertwiners. In order to define the most general partition function, we consider $g$ modules $\Wmod_{\alpha_{a}}$ for $a=1,\ldots,g$ together with corresponding contragredient modules $\Wmod_{\alpha_{a}'}$. 
Let $\{b_{a}\}$ be any $\Wmod_{\alpha_{a}}$-basis and  $\{b_{a}'\}$ the canonical dual $\Wmod_{\alpha'_{a}}$-basis. It is also convenient as in \S\ref{sec:Genusgnpt} to define a scaled $\Wmod_{\alpha'_{a}}$-basis with elements  $b_{-a}:=\rho_{a}^{\wt(b_{a})}b_{a}'$ (with $\rho_{a}$ the usual sewing factor as before).

 Consider the product of $2g$  intertwiners for $\Wmod_{\alpha_{a}}$-basis vectors $b_{ a}$ inserted at $w_{a}$ for $a\in\I$ and with $2g-1$ intermediate module labels $\bm{\beta}:=\beta_{1},\ldots \beta_{2g-1}$ 
\begin{align}
	\label{eq:Y_alph_alphp_beta}
	\bm{\calY^{\beta}_{\alpha \alpha'}(b,w)}:=
	\calY^{0}_{\alpha_{1}\beta_{1}}(b_{1},w_{1})
	\calY^{\beta_{1}}_{\alpha'_{1}\beta_{2}}(b_{-1},w_{-1})
	\ldots
	\calY^{\beta_{2g-2}}_{\alpha_{g}\beta_{2g-1}}(b_{g},w_{g})
	\calY^{\beta_{2g-1}}_{\alpha'_{g}0}(b_{-g},w_{-g}),
\end{align}
where applying~\eqref{eq:beta_alpha} we  find $\beta_{1}=\alpha'_{1}$ and $\beta_{2g-1}=\alpha'_{g}$.

Let $\{\bm{b_{+}}\}=\{b_{1}\otimes\ldots\otimes b_{g}\}$ denote a basis for $\bm{\Wmod_{\alpha}}:=\bigotimes_{a=1}^{g}\Wmod_{\alpha_{a}}$.
Then we can define a genus $g$ partition function for $g$ simple ordinary $V$-modules labelled by $\bm{\alpha}$ and $2g-1$ intermediate modules labelled by $\bm{\beta}$:
\begin{align}
	\label{eq:Zginter}
Z_{\bm{\alpha \beta}}:=Z_{\bm{\alpha \beta}}(\bm{w,\rho}):=
\sum_{\bm{b_{+}}\in \bm{\Wmod_{\alpha}}}\langle \vac,
\bm{\calY^{\beta}_{\alpha \alpha'}(b,w)}
\vac\rangle,
\end{align}
where the sum is taken over any $\bm{\Wmod_{\alpha}}$-basis. Note that Remark~\ref{rem:Zg_rho_fact} applies to the expression~\eqref{eq:Zginter} similarly to $Z_{V}$ and that $Z_{\bm{\alpha \beta}}=Z_{V}$ when $\bm{\alpha}=\bm{0}$ (so that $\bm{\beta}=\bm{0}$). 
\begin{remark}\label{rem:int_order}
Note that we have made one ordering choice for the intertwiners in~\eqref{eq:Y_alph_alphp_beta} and~\eqref{eq:Zginter}. Following Remark~\ref{rem:intertwiner_conv} we conjecture that on identifying the formal parameters $\bm{w,\rho}$ with Schottky sewing parameters then all such partition functions are convergent on $\Cg$ for $C_{2}$-cofinite VOAs of CFT-type and that the set of partition functions for all $\bm{\alpha, \beta}$ is unique up to braiding scalar multiple factors under reordering of the intertwiners.
\end{remark}
We can finally define a genus $g$ general intertwiner $n$-point correlation function and formal differential by the  insertion of an additional intertwiner product operator $\bm{\calY^{\delta}_{\gamma}(v_{\gamma},y)}$ of~\eqref{eq:calY_alpha_beta} for $n$ module vectors with labels $\bm{\gamma}$ and $n-1$ intermediate modules labelled by $\bm{\delta}$
 \begin{align}
 	\label{eq:Zg_gen_inter}
 	Z_{\bm{\alpha \beta}}
 	(\bm{v_{\gamma},y}\,|\bm{\gamma;\delta}):&=
 	\sum_{\bm{b_{+}}}\langle \vac,
 	\bm{\calY^{\delta}_{\gamma}(v_{\gamma},y)}
 	\bm{\calY^{\beta}_{\alpha \alpha'}(b,w)}
 	\vac\rangle,
 	\\
 	\label{eq:Fg_gen_inter}
 	\F_{\bm{\alpha \beta}}(\bm{v_{\gamma},y}\,|\bm{\gamma;\delta})
 	:&=Z_{\bm{\alpha \beta}}(\bm{v_{\gamma},y}\,|\bm{\gamma;\delta})\bm{dy^{\wt(v_{\gamma})}},
 \end{align}
with $v_{\gamma_{k}}$ of weight $\wt(v_{\gamma_{k}})$ inserted at $y_{k}$ for $k=1,\ldots,n$ and $n-1$ intermediate module labels $\delta_{1},\ldots,\delta_{n-1}$ subject to $\delta_{1}=\gamma'_{1}$ and $\delta_{n-1}=\gamma_{n}$. We suppress the $\bm{\gamma;\delta}$ labels if $\bm{\gamma}=\bm{0}$.
The $\bm{\alpha},\bm{\beta}$ labels are subject to the same condition as in \eqref{eq:Y_alph_alphp_beta}.
\begin{remark}
\label{rem:MobWardInt}
\leavevmode
\begin{enumerate}
\item 
Using Proposition~\ref{prop:GenusZeroZhuM} we find that  Propositions~\ref{prop:DpZ} and~\ref{prop:DpyF}, which concern Schottky M\"obius invariance, and~\ref{prop:GenusgWard}, which concerns a general Ward identity, can be suitably generalised in an obvious way to~\eqref{eq:Zg_gen_inter} and~\eqref{eq:Fg_gen_inter}. 
\item 
Gui's Theorem~13.1 \cite{G} again implies that 
$\F_{\bm{\alpha \beta}}(\bm{v_{\gamma},y}\,|\bm{\gamma;\delta})$ is absolutely and locally uniformly convergent on the sewing domain if $V$ is $C_{2}$-cofinite cf. Remark~\ref{rem:Gui_convegence}.
\end{enumerate}
\end{remark}

\subsection{Comparison to genus one  $V$-module trace functions}
The genus one $n$-point function  for $v_{k}\in V$, $k=1,\ldots,n$ for a simple ordinary  module $\Wmod_{\alpha}$ with lowest weight $h_{\alpha}$ is defined by Zhu by the trace~\cite{Z1}
\begin{align*}
	Z^{\Zhu}_{\alpha}(\bm{v,\zeta};q):=&
	\Tr_{\Wmod_{\alpha}}Y_{\alpha}\left(e^{ \wt(v_{1})\zeta_{1}}v_{1},e^{\zeta_{1}}\right)\ldots  
	Y_{\alpha}\left(e^{\wt(v_{n})\zeta_{n} }v_{n},e^{\zeta_{n}}\right)q^{L(0)}.
\end{align*}
Consider the genus one $n$-point correlation function of the form~\eqref{eq:Zg_gen_inter} with $\alpha_{1}=\alpha$ and $\beta_{1}=\alpha'$  and   $v_{k}\in V$ inserted at $z_{k}$ for $k=1,\ldots,n$   given by (with $\bm{\gamma}=\bm{\delta}=\bm{0}$)
\begin{align*}
	Z_{\alpha \alpha'}(\bm{v,z})=&
	\sum_{b_{1}\in \Wmod_{\alpha}}\langle \vac,
	\bm{Y(v,z)}
	\calY^{0}_{\alpha \alpha'}(b_{1},w_{1})
	\calY^{\alpha'}_{\alpha' 0}(b_{-1},w_{-1})
	\vac\rangle 
	\\
	=&\sum_{n\ge 0}\rho^{n+h_{\alpha}}\sum_{b\in \Wmod_{\alpha,n+h_{\alpha}}}
	Z^{(0)}(\bm{v,z};b,w_{1};b',w_{-1}),
\end{align*}
in the notation of~\eqref{eq:Z0vmm}. From~\eqref{eq:rhoa} we have $\rho^{h_{\alpha}} =\left(-q /(1+q)^{2}\right)^{h_{\alpha}}\left(w_{-1}-w_{1}\right)^{2h_{\alpha}}$. 
We may extend $\left(w_{-1}-w_{1}\right)^{2h_{\alpha}}Z^{(0)}(\bm{v,z};b,w_{1};b',w_{-1})$ and 
$\left(w_{-1}-w_{1}\right)^{2h_{\alpha}}Z^{(0)}(b,w_{1};\bm{v,z};b',w_{-1})$ to the same rational function by  Proposition~\ref{prop:Zalpha_rational} so that we define (cf. Remark~\ref{rem:int_order})
\begin{align*}
	Z_{\alpha}^{(1)}(\bm{v,z}):=\sum_{n\ge 0}\rho^{n+h_{\alpha}}\sum_{b\in \Wmod_{\alpha,n+h_{\alpha}}}
	Z^{(0)}(b,w_{1};\bm{v,z};b',w_{-1}).
\end{align*}
 We also define the corresponding formal differential
\begin{align*}
\F_{\alpha}^{(1)}(\bm{v,z}):=	&Z_{\alpha}^{(1)}(\bm{v,z})\bm{dz^{\wt(v)}}.
\end{align*}
Using M\"obius symmetry Proposition~\ref{prop:Z_alpha_Mobius}, we may repeat the arguments of \S\ref{subsec:ZhuTrace} to relate the Zhu definition of a $V$-module trace function $\F_{\alpha}(\bm{v,z})$ as follows:
\begin{theorem}\label{theor:ZZhuM}
	Let $v_{1},\ldots,v_{n}$ be quasiprimary vectors inserted at $z_{1},\ldots, z_{n}$. Then 
	\[
	\F_{\alpha}^{(1)}(\bm{v,z})=Z^{\Zhu}_{\alpha'}(\bm{v,\zeta};q)\bm{d\zeta^{\wt(v)}},
	\]
	for  $\zeta_{i} =\log Z_{i}$ with $Z_{i}=\sigma z_{i}=\frac{z_{i}-W_{-1}}{z_{i}-W_{1}}$ for $i=1,\ldots, n$
	and $\sigma$ of~\eqref{eq:gam_sigma}.
\end{theorem}

\begin{proof}
Let	$\Wmod_{\alpha}$ be of conformal weight $h_{\alpha}$.	Consider  
	\begin{align*}
		\F_{\alpha}^{(1)} (\bm{v,z})
		=
		&\sum_{n\ge 0}\rho^{n+h_{\alpha}}\sum_{b\in \Wmod_{\alpha,n+h_{\alpha}}}
\Zzero(b,w_{1};\bm{v,z};b',w_{-1}) 
\bm{dz^{\wt(v)}},
	\end{align*}
	for $\Wmod_{\alpha}$-basis $\{b\}$ and $\Wmod_{\alpha'} $ dual basis $\{b'\}$. 
	Applying  Proposition~\ref{prop:Z_alpha_Mobius} for $\sigma $ we obtain
	\begin{align*}
		&\F_{\alpha}^{(1)}(\bm{v,z})=
		\sum_{n\ge 0}
		Q^{n+h_{\alpha}}
		\sum_{b\in \Wmod_{\alpha,n+h_{\alpha}}}
\Zzero\left(
e^{\frac{q}{1-q}L(1)}\;b,q^{-1};\bm{v,Z}; e^{-\frac{1}{1-q}L(1)}\;b',q
\right)
		\bm{dZ^{\wt(v)}},
	\end{align*}
for $Q:=-q^{-1}(1-q)^2$  and $Z_{i}:=\sigma z_{i}$.
	From~\eqref{eq:YMMV} and~\eqref{eq:YVMpM} we find that
	\begin{align*}
		& 
		\Zzero\left(
		e^{\frac{q}{1-q}L(1)}\;b,q^{-1};\bm{v,Z}; e^{-\frac{1}{1-q}L(1)}\;b',q
		\right)
		\\
		&=
		\left\langle 
		b,\;
		e^{\frac{q^2}{1-q}L(-1)}(-q^2)^{L(0)}e^{qL(1)}
		\bm{ Y_{\alpha'} \left(v,Z\right)}e^{qL(-1)}e^{-\frac{1}{1-q}L(1)}\;b'\right\rangle.
	\end{align*}
Using similar arguments as given in the proof of Theorem~\ref{theor:ZZhu} we find
	\begin{align*}
		\F_{\alpha}^{(1)}(\bm{v,z})
		&=
		\Tr_{\Wmod_{\alpha'}} \left(\bm{Y_{\alpha'} \left(v,Z\right)}e^{qL(-1)}e^{-\frac{1}{1-q}L(1)}
		Q^{L(0)}
		e^{\frac{q^2}{1-q}L(-1)}(-q^2)^{L(0)}e^{qL(1)}\right)
		\bm{dZ^{\wt(v)}}
		\\
		&
		=\Tr_{\Wmod_{\alpha'}} \left( \bm{ Y_{\alpha'} \left(v,Z\right)}q^{L(0)}\right) \bm{dZ^{\wt(v)}}=Z^{\Zhu}_{\Wmod_{\alpha'}} (\bm{v,\zeta};q)\bm{d\zeta^{\wt(v)}}.
	\end{align*}
\end{proof}

\subsection{Genus $g$ Zhu recursion for intertwiner $n$-point correlation functions}
We give an overview of a genus Zhu recursion formula involving intertwining operators. The proof parallels that for VOAs so we will not discuss it in detail. 
We first give a generalisation of Lemma~\ref{lem:adjoint}. Let $\{ b\}$ be a homogeneous basis for a module $\Wmod$ of conformal weight $h$ with contragredient module $\Wmod^{\,\prime}$ basis $\{b'\}$. 
Using~\eqref{eq:YMdual} for quasiprimary $u\in V$ we find:
\begin{lemma}\label{AdjointLemmaInt}
	For quasiprimary $u\in V$ of weight $N$ and irreducible module $\Wmod$ we have
	\begin{align*}
	\sum_{b\in \Wmod_{n+h}}
	\left(u(k)b\right)\otimes b' =(-1)^{N}\sum_{b\in \Wmod_{n+N-k-1+h}} b\otimes \left(u(2N-2-k)b' \right),
	\end{align*}
summing over any basis for the given homogeneous spaces.
\end{lemma}
We now describe a 
generalisation of Theorem~\ref{theor:ZhuGenusg} for an intertwiner formal $(n+1)$-point differential (cf.~\eqref{eq:Fg_gen_inter})
\begin{align*}
\F_{\bm{\alpha \beta}}(u,x;\bm{v_{\gamma},y}\,|0,\bm{\gamma};0,\bm{\delta})
=
\sum_{\bm{b_{+}}}\langle \vac,Y(u,x)
\bm{\calY^{\delta}_{\gamma}(v_{\gamma},y)}
\bm{\calY^{\beta}_{\alpha \alpha'}(b,w)}
\vac\rangle\,
dx^{N}\bm{dy^{\wt(v_{\gamma})}},
\end{align*} 
for $g$ modules labelled by $\bm{\alpha}$ and $2g-1$ intermediate modules labelled by $\bm{\beta}$. The first vector is $u\in V$ with module label $0$, which is quasiprimary of weight $N$, is inserted at $x$ via  $Y(u,x)$ with module label $0$ as notated above.
 $v_{\gamma_{k}}\in\Wmod_{\gamma_{k}}$ is inserted at $y_{k}$ for $k=1,\ldots, n$ with $n-1$ intermediate module labels
 $\bm{\delta}=\delta_{1},\ldots,\delta_{n-1}$. 
Proposition~\ref{prop:GenusZeroZhuM} and Lemma~\ref{AdjointLemmaInt} imply that we may proceed in a manner identical to that of the VOA case to generalise Theorem~\ref{theor:ZhuGenusg} as follows:
\begin{theorem}
	\label{theor:ZhuGenusgInt}
For quasiprimary vector $u\in V$ of weight $\wt(u)=N$ and with $v_{\gamma_{k}}\in\Wmod_{\gamma_{k}}$ inserted at $y_{k}$ for $k=1,\ldots,n$ with intermediate module labels $\bm{\delta}=\delta_{1},\ldots,\delta_{n-1}$, we have
\begin{align}
\F_{\bm{\alpha \beta}}(u,x;\bm{v_{\gamma},y}\,|0,\bm{\gamma};0,\bm{\delta})=&
\sum_{a\in\Ip}\sum_{\ell\in \calL} \Theta_{N,a}^{\ell}(x) \Res_{a}^{\ell}\F_{\bm{\alpha \beta}}(u;\bm{v_{\gamma},y}\,|0,\bm{\gamma};0,\bm{\delta})
\notag
\\
&
\label{eq:ZhuGenusgInter}
+\sum_{k=1}^{n}\sum_{j\ge 0}\Psi_{N}^{(0,j)}(x,y_{k})
\F_{\bm{\alpha \beta}}\left(\ldots;u(j)v_{\gamma_{k}},y_{k};\ldots |\bm{\gamma;\delta}\right)dy_{k}^{j},
\end{align}
 for $\Psi_{N}$ of~\eqref{eq:PsiNdef} and $\Theta_{N,a}^{\ell}$ of~\eqref{eq:Thetaexp} and where
\begin{align}
	\label{eq:ZhuGenusgInterRes}
\Res_{a}^{\ell}\F_{\bm{\alpha \beta}}(u;\bm{v_{\gamma},y}\,|0,\bm{\gamma};0,\bm{\delta}):=
\Res_{x-w_{a}}\left(x-w_{a}\right)^{\ell}\F_{\bm{\alpha \beta}}(u,x;\bm{v_{\gamma},y}\,|0,\bm{\gamma};0,\bm{\delta}).
\end{align}
\end{theorem}
Remark~\ref{rem:MainTheorem} similarly applies to Theorem~\ref{theor:ZhuGenusgInt}. 
We may also find a natural generalisation of Corollary~\ref{cor:GenGenusgZhu}.

\section{Genus $g$ Conformal Ward Identities}
\label{sec:Wardids}
In this section we analyse general genus $g$ intertwiner correlation functions where we apply Zhu recursion to  conformal vector $\omega$ insertions. 
 We determine various  genus $g$ Virasoro Ward identities which involve variations with respect to the Schottky parameters $w_{a},\rho_{a}$ and the insertion points of any other correlation function vectors.
Thus we explicitly realise  the relationship between Virasoro correlation functions and variations in moduli space exploited in conformal field theory  e.g.~\cite{EO, Ma}. 
We obtain a canonical differential operator $\nabla(x)$ which maps differentiable functions of the Schottky parameters  to the space of holomorphic 2-forms in $x$ (and naturally leads to a similar operator acting on differentiable functions on moduli space). We generalise this to a differential operator $\nabmy{\bfm}{\bfy}(x)$ that utilises  $\Psi_{2}$ and which maps any meromorphic form  $H_{\bfm}(\bfy)=h_{\bfm}(\bfy)dy_{1}^{m_{1}} \ldots dy_{n}^{m_{n}}$ of weight $(\bfm):=(m_{1},\ldots ,m_{n})$ in $\bfy:=y_{1},\ldots,y_{n}$ to a meromorphic form of weight $(2,\bfm)$ in $x,\bfy$. 
The geometric properties of these operators are further described in \cite{TW1}. Many of these can be recovered in the VOA setting.

\subsection{Virasoro $1$-point differentials and variations of moduli}
Consider~\eqref{eq:Fg_gen_inter} for the  1-point differential, for $g$ modules labelled by $\bm{\alpha}$ and $2g-1$ intermediate modules labelled by $\bm{\beta}$, for the Virasoro vector
\begin{align*}
	\F_{\bm{\alpha\beta}}(\omega,x)=\sum_{a\in\Ip} \sum_{\ell=0}^{2}\Theta_{2,a}^{\ell}(x)
	\Res_{a}^{\ell}\F_{\bm{\alpha\beta}}(\omega),
\end{align*}
where $\{ \Theta_{2,a}^{\ell}(x) \}$ spans the $(3g-3)$-dimensional space of holomorphic $2$-forms. We introduce $3g$ (first-order) differential operators
$\del_{a}^{\,\ell}$  for $a\in\Ip$ and $\ell=0,1,2$ as follows:
\begin{align*}
	\del_{a}^{0}:=\del_{w_{a}},\quad
	\del_{a}^{1}:=\rho_{a} \del_{\rho_{a}},\quad
	\del_{a}^{2}:=\rho_{a} \del_{w_{-a}}.
\end{align*}
These form a canonical basis for the tangent space $T(\Cg)$.
Using this basis we may rewrite the M\"obius generator $\D^{p}$ of~\eqref{eq:Dpdiff} for $p\in \Poly_{2}$ with  $p_{a}^{\ell}$ of~\eqref{eq:pal} as follows:
\begin{align*}
	\D^{p}=-\sum_{a\in\Ip}\sum_{\ell=0}^{2}p_{a}^{\ell}\del_{a}^{\,\ell}.
\end{align*} 
\begin{lemma} 
	\label{lem:omRes}
	For $a\in\Ip$ and $\ell=0,1,2$ we find
	\begin{align*}
		\Res_{a}^{\ell}\F_{\bm{\alpha\beta}}(\omega)=\del_{a}^{\,\ell}\Zg_{\bm{\alpha\beta}}.
	\end{align*}
\end{lemma}
\begin{proof}
	From~\eqref{eq:ZhuGenusgInterRes} we have 
	\begin{align*}
		\Res_{a}^{\ell}\F_{\bm{\alpha\beta}}(\omega)
		&=\sum_{\bm{b_{+}}\in \bm{\Wmod_{\alpha}}}
		\Zzero(\ldots;L(\ell-1)b_{a},w_{a};\ldots),
	\end{align*}
	for $\ell=0,1,2$. The result follows from  generalisations  of~\eqref{eq:ZL1}--\eqref{eq:ZL3}, cf. Remark~\ref{rem:MobWardInt}.
\end{proof}
Define the first order differential operator ~\cite{GT, TW1}
\begin{align*}
	\delx:=\sum_{a\in\Ip}\sum_{\ell=0}^{2}\Theta_{2,a}^{\ell}(x)\del_{a}^{\,\ell}.
\end{align*}
\begin{proposition}\label{prop:Virasoro1pt}
	The Virasoro $1$-point differential for module labels $\bm{\alpha,\beta}$ is given by 
	\begin{align}\label{eq:VirasoroNabla}
		\F_{\bm{\alpha\beta}}(\omega,x)
		=\delx Z_{\bm{\alpha\beta}}.
	\end{align}
\end{proposition}
\begin{remark}
	\eqref{eq:VirasoroNabla} is a  generalisation of the  formula  
	$Z^{\Zhu}_{V}({\omega},x;q)=q\del_{q} Z^{\Zhu}_{V}(q)$ at genus one~\cite{Z1}
	and a similar genus two result  in~\cite{GT}.  
\end{remark}
The geometric meaning of $\delx$ in terms of variations in the Schottky parameter space $\Cg$ and the dependence of $\delx$ on the choice of Bers quasiform $\Psi_{2}$ is  discussed in \S5 of \cite{TW1}. 
From~\eqref{eq:Psitilde} for a new Bers quasiform 
$\widehat{\Psi}_{2}(x,y)=\Psi_{2}(x,y)-\sum_{r=1}^{3g-3}\Phi_{r}(x)P_{r}(y)$ for quadratic form basis $\{\Phi_{r}(x)\}$ and any $P_{r}(y)=p_{r}(y)dy^{-1}$ with $p_{r}\in \mathfrak{P}_{2}$ we find
\begin{align}\label{eq:nabtilde}
	\widehat{\nabla}(x)=\delx+ \sum_{r=1}^{3g-3} \Phi_{r}(x) \D^{p_{r}},
\end{align}
for $\D^{p}$ of~\eqref{eq:Dpdiff}. 
\eqref{eq:VirasoroNabla} holds for all choices of  $\Psi_{2}$ by Remark~\ref{rem:MainTheorem}~(ii). As a result, $\nabla(x)$ determines a unique tangent vector field, independent of the choice of Bers quasiform, on the moduli tangent space $T(\Mg)=T(\Schg)$, 
 for Schottky space $\Schg=\Cg/\SL_{2}(\C)$.
 We denote this tangent vector field by $\delMg(x)$.   For any local coordinates $\{\eta_{r}\}  $ on  moduli space $\Mg$, with local $T(\Mg)$ basis $\{\partial_{\eta_{r}}\}_{r=1}^{3g-3}$, we can find a corresponding holomorphic 2-form basis $\{\Phi_{r}\} _{r=1}^{3g-3}$ where \cite{O,TW1}
\begin{align}
	\label{eq:nablaMg}
	\delMg(x):=\sum_{r=1}^{3g-3}\Phi_{r}(x)\partial_{\eta_{r}}.
\end{align}

\subsection{Genus $g$ conformal Ward identities}
$\nabla(x)$ maps differentiable functions on $\Cg$ to holomorphic 2-forms in $x$. This generalises to an operator which acts on meromorphic multi-variable forms as follows \cite{TW1}. 
Let $H_{\bfm}(\bfy)=h_{\bfm}(\bfy)dy_{1}^{m_{1}} \ldots dy_{n}^{m_{n}}$, for $\bfm:=m_{1},\cdots ,m_{n}$, denote a meromorphic form of weight $(\bfm)$ in  $\bfy:=y_{1},\cdots ,y_{n}\in \Sg$. 
Define the differential operator 
\begin{align}
	\label{eq:nabla_xym}
	\nabmy{\bfm}{\bfy}(x):=&
	\nabla(x)
	+\sum_{k=1}^{n}\left(\Psi_{2}(x,y_{k})\,d_{y_{k}}+m_{k} d_{y_{k}}\Psi_{2}(x,y_{k})\right).
\end{align} 
\begin{proposition}\cite{TW1}\label{prop:nablaHN}
 $\nabmy{\bfm}{\bfy}(x)\,H_{\bfm}(\bfy)$ is a meromorphic form of weight $(2,\bfm)$ in $(x,\bfy)$ with additional pole structure for  $x\sim y_{k}$ given by 
		\begin{align*}
			\nabmy{\bfm}{\bfy}(x)\,H_{\bfm}(\bfy)\sim 
			dx^{2}\left(
			\frac{m_{k}H_{\bfm}(\bfy)}{(x-y_{k})^{2}}
			+
			\frac{\del_{y_{k}}H_{\bfm}(\bfy)}{x-y_{k}}
			\right).
		\end{align*}
\end{proposition}
 Using general intertwiner Zhu recursion (Theorem~\ref{theor:ZhuGenusgInt}) we obtain genus $g$ conformal Ward identities (e.g.~\cite{EO}) involving $\nabmy{\bfm}{\bfy}(x)$  as follows:
\begin{proposition}\label{prop:nptWard}
	For primary $v_{\gamma_{k}}\in\Wmod_{\gamma_{k}}$ of weight $m_{k}\in\C$ inserted at $y_{k}$ for $k=1,\ldots,n$, we find 
	\begin{align}\label{eq:Ward1}
		\F_{\bm{\alpha\beta}}(\omega,x;\bm{v_{\gamma},y}
		|0,\bm{\gamma};0,\bm{\delta})=
		\nabla_{\bm{y}}^{(\bm{m})}(x)
		\F_{\bm{\alpha\beta}}(\bm{v_{\gamma},y}|\bm{\gamma};\bm{\delta}).
	\end{align}
\end{proposition}
\begin{proof}
	Using Theorem~\ref{theor:ZhuGenusgInt} with $u=\omega$ we obtain
	\begin{align*}
		\F_{\bm{\alpha\beta}}(\omega,x;\bm{v,y}
		|0,\bm{\gamma};0,\bm{\delta})&
		=\sum_{a\in\Ip}\sum_{\ell=0}^{2}\Theta_{a}^{\ell}(x)
		\Res_{a}^{\ell}\F_{\bm{\alpha\beta}}(\omega;\bm{v,y}|0,\bm{\gamma};0,\bm{\delta})\\
		&\quad+\sum_{j\geq 0}\sum_{k=1}^{n}\Psi_{2}^{(0,j)}(x,y_k)
		\F_{\bm{\alpha\beta}}(\ldots;L(j-1)v_{\gamma_{k}},y_k;\ldots|\bm{\gamma};\bm{\delta})\,dy_{k}^{j}.
	\end{align*}
Similarly to Lemma~\ref{lem:omRes} we have $\Res_{a}^{\ell}\F_{\bm{\alpha\beta}}(\omega;\bm{v,y}|0,\bm{\gamma};0,\bm{\delta})=\del_{a}^{\,\ell}\F_{\bm{\alpha\beta}}(\bm{v_{\gamma},y}|\bm{\gamma};\bm{\delta})$. 	Furthermore, since $v_{\gamma_{k}}$ is  primary, we have $L(n)v_{\gamma_{k}}=0$ for all $n>0$. Thus
	\begin{align*}
		\F_{\bm{\alpha\beta}}(\omega,x;\bm{v_{\gamma},y}
		|0,\bm{\gamma};0,\bm{\delta})
		&=\nabla(x)\F_{\bm{\alpha\beta}}(\bm{v,y})+\sum_{k=1}^{n}\Psi_{2}(x,y_k)\F_{\bm{\alpha\beta}}(\ldots;L(-1)v_{\gamma_{k}},y_{k};\ldots
		|\bm{\gamma};\bm{\delta})\\
		&\quad+\sum_{k=1}^{n}\Psi_{2}^{(0,1)}(x,y_k)dy_{k}\,
		\F_{\bm{\alpha\beta}}(\ldots;L(0)v_{\gamma_{k}},y_{k};\ldots
		|\bm{\gamma};\bm{\delta}).
	\end{align*}
	Using $\wt(v_k)=m_k$ and translation covariance, we obtain the result.
\end{proof}
\begin{proposition}[Commutativity]\label{prop:Delcom}
For $v_{\gamma_{k}}\in\Wmod_{\gamma_{k}}$ inserted at $z_{k}$ for $k=1,\ldots,n$ with intermediate module labels $\bm{\delta}=\delta_{1},\ldots,\delta_{n-1}$ we find, up to a multiplier $(x-y)^{4}$, that
	\begin{align*}
\left(\nabla_{y,\bm{z}}^{(2,\bm{m})}(x)\nabla_{\bm{z}}^{(\bm{m})}(y)
		-\nabla_{x,\bm{z}}^{(2,\bm{m})}(y)\nabla_{\bm{z}}^{(\bm{m})}(x)\right)
		\F_{\bm{\alpha\beta}}(\bm{v_{\gamma},z}
		|\bm{\gamma};\bm{\delta})=0.
	\end{align*}
\end{proposition}
\begin{proof}
Apply the intertwiner Zhu recursion Theorem~\ref{theor:ZhuGenusgInt} to the $(n+2)$-point correlation form with a Virasoro vector $\omega$ inserted at both $x$ and $y$  and $v_{\gamma_{k}}$ inserted at $z_{k}$ to find, in way similar to the proof of Proposition~\ref{prop:nptWard}, that
\begin{align*}
&\F_{\bm{\alpha \beta}}(\omega,x;\omega,y;\bm{v_{\gamma},z}\,|0,0,\bm{\gamma};0,0,\bm{\delta})
\\
&=\nabla_{y,\bm{z}}^{(2,\bm{m})}(x)
\nabla_{\bm{z}}^{(\bm{m})}(y)\F_{\bm{\alpha \beta}}(\bm{v_{\gamma},z}\,|\bm{\gamma};\bm{\delta})
+\frac{C}{2}\omega_{2}(x,y)
\F_{\bm{\alpha \beta}}(
\bm{v_{\gamma},z}\,|\bm{\gamma};\bm{\delta}).
\end{align*}
where $\omega_{2}(x,y)=	\Psi_{N}^{(0,3)}(x,y)dy^{3}$ is the symmetric $(2,2)$ form  of~\eqref{eq:omegaN} and using \eqref{eq:Ward1}.
The result follows from the locality of Virasoro vertex operators.
\end{proof} 
\begin{remark}
It is shown in \cite{TW1} that in general,  $\nabla_{y,\bm{z}}^{(2,\bm{m})}(x)\nabla_{\bm{z}}^{(\bm{m})}(y)
=\nabla_{x,\bm{z}}^{(2,\bm{m})}(y)\nabla_{\bm{z}}^{(\bm{m})}(x)$ as differential operators for $x, y,z_{k}\in\mathcal{S}^{(g)}$.
\end{remark}
In a similar fashion to Proposition~\ref{prop:Delcom} we also find
\begin{proposition}\label{prop:omWard}
	The genus $g$ Virasoro $(n+1)$-point differential obeys the identity
	\begin{align}
		\F_{\bm{\alpha \beta}}(\omega,x;\bm{\omega,y})&
		=\nabla_{\bm{y}}^{(\bm{2})}(x)\F_{\bm{\alpha \beta}}(\bm{\omega,y})
		+\frac{C}{2}\sum_{k=1}^{n}\omega_{2}(x,y_{k})\F_{\bm{\alpha \beta}}(\ldots;\widehat{\omega,y_{k}};\ldots),
		\label{eq:Ward2}
	\end{align}
	for $n$-tuple $\bm{(2)}=(2,\ldots,2)$, the symmetric $(2,2)$ form  $\omega_{2}(x,y)$ of~\eqref{eq:omegaN} and where the caret denotes omission of the $\omega$ insertion at $y_{k}$.
\end{proposition}
\begin{remark}\label{rem:Ward}
\leavevmode
\begin{enumerate}
	\item [(i)] Propositions~\ref{prop:nptWard} and~\ref{prop:omWard}  are generalisations of genus two results in~\cite{GT}.
		\item[(ii)] The Ward identities~\eqref{eq:Ward1} and~\eqref{eq:Ward2} are independent of the choice of Bers function described in~\eqref{eq:Psitilde}. In particular, we find that~\eqref{eq:nabtilde} generalises to 
		\begin{align*}
			\widehat{\nabla}^{(\bm{m})}_{\bm{y}}(x)={\nabla}^{(\bm{m})}_{\bm{y}}(x)+ \sum_{r=1}^{3g-3} \Phi_{r}(x) \D^{p_{r}}_{\bm{y}},
		\end{align*}
	for $\D_{\bm{y}}^{p}$  of~\eqref{eq:Dpy}.  Propositions~\ref{prop:nptWard} and~\ref{prop:omWard} hold for $	\widehat{\nabla}^{(\bm{m})}_{\bm{y}}(x)$ by Proposition~\ref{prop:DpyF}.
	Thus the  operator  ${\nabla}^{(\bm{m})}_{\bm{y}}(x)$ appearing in the Ward identity \eqref{eq:Ward1} is describing variations in the moduli space $\M_{g,n}$ of a Riemann surface with $n$ punctures at the insertion points $\bm{y}$ cf. \eqref{eq:nablaMg}.
	\end{enumerate}
\end{remark}

\section{Differential Equations from Zhu Recursion}
\label{sec:Hberg}
\subsection{The generalised Heisenberg VOA}
We consider the rank one Heisenberg VOA $M$ with generator $h$ whose modes obey the bracket
\begin{align*}
	[h(m),h(n)]=m\delta_{m,-n}.
\end{align*}
The Heisenberg VOA has irreducible modules $M_{\alpha}=M \otimes e^{\alpha}$ for $\alpha \in \C$ (with $M_{0}=M$), where for $u \otimes e^{\alpha} \in M_{\alpha}$ of weight $\wt(u)+\half \alpha^{2}$ we have e.g. \cite{K}
\[h(0)(u \otimes e^{\alpha})=\alpha(u \otimes e^{\alpha}),\quad h(n)(u \otimes e^{\alpha}) =  (h(n)u) \otimes e^{\alpha}
\mbox{ for } n\neq 0.\]
For convenience we write $\vac\otimes e^{\alpha}$ as $e^{\alpha}$ throughout.

The Heisenberg intertwiner structure is abelian with $\calY: M_{\alpha}\otimes M_{\beta}\to  M_{\alpha+\beta}$ and the set of all Heisenberg intertwiners form a generalized VOA \cite{TZ2}. In particular, the contragredient label $\alpha'=-\alpha$ and the condition \eqref{eq:Verlinde} automatically holds. Furthermore, for the product of $m$ intertwiners  $\bm{\calY^{\gamma}_{\delta}(v_{\delta},y)}$ of \eqref{eq:calY_alpha_beta}, the module labels satisfy the conditions
\begin{align}
	\label{eq:Hgammadelta}
	\sum_{i=1}^{m}\gamma_{i}=0,
	\quad \delta_k=-\sum_{i=1}^{k}\gamma_{i}, 
\end{align}
for $1\le k\le m-1$, i.e. all of the intermediate module labels $\bm{\delta}$ are determined by $\bm{\gamma}$. 
Similarly, for the operators   $\bm{\calY^{\beta}_{\alpha \alpha'}(b,w)}$ of \eqref{eq:Y_alph_alphp_beta} employed in describing the genus $g$ intertwiner partition function \eqref{eq:Zginter} and correlation function \eqref{eq:Fg_gen_inter}, we find that the intermediate module labels $\beta_{1},\ldots,\beta_{2g-1}$ satisfy the conditions 
\begin{align}
	\label{eq:Hgbeta}
	\beta_{2a-1}=-\alpha_{a} \mbox{ for } a\in\Ip,\quad 
	\beta_{k}=0 \mbox{ for $k$ even}.
\end{align}
%
Let us consider the interwiner $(n+m)$-point form \eqref{eq:Fg_gen_inter} with  $g$ Heisenberg modules $ M_{\alpha_{a}}$ for $n$ Heisenberg states $h$ inserted at $x_i $, $i=1,\ldots, n$ and $m$ module states $e^{\gamma_{j}}\in M_{\gamma_{j}}$ inserted at $z_{j} $ for $j=1,\ldots,m$. The intermediate module labels $\bm{\beta}$ are given in \eqref{eq:Hgbeta} and $\bm{\delta}$ from  \eqref{eq:Hgammadelta}. Thus we define  
\begin{align}	\label{eq:GenHeis}
	&\FgMa(\bm{h,x};\bm{e^{\gamma},z}):=
	\Fg_{\bm{\alpha\beta}}(\bm{h,x};\bm{e^{\gamma},z}| \bm{\gamma;\delta}),
\end{align}
%
where $\sum_{j=1}^{m}\gamma_{j}=0$. $\FgMa(\bm{h,x};\bm{e^{\gamma},z})$ is a generating function for all correlation functions for the rank one generalised Heisenberg VOA~\cite{T}.

Using genus $g$ Zhu recursion for intertwiners Theorem~\ref{theor:ZhuGenusgInt} with respect to the initial state $h$ inserted at $x_{1}$, we find $\Theta_{1,a}^{\ell}(x_{1})=\nu_{a}(x_{1})$  of \eqref{eq:nu}
where $\ell=1$,  $\Psi_{1}(x_{1},z_{j})=\omega_{z_{j}-A_{0}}(x_{1})$ of 
\eqref{eq:Psi1def} and 
$\Psi_{1}^{(0,1)}(x_{1},x_{k})dx_k=\omega(x_{1},x_{k})$ so that we obtain
\begin{align*}
	\FgMa(\bm{h,x};\bm{e^{\gamma},z})=&\nuag(x_{1},\bm{z})\FgMa(h,x_2;\ldots;h,x_n;\bm{e^{\gamma},z})\\
	&+\sum_{k=2}^{n}\omega(x_{1},x_{k})\FgMa(h,x_2;\ldots;\widehat{h,x_k};\ldots;h,x_n;\bm{e^{\gamma},z}),
\end{align*}
where the caret denotes omission of the inserted state  and
\begin{align*}
	\nuag(x,\bm{z}):=\nua(x)+\sum_{j=1}^{m}\gamma_{j}\omega_{z_{j}-z_{0}}(x)
	,\quad \nua(x):=\sum_{a\in\Ip}\alpha_{a}\nu_{a}(x),
\end{align*}
for any $z_{0}$ using $\sum_{j=1}^{m}\gamma_{j}=0$.
Thus the $n$-point function is a linear combination of $(n-1)$- and $(n-2)$-point functions. By induction, we reproduce Corollary 5.6 of~\cite{T} (obtained via combinatorial methods):
\begin{proposition}\label{prop:HbergGenFn}
	\begin{align}\label{eq:HbergGenFn}
		\FgMa(\bm{h,x};\bm{e^{\gamma},z})=
		\Sym_{n}\left(\omega,\nuag\right)
		\FgMa(\bm{e^{\gamma},z}),
	\end{align}
	where
	\begin{align*}
		\Sym_{n}\left(
		\omega,\nuag\right):&=\sum_{\varphi}\prod_{q}\nuag(x_{q},\bm{z})\prod_{(rs)}\omega(x_{r},x_{s}),
	\end{align*}
	summing over all inequivalent involutions $\varphi=(q)\ldots (rs)\ldots$ of the labels $\{1,\ldots,n\}$. 
\end{proposition}
\begin{example}\label{example}
	For $n=1,2,3$ we find
	\begin{align}
		\label{eq:Sym1}
		\Sym_{1}\left(\omega,\nuag\right)
		=&\nuag(x,\bm{z}),
		\\
		\label{eq:Sym2} 
		\Sym_{2}\left(\omega,\nuag\right)
		=&T(x_{1},x_{2}),
		\\
		\label{eq:Sym3} 
		\Sym_{3}\left(\omega,\nuag\right)
		=&T(x_{1},x_{2})\nuag(x_{3},\bm{z})
		+\omega(x_{1},x_{3})\nuag(x_{2},\bm{z})
		\\
		\notag
		&+\omega(x_{2},x_{3})\nuag(x_{1},\bm{z}),
	\end{align}	
	where 
	$
	T(x_{1},x_{2}):=T(x_1,x_2,\bm{z}):=\omega(x_1,x_2)+\nuag(x_1,\bm{z})\nuag(x_2,\bm{z})$. 
\end{example}
Finally, using combinatorial methods, it is shown in \cite{T}, Corollary~5.6 that
\begin{align}\label{eq:HbergTwisted}
	\FgMa\left(\bm{e^{\gamma},z}\right)
	=E(z_{1},z_{2})^{-\gamma^{2}}\exp\left(\im\pi \bm{\alpha}.\Omega.\bm{\alpha}+
	\gamma\int_{z_{2}}^{z_{1}}
	\nua\right)Z_{M},
\end{align}
for prime form $E(z_{1},z_{2})$ of \eqref{eq:prime}, period matrix $\Omega$ of \eqref{eq:period} and where $Z_{M}=\det\left(I-\Atilde\right)^{-\half}$ for $\Atilde$ of \eqref{eq:Rtilde} with $N=1$. 
\subsection{Differential equations for $Z_{M}$ and some forms}\label{DESubsection}	
Let $\nabla_{\bm{y}}^{(\bm{m})}(x)$ be the differential operator of \eqref{eq:nabla_xym} that appears in the Ward identities Propositions~\ref{prop:nptWard} and \ref{prop:omWard}. Considering these Ward identities for Heisenberg correlation functions we obtain
\begin{proposition}\label{prop:delsomnuOm}
	The Heisenberg partition function $Z_{M}$, the prime form	$E(z_1,z_2)$, the period integral $\int_{z_{2}}^{z_{1}}\nu_{a}$,  the period matrix $\Omega_{ab}$, the differential of the third kind $\omega_{z_1-z_2}(y)$, the 1-differential $\nu_{a}(y)$, the bidifferential $\omega(y_{1},y_{2})$ and  the projective connection $s(y)$ satisfy the following differential equations:
	\begin{align}
		&\nabla(x)Z_M=\frac{1}{12}s(x)Z_{M},
		\label{eq:nab_ZM}
		\\
		&\nabla_{z_1,z_2}^{(-\shalf,-\shalf)}(x)E(z_1,z_2)=-\frac{1}{2}\omega_{z_1-z_2}(x)^2 E(z_1,z_2),
		\label{eq:nab_E}
		\\
		&\nabla_{z_1,z_2}^{(0,0)}(x)\int^{z_1}_{z_2}\nu_{a}=\omega_{z_1-z_2}(x)\nu_{a}(x),
		\label{eq:nab_Jac}
		\\
		&\tpi\nabla (x)\Omega_{ab}=\nu_{a}(x)\nu_{b}(x),
		\label{eq:nab_Om}
		\\[2pt]
		&\nabla_{y}^{(1)}(x)\,\nu_{a}(y)=\omega(x,y)\nu_a(x),
		\label{eq:nab_nu}
		\\[2pt]
		&\nabla_{y,z_1,z_2}^{(1,0,0)}(x)\,\omega_{z_1-z_2}(y)
		=\omega(x,y)\omega_{z_1-z_2}(x),
		\label{eq:nab_omthird}
		\\[2pt]
		&\nabla^{(1,1)}_{y_{1},y_{2}}(x)\,\omega(y_{1},y_{2})
		=\omega(x,y_{1})\omega(x,y_{2}),\label{eq:nab_ombidiff}
		\\[2pt]
		&\nabla_{y}^{(2)}(x)\,s(y)=
		6\left(\omega(x,y)^{2}-\omega_{2}(x,y)\right),
		\label{eq:nab_som}
	\end{align}
	for  symmetric $(2,2)$ form $\omega_{2}(x,y)$ of \eqref{eq:omegaN}.
\end{proposition}
\begin{remark}\label{TW1DERemark}  
\eqref{eq:nab_Om} is known as Rauch's formula \cite{R}.
\eqref{eq:nab_E}--\eqref{eq:nab_som} are all derived by alternative analytic means in~\cite{TW1}. Some are generalizations of genus two formulas~\cite{GT}.
Similar results appear in \cite{O} but  $\Psi_{2}$ is not defined nor is Proposition~\ref{prop:nablaHN} established there. 
\end{remark}
\begin{proof}	
	We consider correlation functions with Heisenberg VOA state  insertions 
	together with $e^{\pm\gamma}$ of weight $\kappa=\half\gamma^{2}$ inserted at $\bm{z}=z_{1},z_{2}$.  
	Thus  $\nuag(x,\bm{z})=\nua(x)+\gamma\omega_{z_{1}-z_{2}}(x)$. 
	Using Proposition~\ref{prop:nptWard}, \eqref{eq:HbergTwisted} and  the $\nabla_{\bm{y}}^{(\bm{m})}$ Leibniz rule  \cite{TW1} we find 
	\begin{align}
		\notag
		&\frac{\FgMa(\omega,x;\bm{e^{\gamma},z})}
		{\FgMa(\bm{e^{\gamma},z})}= 
		\frac{\nabla_{\bm{z}}^{(\kappa,\kappa)}(x)\FgMa(\bm{e^{\gamma},z})}
		{\FgMa(\bm{e^{\gamma},z})}
		\\
		\label{eq:FMom1}
		&=-\gamma^{2}\frac{\nabla_{\bm{z}}^{(-\shalf,-\shalf)}(x) E(\bm{z})}{E(\bm{z})}
		+\gamma\,\nabla_{\bm{z}}^{(0,0)}(x)\int_{z_1}^{z_2}\nua
		+\im\pi\nabla(x)\bm{\alpha}.\Omega.\bm{\alpha}
		+\nabla(x)\log Z_{M}.
	\end{align} 
	Since $\omega=\half h(-1)h$ we may also compute  this $3$-point function as a formal limit of $\FgMa(\bm{h,x};\bm{e^{\gamma},z})$ for  $h$ inserted at $\bm{x}=x_{1},x_{2}$ and using \eqref{eq:Sym2} to find
	\begin{align}
		\notag
		&\frac{\FgMa(\omega,x;\bm{e^{\gamma},z})}
		{\FgMa(\bm{e^{\gamma},z})}\notag=
		\half\lim_{x_{i}\rightarrow x}
		\left(
		\frac{\FgMa(\bm{h,x};\bm{e^{\gamma},z})}
		{\FgMa(\bm{e^{\gamma},z})}
		-\Delta(\bm{x})
		\right)
		\\
		&=
		\half\lim_{x_{i}\rightarrow x}
		\left(T(x_{1},x_{2})
		-\Delta(\bm{x})
		\right)
		=
		\frac{1}{12}s(x)+\half\nuag(x,\bm{z})^{2},
		\label{eq:FMom2}
	\end{align}
	where $\Delta(\bm{x}):=(x_{1}-x_{2})^{-2}dx_{1} dx_{2}$ and $s(x)$ is the 
	projective connection \eqref{eq:projcon}.  Comparing 
	\eqref{eq:FMom1} to \eqref{eq:FMom2} yields~\eqref{eq:nab_ZM}--\eqref{eq:nab_Om}. Similarly, we consider the correlation function with $\omega$ inserted at $x$ and $h$ at $y$ which using  \eqref{eq:Sym1} is given by
	\begin{align}
		\notag
		\frac{\FgMa(\omega,x;h,y;\bm{e^{\gamma},z})}
		{\FgMa(\bm{e^{\gamma},z})}
		&=
		\frac{\nabla_{y,\bm{z}}^{(1,\kappa,\kappa)}(x)
			\left(\nuag(y,\bm{z})\FgMa(\bm{e^{\gamma},z})\right)}
		{\FgMa(\bm{e^{\gamma},z})}
		\\[2pt]
		&= \nabla_{y,\bm{z}}^{(1,0,0)}(x)\nuag(y,\bm{z})+
		\frac
		{\nabla_{\bm{z}}^{(\kappa,\kappa)}(x)\FgMa(\bm{e^{\gamma},z})}
		{\FgMa(\bm{e^{\gamma},z})}.
		\label{eq:omegaDE1} 
	\end{align}
	Using \eqref{eq:Sym3}  we may compute  this  as a formal limit of the 5-point function:
	\begin{align*}
		\frac{\FgMa(\omega,x;h,y;\bm{e^{\gamma},z})}
		{\FgMa(\bm{e^{\gamma},z})}
		=&
		\half\lim_{x_{i}\rightarrow x}
		\left(
		\frac{\FgMa(\bm{h,x};h,y;\bm{e^{\gamma},z})}
		{\FgMa(\bm{e^{\gamma},z})}
		-\Delta(\bm{x})
		\frac{\FgMa(h,y;\bm{e^{\gamma},z})}
		{\FgMa(\bm{e^{\gamma},z})}
		\right)
		\\
		=&
		\omega(x,y)\nuag(x,\bm{z})
		+\left(\frac{1}{12}s(x)+\half\nuag(x,\bm{z})^{2}\right)\nuag(y,\bm{z}).
	\end{align*}
	Comparing this to \eqref{eq:omegaDE1} and using \eqref{eq:FMom2} we obtain \eqref{eq:nab_nu} and \eqref{eq:nab_omthird}.
	
	In order to prove \eqref{eq:nab_ombidiff} and \eqref{eq:nab_som} we consider two separate formal limits of the Heisenberg VOA $4$-point function for $h$ inserted at $\bm{x}=x_{1},x_{2}$ and $\bm{y}=y_{1},y_{2}$ 
	\begin{align}
		\label{eq:H4pt}
		\frac{\F_{M}(\bm{h,x;h,y})}{Z_{M}}= \omega(x_{1},x_{2})\omega(y_{1},y_{2})
		+\omega(x_{1},y_{1})\omega(x_{2},y_{2})
		+\omega(x_{1},y_{2})\omega(x_{2},y_{1}),
	\end{align}
	from Proposition~\ref{prop:HbergGenFn}.  Proposition~\ref{prop:nptWard},  \eqref{eq:Sym2} and \eqref{eq:nab_ZM} imply that
	\begin{align*}
		\frac{\F_{M}(\omega,x;\bm{h,y})}{Z_{M}}&
		=\frac{\nabla_{\bm{y}}^{(1,1)}(x)\left(\omega(y_{1},y_{2})Z_{M}\right)}{Z_{M}}
		= \nabla_{\bm{y}}^{(1,1)}(x)\omega(y_{1},y_{2})
		+\frac{1}{12}\omega(y_{1},y_{2})s(x).
	\end{align*}
	This can be alternatively expressed as
	\begin{align*}
		&\half\lim_{x_{i}\rightarrow x}
		\left(
		\frac{\F_{M}(\bm{h,x;h,y})}
		{Z_{M}}
		-\Delta(\bm{x})
		\frac{\F_{M}(\bm{h,y})}
		{Z_{M}}
		\right)
		=
		\omega(x,y_{1})\omega(x,y_{2})
		+\frac{1}{12}\omega(y_{1},y_{2})s(x)
		,
	\end{align*}
	leading to \eqref{eq:nab_ombidiff}. Lastly, Proposition~\ref{prop:omWard} and \eqref{eq:nab_ZM} imply that 
	\begin{align*}
		\frac{\F_{M}(\omega,x;\omega,y)}{ \Zg_{M}}
		&=\frac{1}{12}\nabla^{(2)}_{y}(x) s(y)+\frac{1}{144}s(x)s(y)+\frac{1}{2}\omega_{2}(x,y).
	\end{align*}
	Using \eqref{eq:H4pt}, this can also be expressed as the formal limit:
	\begin{align*}
		&\frac{1}{4}\lim_{\substack{x_i\rightarrow x\\y_i\rightarrow y}}		 
		\left(
		\frac{\Fg_{M}(\bm{h,x;h,y})}{Z_{M}}
		-\Delta(\bm{x})\frac{\Fg_{M}(\bm{h,y})}{Z_M}
		-\Delta(\bm{y})\frac{\Fg_{M}(\bm{h,x})}{Z_M}
		+\Delta(\bm{x}) \Delta(\bm{y})
		\right)
		\\
		&=\frac{1}{2}\omega(x,y)^2 + \frac{1}{144}s(x)s(y),
	\end{align*}
	which implies \eqref{eq:nab_som}.  Thus the proposition has been proved.
\end{proof}
We can derive $
\FgMa(\bm{e^{\gamma},z})$ of \eqref{eq:HbergTwisted}  without recourse to the combinatorial arguments of \cite{T} by using  \eqref{eq:nab_E}--\eqref{eq:nab_Om} as proved by analytic methods  in~\cite{TW1}. We first note that \eqref{eq:FMom1} and \eqref{eq:FMom2} (which were derived by VOA methods) imply that
\begin{align*}
	\nabla_{\bm{z}}^{(\kappa,\kappa)}(x)
	\FgMa(\bm{e^{\gamma},z})
	=&
	\left(\frac{1}{12}s(x)+\half\nuag(x,\bm{z})^{2} \right)
	\FgMa(\bm{e^{\gamma},z}).
\end{align*}
\eqref{eq:nab_ZM} follows by choosing $\bm{\alpha}=\bm{0}$ and $\bm{\gamma}=\bm{0}$. Lastly, assuming \eqref{eq:nab_E}--\eqref{eq:nab_Om} then clearly 
\begin{align*}
	\FgMa(\bm{e^{\gamma},z})
	E(z_{1},z_{2})^{\gamma^{2}}
	\exp\left(-\im\pi \bm{\alpha}.\Omega.\bm{\alpha}-
	\gamma\int_{z_{2}}^{z_{1}}
	\nua\right)Z_{M}^{-1},
\end{align*}
is a constant which is unity by consideration of the genus zero limit $\rho_{a}\rightarrow 0$.
$\FgMa(\bm{e^{\gamma},z})$ can be utilised to compute all $n$-point correlation functions for any generalised VOA which can finitely  decomposed with respect to the modules of a Heisenberg subalgebra $M^{r}$ of  rank $r\ge 1$. Thus consider the lattice generalised VOA $V_{L}$ associated with a rank $r$ Euclidean rational lattice  $L$. 
 Decomposing $V_{L}$ into  $M^{r}$ modules  
$M_{\lambda}=M^{r}\otimes e^{\lambda}$ we find
$V_{L}=\oplus_{\lambda\in L}M_{\lambda}$ where $\lambda=(\lambda_{1},\ldots,\lambda_{r})$. Following Remark~\ref{rem:Zg_rho_fact} and \eqref{eq:HbergTwisted} we find the $M^{r}$ module genus $g$ partition function for $g$ lattice vectors $\bm{\lambda}=(\lambda^{1},\ldots ,\lambda^{g})\in L^{g}$  is
\begin{align*}
	\Zg_{M_{\bm{\lambda}}}=e^{\im\pi\,\bm{\lambda}.\Omega.\bm{\lambda}}\left(\Zg_{M}\right)^{r},
\end{align*}
where $\bm{\lambda}.\Omega.\bm{\lambda}
=\sum_{i=1}^{r}\sum_{a,b\in\Ip}\lambda_{i}^{a}\Omega_{ab}\lambda^{b}_{i}$.
Thus we find
\begin{align*}
	\Zg_{V_{L}}=\Theta_{L}(\Omega)\left(\Zg_{M}\right)^{r},
\end{align*}
for genus $g$ Siegel lattice theta function $	\Theta_{L}(\Omega):=\sum_{\bm{\lambda}\in L^{g}}e^{\im\pi\,\bm{\lambda}.\Omega.\bm{\lambda}}$. 
This result is confirmed for any even self-dual integral lattice VOA in~\cite{C}.


\begin{thebibliography}{AGMV}
\bibitem[A]{A} Ahlfors, L.,
Some remarks on  Teichm\"uller's space of Riemann surfaces,
Ann.~Math.~\textbf{74} (1961) 171--191.

\bibitem[AGMV]{AGMV} Alvarez-Gaumé, L., Moore, G. and Vafa, C.,
Theta functions, modular invariance, and strings,
Comm.~Math.~Phys. \textbf{106}  1--40 (1986).

\bibitem[Be1]{Be1} 
Bers, L.,
Inequalities for finitely generated Kleinian groups, 
J.~Anal.~Math. \textbf{18} 23--41 (1967).

\bibitem[Be2]{Be2} 
Bers, L.,
Automorphic forms for Schottky groups, 
Adv.~Math. \textbf{16}  332--361 (1975).

\bibitem[Bo]{Bo} 
Bobenko, A.,
Introduction to compact Riemann surfaces,
in \textit{Computational Approach to Riemann Surfaces}, 
edited Bobenko, A. and Klein, C.,  Springer-Verlag (Berlin, Heidelberg, 2011).

\bibitem[Bu]{Bu} 
Burnside, W.,
On a class of automorphic functions, 
Proc.~L.~Math.~Soc.  \textbf{23} 49--88 (1891).

\bibitem[C]{C} 
Codogni, G.,
Vertex algebras and Teichmüller modular forms, 
arXiv:1901.03079.

\bibitem[DGM]{DGM}
Dolan, L., Goddard, P.	and Montague, P.,
Conformal field theories, representations and lattice constructions,
Comm.~Math.~Phys. \textbf{179} 61--120 (1996).

\bibitem[DGT1]{DGT1}
Damiolini, C., Gibney, A., and Tarasca, N.,
Conformal blocks from vertex algebras and their connections on $\overline{\mathcal{M}}_{g,n}$. 
Geometry \& Topology \textbf{25} 2235–2286 (2021).


\bibitem[DGT2]{DGT2}
Damiolini, C., Gibney, A. and Tarasca, N.,
On factorization and vector bundles of conformal blocks from vertex algebras,
to appear in Ann.~Sci.~Ec.~Norm.~Super.,
arXiv:1909.04683.

\bibitem[DGT3]{DGT3}
Damiolini, C., Gibney, A. and Tarasca, N., 
Vertex algebras of CohFT-type, in \textit{Facets of Algebraic Geometry - 
A Collection in Honor of William Fulton's 80th Birthday Vol I}, edited by Aluffi, P et al.,  LMS~Lect~Note~Ser., 
\textbf{472} 164–189  CUP (Cambridge, 2022).

\bibitem[DLM]{DLM}
Dong, C., Li, H.~and Mason, G.,
Regularity of rational vertex operator algebras,
Adv.~Math. \textbf{132} 148--166 (1997).


\bibitem[EO]{EO}
Eguchi, T.~and Ooguri, H.,
Conformal and current algebras on a general Riemann surface,
Nucl.~Phys. \textbf{B282} 308--328 (1987).

\bibitem[Fa]{Fa} 
Fay, J.D., 
\textit{Theta Functions on Riemann Surfaces},
Lecture Notes in Mathematics, 
Vol. 352. Springer-Verlag, (Berlin-New York, 1973). 

\bibitem[FBZ] {FBZ} Frenkel, E. and Ben-Zvi, D.,  
\textit{Vertex Algebras and Algebraic Curves}, 
Mathematical Surveys and Monographs \textbf{88}, AMS, (Rhode Island,  2004).

\bibitem[FHL]{FHL}
Frenkel, I.,  Lepowsky, J.  and Huang, Y.-Z.,
\textit{On Axiomatic Approaches to Vertex Operator Algebras and Modules}, 
Mem.~AMS \textbf{104} No. 494, (1993).

\bibitem[FK]{FK}  
Farkas, H.K. and  I. Kra, I.,
\textit{Riemann Surfaces},
Springer-Verlag (New York, 1980).

\bibitem[Fo]{Fo} Ford, L.R.,
\textit{Automorphic Functions},
AMS-Chelsea, (Providence, 2004).

\bibitem[GT]{GT} 
Gilroy, T. and Tuite, M.P., 
Genus two Zhu theory for vertex operator algebras,
arXiv:1511.07664.

\bibitem[G]{G} 
Gui, B.,
Convergence of sewing conformal blocks,
arXiv:2011.07450, to appear Comm.~Contemp.~Math.

\bibitem[H1]{H1}
Huang, Y.-Z.,
Vertex operator algebras and the Verlinde conjecture,
Comm.~Contemp.~Math. \textbf{10}  103--154 (2008).

\bibitem[H2]{H2}
Huang, Y.-Z.,
Differential equations and intertwining operators,
Comm.~Contemp.~Math. \textbf{7}  375--400 (2005).

\bibitem[H3]{H3}
Huang, Y.-Z.,
Generalized rationality and a ``Jacobi identity'' for intertwining operator algebras,
Select.~Math.  \textbf{6} 225–-267 (2000).

\bibitem[K]{K} 
Kac, V.,
\textit{Vertex Algebras for Beginners},
Univ.~Lect.~Ser. \textbf{10}, AMS, (1998).

\bibitem[L]{L} 
Li, H.,
Symmetric invariant bilinear forms on vertex  operator algebras, 
J.~Pure.~Appl.~Alg. \textbf{96} 279--297 (1994).

\bibitem[LL]{LL} 
Lepowsky, J.  and Li, H.,
\textit{Introduction to Vertex Operator Algebras and Their Representations}, 
Progress in Mathematics Vol. 227, Birkh\"auser, (Boston, 2004).

\bibitem[Ma]{Ma} Martinec, E.,
Conformal field theory on a (super-)Riemann surface,
Nucl.~Phys. \textbf{B281} 157--210 (1987).

\bibitem[McIT]{McIT}
McIntyre, A. and Takhtajan, L.A., 
Holomorphic factorization of determinants of Laplacians on Riemann surfaces and a higher genus generalization of Kronecker's first limit formula,
GAFA, Geom.~Funct.~Anal. \textbf{16}  1291--1323 (2006).

\bibitem[MT1]{MT1} 
Mason, G. and Tuite, M.P.,
Vertex operators and modular forms, 
\textit{A Window into Zeta and Modular Physics}, eds. K. Kirsten and F. Williams, 
Cambridge University Press, (Cambridge, 2010),
MSRI Publications \textbf{57} 183--278 (2010).

\bibitem[MT2]{MT2}
Mason, G. and Tuite, M.P.,
Free bosonic vertex operator algebras on genus two Riemann surfaces~I, 
Comm.~Math.~Phys. \textbf{300} 673--713  (2010).

\bibitem[MT3]{MT3}
Mason, G. and Tuite, M.P.
Free bosonic vertex operator algebras on genus two Riemann surfaces~II,
in  \textit{Conformal Field Theory, Automorphic Forms and Related Topics}, 
Contributions in Mathematical and Computational Sciences \textbf{8} 183--225, Springer Verlag, (Berlin,  2014). 

\bibitem[Mu]{Mu} 
Mumford, D.,
\textit{Tata Lectures on Theta I and II},
Birkh\"{a}user, (Boston, 1983).

\bibitem[O]{O} 
Odesskii, A.,
Deformations of complex structures on Riemann surfaces and integrable structures of Whitham type hierarchies,
arXiv:1505.07779.

\bibitem[R]{R}
Rauch, H.E.~
On the  transcendental moduli of algebraic Riemann surfaces,
Proc.~Nat.~Acad.~Sc. \textbf{11} 42--48 (1955).

\bibitem[T]{T}
Tuite, M.P.,
The Heisenberg generalized vertex operator algebra on a Riemann surface, 
Contemp.~Math. \textbf{768} 321–-342 (2021). 


\bibitem[TW1]{TW1}
Tuite, M.P. and Welby, M.,
Some properties and applications of the Bers quasi-form on a Riemann surface, arXiv:2306.08404 (2023).

\bibitem[TW2]{TW2}
Tuite, M.P. and Welby, M.,
General genus Zhu recursion for vertex operator algebras, 
arXiv:1911.06596 (2019).

\bibitem[TZ1]{TZ1}
Tuite, M.P. and  Zuevsky, A., 
The bosonic vertex operator algebra on a genus $g$ Riemann surface,
RIMS Kokyuroko \textbf{1756} 81--93 (2011).

\bibitem[TZ2]{TZ2} Tuite, M.P. and Zuevsky, A.,
A generalized vertex operator algebra for Heisenberg intertwiners,
J.~Pure.~Appl.~Alg. \textbf{216} 1442--1453 (2012).

\bibitem[V]{V}  Vafa, C.,
Conformal theories and punctured surfaces, 
Physics Letters \textbf{B199} 195--202 (1987).

\bibitem[WW]{WW}
Whittaker, E.T. and Watson, G.N., 
\textit{A Course of Modern Analysis},
Cambridge University Press, 

\bibitem[Z1]{Z1}
Zhu, Y.,
Modular-invariance of characters of vertex operator algebras,
J.~Amer.~Math.~Soc. \textbf{9} 237--302  (1996).

\bibitem[Z2]{Z2}
Zhu, Y., 
Global vertex operators on Riemann surfaces.
Comm.~Math.~Phys. \textbf{165} 485--531 (1994).
\end{thebibliography}
\end{document}